\documentclass[11pt,letterpaper]{amsart}
\usepackage[utf8]{inputenc}
\usepackage{amsmath}
\usepackage{amsfonts}
\usepackage{amssymb}
\usepackage{graphicx}
\usepackage[left=2.4cm,right=2.4cm,top=2.5cm,bottom=2.5cm]{geometry}
\numberwithin{equation}{section}
\graphicspath{ {images/} }
\usepackage{amsmath,amssymb, amsthm} 
\usepackage{lipsum}
\usepackage{stmaryrd}
\usepackage{dsfont}
\usepackage{color}
\usepackage[all]{xy}

\makeatletter
\def\@settitle{\begin{center}%
  \baselineskip14\p@\relax
  \bfseries
  \uppercasenonmath\@title
  \@title
  \ifx\@subtitle\@empty\else
     \\[1ex]\uppercasenonmath\@subtitle
     \footnotesize\mdseries\@subtitle
  \fi
  \end{center}%
}
\def\subtitle#1{\gdef\@subtitle{#1}}
\def\@subtitle{}
\makeatother

\usepackage{tikz-cd}
\usepackage{hyperref}
\usepackage{bm}
\hypersetup{colorlinks, linkcolor=blue, citecolor=black, urlcolor=black}

\usepackage[english]{babel}
\usepackage[colorinlistoftodos]{todonotes}

\theoremstyle{plain}
\newtheorem{thm}{Theorem}[subsection] 

\usepackage{mathrsfs}

\theoremstyle{definition}
\newtheorem{defi}[thm]{Definition}
\newtheorem{conj}[thm]{Conjecture}

\newtheorem{rem}[thm]{Remark}
\theoremstyle{definition}

\theoremstyle{plain}
\newtheorem{prop}[thm]{Proposition}
\theoremstyle{plain}
\newtheorem{lemma}[thm]{Lemma}
\theoremstyle{plain}
\newtheorem{cor}[thm]{Corollary}
\theoremstyle{plain}
\newtheorem{thmintro}{Theorem}


\newcounter{parentnumber}

\newcommand{\cO}{\mathcal{O}}
\newcommand{\scO}{\mathscr{O}}

\newcommand{\BK}{{\rm BK}}

\newcommand{\st}{{\tau}}

\newcommand{\nr}{\mathbf{N}^{-r}}

\newcommand{\bQ}{\mathbb{Q}}

\makeatletter  
\newcommand{\colim@}[2]{%
  \vtop{\m@th\ialign{##\cr
    \hfil$#1\operator@font colim$\hfil\cr
    \noalign{\nointerlineskip\kern1.5\ex@}#2\cr
    \noalign{\nointerlineskip\kern-\ex@}\cr}}%
}
\newcommand{\colim}{%
  \mathop{\mathpalette\colim@{\rightarrowfill@\scriptscriptstyle}}\nmlimits@
}
\renewcommand{\varprojlim}{%
  \mathop{\mathpalette\varlim@{\leftarrowfill@\scriptscriptstyle}}\nmlimits@
}
\renewcommand{\varinjlim}{%
  \mathop{\mathpalette\varlim@{\rightarrowfill@\scriptscriptstyle}}\nmlimits@
}
\makeatother

\newcommand{\Z}{\mathbb{Z}}
\newcommand{\Q}{\mathbb{Q}}

\newcommand{\R}{\mathbb{R}}
\newcommand{\C}{\mathbb{C}}
\newcommand{\bC}{\mathbb{C}}

\newcommand{\PP}{\mathfrak{P}}
\newcommand{\cc}{\mathfrak{c}}

\font\wncyr=wncyr9.8
\newcommand{\sha}{\text{\wncyr{W}}}

\newcommand{\rH}{\mathrm{H}}
\newcommand{\Lg}{\mathcal{L}}

\begin{document}

\title[TNC for CM modular forms and Rankin--Selberg convolutions]
{Tamagawa number conjecture for CM modular forms and Rankin--Selberg convolutions} 

\author{Francesc Castella}
\address{University of California Santa Barbara, South Hall, Santa Barbara, CA 93106, USA}
\email{castella@ucsb.edu}

\date{\today}

\dedicatory{In memory of Jan Nekov{\'a}{\v{r}}}
\thanks{This research was partially supported by the NSF grants DMS-2101458 and DMS-2401321.}

\maketitle
\vspace{-5.mm}
\begin{center}
	\footnotesize{(with an appendix by FRANCESC CASTELLA AND MYCHELLE PARKER)}
\end{center}

\begin{abstract}
Let $E/F$ be an elliptic curve defined over a number field $F$ with complex multiplication by the ring of integers of  an imaginary quadratic field $K$ such that the torsion points of $E$ generate over $F$ an abelian extension of $K$. In this paper we prove the $p$-part of the Birch--Swinnerton-Dyer formula for $E/F$ in analytic rank $1$ for primes $p>3$  split in $K$. This was previously known for  $F=\Q$ by work of Rubin \cite{rubin-IMC} as a consequence of his proof of Mazur's Main Conjecture for rational CM elliptic curves,  
but the problem for $[F:\Q]>1$ remained wide open. 

The approach introduced in this paper also yields a proof of similar results for CM abelian varieties $A/K$ and for CM modular forms, as well as an analogue in this setting of Skinner's $p$-converse to the theorem of Gross--Zagier and Kolyvagin.  
%
\end{abstract}

\tableofcontents

\addtocontents{toc}{\protect\setcounter{tocdepth}{1}}

\section{Introduction}

\subsection{Statement of the main results}

In this paper we prove the following result towards the Birch and Swinnerton-Dyer conjecture for elliptic curves with complex multiplication.

\begin{thmintro}\label{thmintro:p-BSD-E}
Let $E$ be an elliptic curve  over a number field $F$ with  complex multiplication by the ring of integers of an imaginary quadratic field $K$, ${\rm End}_F(E)\simeq\cO_K$, such that $F(E_{\rm tors})/K$ is abelian. Let $\psi_E:F^\times\backslash\mathbb{A}_F^\times\rightarrow\bC^\times$ be the Hecke character associated to $E/F$. Suppose 
\[
{\rm ord}_{s=1}L(\psi_E,s)=1,
\] 
and let $p\nmid 6h_K$ be a prime split in $K$, where $h_K:=\#{\rm Pic}(\cO_K)$ is the class number of $K$. Assume the conductor $\mathfrak{f}_\psi$ of $\psi_E$ is prime to $p$ and satisfies $\mathfrak{d}_K\Vert N_{F/K}(\mathfrak{f}_\psi)$, where $\mathfrak{d}_K:=(\sqrt{-D_K})$. Then 
\[
{\rm rank}_{\Z}E(F)={\rm ord}_{s=1}L(E/F,s).
\] 
Moreover, 
for all primes $\wp\mid p$ in $K$ we have $\#\sha(E/F)[\wp^\infty]<\infty$, with
\[
{\rm ord}_{\wp}\biggl(\frac{L^*(E/F,1)}{{\rm Reg}(E)\cdot\Omega(E)}\biggr)={\rm ord}_\wp(\#\sha(E/F)[\wp^\infty]\cdot{\rm Tam}(E/F)),
\]
where $L^*(E/F,1)$ is the leading Taylor coefficient of the Hasse--Weil $L$-function for $E/F$ at $s=1$. 
In other words, the $\wp$-part of the Birch--Swinnerton-Dyer formula holds for $E/F$.
\end{thmintro}

\begin{rem}
\hfill
\begin{enumerate}
\item
Upon the choice of an identification ${\rm End}_F(E)\simeq\cO_K$,  the action of $\cO_K$ on the Lie algebra of $E/F$ induces an embedding $K\hookrightarrow F$, and classical results of Deuring \cite{deuring} (see also \cite[Thm.~10.5]{silverman-adv}) show that
\[
L(E/F,s)=L(\psi_E,s)\cdot L(\overline{\psi}_E,s),
\]
where $\overline{\psi}_E$ is the complex conjugate of $\psi_E$.  
Hence the first assertion in  Theorem~\ref{thmintro:p-BSD-E} is that ${\rm rank}_{\Z}E(F)=2$. 
\item
The relevance of the condition that the extension $F(E_{\rm tors})/K$ be  abelian for the application of Iwasawa theory of $K$ to the arithmetic of elliptic curves with CM by $K$ was first highlighted in work of Arthaud \cite{arthaud} generalizing Coates--Wiles \cite{coates-wiles} (see also \cite{rubin-CW,goldstein-shappacher}).
\end{enumerate}
\end{rem}

For $F=\bQ$ (which forces $K$ to have class number one), Theorem~\ref{thmintro:p-BSD-E} was obtained by Rubin \cite{rubin-IMC} as a consequence of his proof of the Iwasawa Main Conjecture for $K$, the Gross--Zagier formula \cite{grosszagier}, and Perrin-Riou's work \cite{PR-HP,PR-GZ}. 
The restriction to $F=\Q$ seems essential to Rubin's result, as it relies on a link 
with Mazur's Main Conjecture \cite{mazur-towers,M-SwD} for rational elliptic curves. 

Our proof of Theorem~\ref{thmintro:p-BSD-E} also gives a new proof of Rubin's result in the case $F=\Q$\footnote{Indeed, this follows from Theorem~\ref{thmintro:p-BSD-E} by the same arguments from Milne's work \cite{milne-invmath-abvar} as in \cite[Cor.~2]{burungale-flach}.}, 
circumventing the use of $p$-adic heights and Bertrand's transcendence results \cite{bertrand-LNM} (which appear to be  unknown in higher dimensions). The method also yields a similar result on the Birch--Swinnerton-Dyer conjecture for higher-dimensional CM abelian varieties $A/K$. 

\begin{thmintro}\label{thmintro:p-BSD-A}
Let $A/K$ be an abelian variety with ${\rm End}_K(A)\simeq\cO_{L}$ for a CM field $L$ with $[L:K]={\rm dim}(A)$, and let $\lambda:K^\times\backslash\mathbb{A}_K^\times\rightarrow\bC^\times$ be the associated Hecke character of conductor $\cc$. 
Suppose 
\[
{\rm ord}_{s=1}L(\lambda,s)=1,
\] 
and let $p\nmid 6h_K$ be a prime split in $K$. Assume $\cc$ is prime to $p$ and satisfies $\mathfrak{d}_K\Vert\cc$. Then 
\[
{\rm rank}_{\Z}A(K)={\rm ord}_{s=1}L(A/K,s). 
\] 
Moreover, for all primes $\PP\mid p$ in $L$ we have $\#\sha(A/K)[\PP^\infty]<\infty$, with
\[
{\rm ord}_{\PP}\biggl(\frac{L^*(A/K,1)}{{\rm Reg}(A)\cdot\Omega(A)}\biggr)={\rm ord}_{\PP}(\#\sha(A/K)[\PP^\infty]\cdot{\rm Tam}(A/K)),
\]
and hence the $\mathfrak{P}$-part of the Birch--Swinnerton-Dyer formula holds for $A/K$.
\end{thmintro}

The periods $\Omega(A)$ and $\Omega(E)$ in the above results are as defined in \cite[Def.~30]{flach-siebel}, which as explained in \cite[\S{1}]{burungale-flach} agree with the periods in the conjecture of Birch Swinnerton-Dyer for abelian varieties over number fields.

\subsection{About the proofs}

The starting point in the proof of our main results is an idea introduced by  Bertolini--Darmon--Prasanna \cite{BDP-PJM} in their proof of (a generalization of) Rubin's formula \cite{rubin-points} expressing the $p$-adic logarithm of Heegner points in terms of special values of Katz $p$-adic $L$-functions. In Rubin's original proof of the formula, the arithmetic of  CM elliptic curves $E/\bQ$ is studied using Heegner points for an auxiliary imaginary quadratic field $K'$ satisfying the Heegner hypothesis relative to the conductor of $E$. In particular, one is forced to take $K'\neq K$, as the $L$-function 
\[
L(E/K,s)=L(E,s)^2
\] 
has always sign $+1$. Letting $\lambda$ be the Hecke character of $K$ attached to $E/\Q$, the ingenious idea of \cite{BDP-PJM} is to write 
\[
\lambda=\psi\chi
\] 
as the product 
of a suitable Hecke character $\psi$ of infinity type $(-1,0)$ and a ray class character $\chi$, so that $L(E,s)=L(\lambda,s)$ appears as a factor of the Rankin--Selberg convolution
\[
L(g/K,\chi,s)=L(\lambda,s)\cdot L(\psi^\st\chi,s)
\]
for $g=\theta_\psi$, where $\psi^\st$ denotes the composition of $\psi$ with the action of the non-trivial automorphism of $K/\Q$, and exploit Heegner cycles for the pair $(g,\chi)$,  as featured in the general Gross--Zagier formula \cite{YZZ}, to study the arithmetic of $\lambda$.  

Towards the application of this idea to our problem, the first part of the paper is devoted to the study of the anticyclotomic Iwasawa Main Conjecture for ``self-dual pairs'' $(g,\chi)$ of the form 
\[
(g,\chi)=(\theta_\psi,\chi)
\]
with $\psi$ a Hecke a character of $K$ of infinity type $(1-2r,0)$ for any $r\geq 1$ and $\chi$ a finite order Hecke character with central character $\varepsilon_\chi=\varepsilon_g^{-1}$, where $\varepsilon_g$ is the nebentypus of $g$.  
Expanding on the work of Agboola--Howard \cite{AH-ord}  and Arnold \cite{arnold}, 
incorporating key ideas and results from Kato \cite{kato-295} and Johnson-Leung--Kings \cite{kings-JL}, we prove the Main Conjecture  for $(g,\chi)$ in this setting, both 
\begin{itemize}
\item in terms of the $p$-adic $L$-function $\mathscr{L}_w(g,\chi)$ of Bertolini--Darmon--Prasanna (Theorem~\ref{thm:BDP-IMC}),  
\item in terms of the $\Lambda_\scO$-adic Heegner classes $\mathbf{z}_{g,\chi}$ of \cite{cas-hsieh1} (Corollary~\ref{cor:HP-IMC}).
\end{itemize}

After inverting $p$ and for $r=1$, the ``lower bound'' divisibility predicted by these Main Conjectures was obtained in earlier work of the author with Burungale, Skinner, and Tian in \cite{BCST}; 
here we further develop the method to prove the equality predicted by the Main Conjecture without any ambiguity by powers of $p$ (as is essential for results such as Theorem~\ref{thmintro:p-BSD-E}). 

The next step is to deduce from the equality of characteristic ideals in the  Iwasawa Main Conjectures, a formula for the order of 
the Bloch--Kato Tate--Shafarevich group attached 
of the pair $(g,\chi)$. In our rank $1$ case, the formula we obtain is in terms of the index of a Heegner class $z_{g,\chi}$ (introduced Theorem~\ref{thm:BDP-formula}) inside the Bloch--Kato Selmer group ${\rm Sel}_\BK(K,T_{g,\chi})$. 

With future arithmetic applications in mind (see Remark~\ref{rem:Parker}), this result also applies in arbitrary weights $2r\geq 2$. For the statement, let $\mathscr{L}_v(\lambda\mathbf{N}^{-r})$ denote the anticyclotomic Katz $p$-adic $L$-function introduced in $\S\ref{subsubsec:katz-ac}$, which has the trivial character $\mathds{1}$ outside the range of interpolation. 

\begin{thmintro}\label{thmintro:sha}
Let $\lambda$ be a Hecke character of $K$ of infinity type $(1-2r,0)$ for some $r\geq 1$, central character $\varepsilon_\lambda=\eta_K$, where $\eta_K$ is the quadratic character associated to $K/\Q$, and conductor $\cc$. Suppose $p\nmid 6h_K$ splits in $K$, $\mathfrak{c}$ is coprime to $p$ and satisfies $\mathfrak{d}_K\Vert\mathfrak{c}$, and  $\mathscr{L}_v(\lambda\mathbf{N}^{-r})(\mathds{1})\neq 0$.  
Let $(\psi,\chi)$ be a good pair for $\lambda$ in the sense of Definition~\ref{def:good}, and let 
\[
z_{g,\chi}\in{\rm Sel}_\BK(K,T_{g,\chi})
\] 
be the Heegner class associated to the self-dual pair $(g,\chi)=(\theta_\psi,\chi)$. Then ${\rm rank}_\scO\,{\rm Sel}_\BK(K,T_{g,\chi})=1$, $z_{g,\chi}$ is non-torsion, and
\[
\#\bigl({\rm Sel}_\BK(K,T_{g,\chi})/\scO\cdot z_{g,\chi}\bigr)^2=\#\sha_\BK(W_{g,\chi}/K)
\cdot\prod_{\substack{w\in\Sigma,v\nmid p}}c_w(W_{g,\chi}/K),
\]
where $c_w(W_{g,\chi}/K)$ 
is the Tamagawa number of $W_{g,\chi}$ at $w$. 
\end{thmintro}

\begin{rem}
When $r=1$, the assumption $\mathscr{L}_v(\lambda\mathbf{N}^{-r})(\mathds{1})\neq 0$ 
follows from ${\rm ord}_{s=r}L(\lambda,s)=1$ (see Remark~\ref{rem:wt2-ran1}); the same implication is expected to hold for $r>1$, but this is not known at present. 
\end{rem}

Let $\Gamma={\rm Gal}(K_\infty/K)$ denote the Galois group of the anticyclotomic $\Z_p$-extension of $K$. The proof of Theorem~\ref{thmintro:sha} is based on a new calculation of the $\Gamma$-Euler characteristic (in the sense of \cite[\S{4}]{greenberg-cetraro}) of a certain Selmer group $\mathcal{X}_v(g,\chi)$. 
%
The starting point is the decomposition
\begin{equation}\label{eq:intro-dec-Sel}
\mathcal{X}_v(g,\chi)\simeq\mathcal{X}_v(\lambda\nr)\oplus\mathcal{X}_v(\psi^\st\chi\nr)
\end{equation}
of Proposition~\ref{prop:dec-Sel}, where $\mathcal{X}_v(\psi^\st\chi\nr)$ interpolates the Bloch--Kato Selmer group of $\psi^\st\chi\nr$ over $K_\infty/K$, while $\mathcal{X}_v(\lambda\nr)$ is seen to agree with the Selmer group obtained by \emph{reversing} the local conditions at the primes above $p$ defining the Bloch--Kato Selmer group of $\lambda\nr$ over the anticyclotomic tower. 
After computing the $\Gamma$-Euler characteristic for each of the summands in \eqref{eq:intro-dec-Sel}, 
the expressions we obtain combine quite pleasantly into the formula of Theorem~\ref{thmintro:sha}.

The proof of Theorem~\ref{thmintro:p-BSD-E} is easily deduced from an application of Theorem~\ref{thmintro:p-BSD-A} to (the isogeny factors of)  $B={\rm Res}_{F/K}(E)$, so it remains to explain how to go from Theorem~\ref{thmintro:sha} to Theorem~\ref{thmintro:p-BSD-A}, for which we rely on the general Gross--Zagier formula \cite{YZZ} (in its explicit form in \cite{CST}). Letting $B_{g,\chi}/K$ be the CM abelian variety (well-defined up to isogeny) attached to the pair $(g,\chi)$ arising from Serre's tensor construction, the deduction relies on the factorization
\[
L(B_{g,\chi}/K,s)=\prod_{\sigma:L\hookrightarrow\C}L(g^\sigma/K,\chi^\sigma,s)
\]
and the explicit description of the motivic structure attached to Hecke characters in the work of Kato \cite{kato-295} and Burungale--Flach \cite{burungale-flach}. The existence of \emph{strong good pairs} in the sense introduced in Definition~\ref{def:good} (whose existence is established in Lemma~\ref{lem:good} building  on Hida's mod $p$ non-vanishing results \cite{hida-durham,hida-ICTS}) is also used at this point 
to address a certain period comparison building on work of  Hida--Tilouine \cite{HT-ENS,HT-invmath}, as completed by Hida \cite{hida-coates}, on the anticyclotomic Iwasawa Main Conjecture for CM fields. 

%

\begin{rem}\label{rem:Parker}
Granted a Gross--Zagier type formula for the generalized Heegner cycles of \cite{BDP} in the style of S.-W.\,Zhang \cite{zhang-GZ}\footnote{See recent work of Lilienfeldt--Shnidman \cite{GZ-GHC} for progress in this direction.}, Theorem~\ref{thmintro:sha} should yield a proof of the $p$-part of the equivariant Tamagawa number conjecture of Burns--Flach \cite{burns-flach-I} for CM motives in analytic rank $1$ under standard hypotheses\footnote{That is, injectivity of the $p$-adic Abel--Jacobi map and non-degeneracy of a (Gillet--Soul\'{e}) height pairing.}. (See \cite{LV-TNC} for results in this direction in the non-CM case building on \cite{zhang-GZ}.) 
\end{rem}

\begin{rem}\label{rem:kol-conv}
In the course of proving Theorem~\ref{thmintro:sha}, we show the implication 
\begin{equation}\label{eq:intro-kol}
z_{g,\chi}\neq 0\quad\Longrightarrow\quad{\rm rank}_{\scO}\,{\rm Sel}_{\BK}(K,T_{g,\chi})=1.
\end{equation}
This extends 
\cite[Thm.\,B]{cas-hsieh1} to the case of newforms with CM by $K$ (note that this case is excluded by the Heegner hypothesis in \emph{op.\,cit.}). 
For the proof, we build on the Euler system of elliptic units. 

On the other hand, as a consequence of our results on the Main Conjecture for $(g,\chi)$, we also deduce a proof of a converse implication to \eqref{eq:intro-kol}, which together with the expected Gross--Zagier formula for 
$z_{g,\chi}$ yields an analogue of Skinner's $p$-converse to the theorem of Gross--Zagier and Kolyvagin \cite{skinner-converse} for higher weight CM modular forms (see Theorem~\ref{thm:p-conv}).
\end{rem}

\subsection{About the hypotheses}

The restriction on the conductor in our results arises from the \emph{Heegner hypothesis} present in \cite{cas-hsieh1,JLZ}, and could be removed by a suitable refinement of the results in \cite{LZZ}. Our second additional hypothesis that $p\nmid h_K$ would also become superfluous with a better understanding of the relation between modular degrees and congruence numbers studied in \cite{ARS}. 
Removing the hypotheses that $p>3$ and $p$ be split in $K$ would seem to require new ideas. We hope to come back to this in future work.


\subsection{Acknowledgements} It is a pleasure to thank Chris Skinner for stimulating exchanges related to this work. 
We would also like to thank Ashay Burungale, Henri Darmon, Giada Grossi, Ye Tian, and Xin Wan for interesting conversations related to the topics in this paper. We also thank Matthias Flach, Dick Gross, and Haruzo Hida for comments on an earlier draft. Finally, we thank the referee for a very careful reading of the paper, and a number of useful comments and suggestions. 

We would like to dedicate this work to the memory of the late  Jan Nekov{\'a}{\v{r}}, whose encouragement was heartfelt and instrumental to us, specially in the early career stages.


\addtocontents{toc}{\protect\setcounter{tocdepth}{2}}

\section{$p$-adic $L$-functions}

In this section we describe the $p$-adic $L$-functions needed for our arguments. The key result is the factorization of Proposition~\ref{prop:factor-BDP}. The discussion in this section parallels \cite[\S{4}]{BCST}, 
but here we need to pay more attention to integrality properties.

\subsection{$p$-adic $L$-functions for self-dual pairs}\label{subsec:Lp-pairs}

\subsubsection{Self-dual pairs}

Let $g\in S_{2r}(\Gamma_0(N_g),\varepsilon_g)$ be a newform of even weight $k=2r\geq 2$, and let $K/\Q$ be an imaginary quadratic field of discriminant $-D_K<0$ satisfying the 
\emph{Heegner hypothesis}: 
\begin{equation}\label{eq:Heeg}
\textrm{there exists an ideal $\mathfrak{N}_g$ with $\cO_K/\mathfrak{N}_g\simeq\Z/N_g\Z$.}\tag{Heeg}
\end{equation}
We also fix once and for all embeddings  $\imath_\infty:\overline{\Q}\hookrightarrow\C$ and $\imath_p:\overline{\Q}\hookrightarrow\overline{\Q}_p$, and assume that
\begin{equation}\label{eq:spl}
\textrm{$(p)=v\overline{v}$ splits in $K$,}\tag{spl}
\end{equation}
with $v$ the prime of $K$ above $p$ induced by $\imath_p$.

As in \cite[\S{2.1}]{BCST}, we say that a Hecke character $\psi=(\psi_w)_w:K^\times\backslash\mathbb{A}_K^\times\rightarrow\C^\times$ has \emph{infinity type} $(a,b)\in\mathbb{Z}^2$ if 
$\psi_\infty(z)=z^a\overline{z}^b$ for all $z\in(K\otimes_\Q\R)^\times\simeq\C^\times$ under the identification induced by $\imath_\infty$. In particular, the norm character $\mathbf{N}$ given by $\mathfrak{a}\mapsto\#(\cO_K/\mathfrak{a})$ on ideals of $\cO_K$, has infinity type $(-1,-1)$.  Then the \emph{central character} of $\psi$ is the Dirichlet character $\varepsilon_\psi$ defined by
\[
\psi\vert_{\mathbb{A}^\times}=\varepsilon_\psi\cdot\vert\;\vert_{\mathbb{A}^\times}^{a+b}.
\]
We say that $\psi$ is \emph{anticyclotomic} if $\psi\vert_{\mathbb{A}^\times}=1$; in particular, such $\psi$ has trivial central character and its infinity type is of the form $(n,-n)$.

\begin{defi}\label{def:sd-pair}
Let $\chi$ be a finite order character of $K$. We say that $(g,\chi)$ is a \emph{self-dual pair} if 
\[
\varepsilon_\chi=\varepsilon_g^{-1}.
\]
Then the Rankin--Selberg $L$-function $L(g/K,\chi,s)$ is self-dual, with a functional equation relating its values at $s$ and $2r-s$. The sign $\epsilon(g,\chi)\in\{\pm{1}\}$ in the functional equation is a product of local signs:
\[
\epsilon(g,\chi)=\prod_{q}\epsilon_q(g,\chi),
\]
where $q$ runs over all places of $\Q$. 
\end{defi}

Let $\hat K^\times$ denote the group of finite id\`{e}les of $K$. Attached to a character $\psi$ of $K$ of infinity type $(a,b)$ is its \emph{$p$-adic avatar} $\hat\psi:K^\times\backslash\hat K^\times\rightarrow\C_p^\times$ defined by
\[
\hat\psi(x)=\imath_p\circ\imath_\infty^{-1}(\psi(x))\sigma(x_p)^a\overline{\sigma}(x_p)^b
\]
for all $x\in\hat K^\times$, where $\sigma:K\otimes_{\Q}\Q_p\rightarrow\C_p$ is the map induced by $\imath_p$, and $\overline{\sigma}:=\sigma\circ\st$ for the non-trivial automorphism $\st$ of $K/\Q$. Via the (geometrically normalized) reciprocity map ${\rm rec}_K:K^\times\backslash\hat{K}^\times\rightarrow G_K^{\rm ab}$, we shall view $\hat\psi$ as a Galois character.   Since it should not lead to confusion, in the following we shall still use $\psi$ to denote its $p$-adic avatar $\hat\psi$ (except for $\psi=\mathbf{N}$, whose $p$-adic avatar is $\varepsilon_{\rm cyc}^{-1}$ for the $p$-adic cyclotomic character $\varepsilon_{\rm cyc}:G_K\rightarrow\Z_p^\times$).

For $\scO$ the ring of integers of a finite extension $\Phi$ of $\Q_p$, we say that the pair $(g,\chi)$ is \emph{defined over $\scO$} if (under our fixed embeddings $\imath_\infty,\imath_p$) the Fourier coefficients of $g$ and the values of $\chi$ are all in $\scO$. 

\begin{defi}
For a positive integer $c$ prime to $N_g$, we let $\Sigma_{\rm cc}(c,\mathfrak{N}_g,\varepsilon_g)$ denote the set of finite order characters $\chi$ such that $(g,\chi)$ is a self-dual pair and moreover:
\begin{itemize}
\item[(i)] The conductor of $\chi$ is
\[
\mathfrak{f}_\chi=(c)\mathfrak{N}_{\varepsilon_g},
\]
where $\mathfrak{N}_{\varepsilon_g}$ is the unique divisor of $\mathfrak{N}_g$ with norm equal to the conductor of $\varepsilon_g$;
\item[(ii)]  $\epsilon_q(g,\chi)=+1$ for all finite primes $q$.
\end{itemize}
\end{defi}

In particular, for $\chi\in\Sigma_{\rm cc}(c,\mathfrak{N}_g,\varepsilon_g)$ and $\xi$ any anticyclotomic Hecke character of $K$ of infinity type $(n,-n)$ and conductor divisible only by primes that split in $K$, the $L$-series $L(g/K,\chi\xi,s)$ is self-dual with center at $s=r$ and sign
\[
\epsilon(g,\chi\xi)=\begin{cases}
+1&\textrm{if $\vert n\vert\geq r$,}\\[0.2em]
-1&\textrm{if $\vert n\vert <r$.}
\end{cases}
\]
The $p$-adic $L$-functions in this section interpolate the central $L$-values $L(g/K,\chi\xi,r)$ for $\vert n\vert\geq r$.

\subsubsection{CM periods}\label{subsub:CM-periods}

Fix an elliptic curve $A_0/F$ defined over a number field $F$ with complex multiplication by $\cO_K$, and let $(\Omega_K,\Omega_p)\in\bC^\times\times\cO_{\bC_p}^\times$ be the complex and $p$-adic CM periods in \cite[\S{2.5}]{cas-hsieh1} (with $A_0$ corresponding to the CM elliptic curve denoted $A$ in \emph{loc.\,cit.}). Put also 
\[
\Omega:=2\pi i\cdot\Omega_K,
\] 
and note that this recovers (up to $F^\times$) the complex CM period appearing in \cite[II.4.2]{deshalit}. For any embedding $\sigma:F\rightarrow\bC$, we define $\Omega_\sigma\in\C^\times$ by replacing $A_0$ by $A_0^\sigma$ in the above definition.

\subsubsection{$p$-adic interpolation}

For any abelian extension $K'/K$, let $\Lambda_\scO(K')$ denote the Iwasawa algebra  $\scO[\![{\rm Gal}(K'/K)]\!]:=\varprojlim_{}\scO[{\rm Gal}(K''/K)]$, where $K''$ runs over the finite extensions of $K$ contained in $K'$, and the inverse limit is with respect to the natural projection maps. We also put
\[
\Lambda_\scO^{\rm ur}(K'):=\Lambda_\scO(K')\hat\otimes_{\Z_p}\Z_p^{\rm ur},
\]
where $\Z_p^{\rm ur}$ denotes the completion of the ring of integers of the maximal unramified extension of $\Q_p$. 

Let $H_{p^\infty}$ be the union of the ring class fields of $K$ of $p$-power conductor; the maximal torsion-free quotient of $\widetilde{\Gamma}:={\rm Gal}(H_{p^\infty}/K)$ is by definition the Galois group 
\[
\Gamma:={\rm Gal}(K_\infty/K)
\] 
of the anticyclotomic $\Z_p$-extension of $K$. Since anticyclotomic Iwasawa algebras will play a prominent role in the paper, we set
\[
\Lambda_\scO:=\Lambda_\scO(K_\infty),\quad\quad\Lambda_\scO^{\rm ur}:=\Lambda_\scO^{\rm ur}(K_\infty)
\]
for the ease of notation.

\begin{thm}\label{thm:BDP}
Let $(g,\chi)\in S_{2r}(\Gamma_1(N_g))\times\Sigma_{\rm cc}(c,\mathfrak{N}_g,\varepsilon_g)$ be a self-dual pair for some positive 
integer $c$ prime to $pN_g$ defined over $\scO$. There exists a ``square-root'' $p$-adic $L$-function 
\[
\mathscr{L}_v(g,\chi)\in\Lambda_\scO^{\rm ur}(H_{p^\infty})
\] 
characterized by the property that for every character $\xi$ of $\widetilde{\Gamma}$ crystalline at both $v$ and $\overline{v}$ corresponding to a Hecke character of $K$ of infinity type $(n,-n)$ with $n\geq r$ and $n\equiv 0\pmod{p-1}$, we have
\begin{align*}
\mathscr{L}_v(g,\chi)^2(\xi)&=\frac{\Omega_p^{4n}}{\Omega^{4n}}\cdot\Gamma(r+n)\Gamma(n+1-r)\cdot\biggl(\frac{2\pi}{\sqrt{D_K}}\biggr)^{2n-1}\\
&\quad\times\bigl(1-a_p(g)\chi\xi(\overline{v})p^{-r}+\varepsilon_g(p)\chi\xi(\overline{v})^2p^{-1}\bigr)^2\cdot L(g/K,\chi\xi,r).
\end{align*} 
\end{thm}

\begin{proof}
This follows from the results of \cite[\S{3}]{cas-hsieh1} as in \cite[Thm.~4.1]{BCST}. (Note that in the above interpolation property we have replaced the CM period $\Omega_K$ from \cite{cas-hsieh1} by the above $\Omega$, and have omitted terms interpolated by a unit in $\Lambda_{\scO}^{\rm ur}(H_{p^\infty})^\times$.) 
\end{proof}

\begin{rem}
Reversing the roles of $v$ and $\overline{v}$ in Theorem~\ref{thm:BDP}, we obtain an element $\mathscr{L}_{\overline{v}}(g,\chi)\in\Lambda_{\scO}^{\rm ur}(H_{p^\infty})$ interpolating the central critical values $L(g/K,\chi\xi,r)$ for $\xi$ of infinity type $(-n,n)$ with $n\geq r$ and $n\equiv 0\pmod{p-1}$. We shall use both $\mathscr{L}_v(g,\chi)$ and $\mathscr{L}_{\overline{v}}(g,\chi)$ in the following.
\end{rem}

With a slight abuse of notation, we also denote by $\mathscr{L}_v(g,\chi)$ its image under the natural projection $\Lambda_{\scO}^{\rm ur}(H_{p^\infty})\rightarrow\Lambda_\scO^{\rm ur}$, and similarly for $\mathscr{L}_{\overline{v}}(g,\chi)$.

\subsection{Katz $p$-adic $L$-functions}

The construction of $p$-adic $L$-functions in the next result is originally due to Katz \cite{katz}, but the alternative construction by de Shalit \cite{deshalit}, whose exposition we follow, will be most convenient for our purposes.

\subsubsection{Two-variable $p$-adic $L$-functions}

For an integral ideal $\cc$ of $K$, let $K(\cc p^\infty)$ 
be the union of the ray class fields of $K$ conductor $\cc p^n$ for $n\geq 0$, and put
\[
G_\cc:={\rm Gal}(K(\cc p^\infty)/K).
\]

For a character $\xi$ of $G_\cc$, let $\xi^\st$ denote the composition of $\xi$ with action of the non-trivial automorphism of $K/\Q$, and put $\xi^{-\tau}:=(\xi^\st)^{-1}$.

\begin{thm}\label{thm:katz}
Let $\cc$ be an integral ideal of $K$ prime to $p$. There exists an element 
\[
\mathscr{L}_{v,\cc}\in\Lambda_{\scO}^{\rm ur}(K(\cc p^\infty))
\]
characterized by the property that for every character $\xi$ of $G_\cc$ crystalline at both $v$ and $\overline{v}$ corresponding to a Hecke character of infinity type $(a,b)$ with $a>0$ and $b\leq 0$ we have 
\[
\mathscr{L}_{v,\cc}(\xi)=\frac{\Omega_p^{a-b}}{\Omega^{a-b}}\cdot\Gamma(a)\cdot\biggl(\frac{\sqrt{D_K}}{2\pi}\biggr)^b\cdot(1-\xi^{-1}(v)p^{-1})(1-\xi(\overline{v}))
\cdot L_\cc(\xi,0),
\]
where $L_\cc(\xi,s)$ denotes the Hecke $L$-function of $\xi$ with the Euler factors at the primes $w\mid\cc$ removed. Similarly, there exists an element $\mathscr{L}_{\overline{v},\cc}\in\Lambda_{\scO}^{\rm ur}(K(\cc p^\infty))$ such that for every character $\xi$ of $G_\cc$ crystalline at both $v$ and $\overline{v}$ corresponding to a Hecke character of infinity type $(b,a)$ with $a>0$ and $b\leq 0$ we have 
\[
\mathscr{L}_{\overline{v},\cc}(\xi)=\frac{\Omega_p^{a-b}}{\Omega^{a-b}}\cdot\Gamma(a)\cdot\biggl(\frac{\sqrt{D_K}}{2\pi}\biggr)^b\cdot(1-\xi^{-1}(\overline{v})p^{-1})(1-\xi(v))
\cdot L_\cc(\xi,0).
\]
Moreover, we have the functional equation 
\[
\mathscr{L}_{v,\cc}(\xi)=\mathscr{L}_{v,\overline{\cc}}(\xi^{-\tau}\mathbf{N}^{-1}),
\] 
where the equality is up to a $p$-adic unit, and similarly for $\mathscr{L}_{\overline{v},\cc}$.
\end{thm}

\begin{proof}
The first assertion follows from \cite[Thm.~II.4.14]{deshalit}: 
$\mathscr{L}_{v,\cc}$ (resp. $\mathscr{L}_{\overline{v},\cc}$) corresponds to the $p$-adic measure $\mu(\cc\overline{v}^\infty)$ (resp. $\mu(\cc v^\infty)$) on $G_\cc$ constructed in \emph{loc.\,cit.}. The functional equation is proved in \cite[Thm.~II.6.4]{deshalit}, which implies that the interpolation formulae  extend from the range $0\leq -b<a$ to the range $a>0$ and $b\leq 0$ (see \cite[Cor.~II.6.7]{deshalit}).
\end{proof}

\subsubsection{Anticyclotomic projection}\label{subsubsec:katz-ac}

We shall be particularly interested in the anticyclotomic projection of Katz's $p$-adic $L$-functions twisted by characters that are self-dual in the following sense.

\begin{defi}\label{def:sd-char}
We say that a Hecke character $\phi$ of $K$ is \emph{self-dual} if it satisfies
\[
\phi^\tau=\phi^{-1}\mathbf{N}^{-1}\quad\textrm{and}\quad\varepsilon_\phi=\eta_K,
\]
where $\eta_K$ is the quadratic Dirichlet character corresponding to $K/\Q$.
\end{defi}

In particular, if $\phi$ is self-dual then its $p$-adic avatar is a conjugate self-dual character of $G_K$, i.e. 
\[
\phi^\tau=\varepsilon_{\rm cyc}\phi^{-1}.
\] 
Note that for self-dual $\phi$, the Hecke $L$-function $L(\phi,s)$ is self-dual, with a functional equation relating its values at $s$ and $-s$. Moreover, the infinity type of such $\phi$ is of the form $(1-r,r)$ for some $r\in\Z$ and its conductor is invariant under complex conjugation.

Let $F/K$ be an abelian extension contained in $K(\cc p^\infty)$. For a character $\phi$ of ${\rm Gal}(F/K)$ valued in $\scO$ and having $\cc$ as the prime-to-$p$ part of its conductor, we denote by $\mathscr{L}_v(\phi)$ the image of $\mathscr{L}_{v,\cc}$ under the composite map
\[
\Lambda_\scO^{\rm ur}(K(\cc p^\infty))\xrightarrow{{\rm Tw}_\phi}\Lambda_\scO^{\rm ur}(K(\cc p^\infty))\twoheadrightarrow\Lambda_{\scO}^{\rm ur},
\]
where ${\rm Tw}_\phi$ is the $\scO$-linear isomorphism given by $\gamma\mapsto\phi(\gamma)\gamma$ for $\gamma\in G_\cc$ 
and the second arrow is the natural projection (noting that $K(\cc p^\infty)$ contains $H_{p^\infty}$, and hence also $K_\infty$).

\begin{rem}
Although not reflected in the notation for simplicity, when we write $\mathscr{L}_v(\phi)$ (resp. $\mathscr{L}_{\overline{v}}(\phi)$) it will be tacitly understood that this is obtained from $\mathscr{L}_{v,\cc}$ (resp. $\mathscr{L}_{\overline{v},\cc}$) with $\cc$ the prime-to-$p$ part of the conductor of $\phi$.
\end{rem}

\subsection{Factorization} 

Recall that if $\psi$ is a Hecke character of infinity type $(1-2r,0)$ and conductor $\mathfrak{f}_\psi$, then the theta series $g=\theta_\psi$ is an eigenform in $S_{2r}(\Gamma_0(N_g),\varepsilon_g)$ with $N_g=D_K\mathbf{N}(\mathfrak{f}_\psi)$ and $\varepsilon_g=\varepsilon_\psi\eta_K$.

\begin{prop}\label{prop:factor-BDP}
Let $\psi$ be a Hecke character of $K$ of infinity type $(1-2r,0)$ for some $r\geq 1$ and conductor a cyclic ideal $\mathfrak{f}_\psi$ of norm prime to $pD_K$. Let 
\[
(g,\chi)=(\theta_\psi,\chi)\in S_{2r}(\Gamma_0(N_g),\varepsilon_g)\times\Sigma_{\rm cc}(c,\mathfrak{N}_g,\varepsilon_g)
\] 
be a self-dual pair for some positive integer $c$ prime to $pN_g$ defined over $\scO$.  
Then for every $w\in\{v,\overline{v}\}$ we have
\[
\mathscr{L}_w(g,\chi)^2=u\cdot\mathscr{L}_{w}(\psi\chi\mathbf{N}^{-r})\cdot\mathscr{L}_{w}(\psi^\st\chi\mathbf{N}^{-r}),
\]
where $u$ is a unit in $\Lambda_\scO^{{\rm ur},\times}$.
\end{prop}

\begin{proof}
We only explain the proof in the case $w=v$; the argument for $w=\overline{v}$ is of course the same.  By our assumption on $\mathfrak{f}_\psi$, the imaginary quadratic field $K$ satisfies  \eqref{eq:Heeg} relative to $N_g$ with $\mathfrak{N}_g=\mathfrak{d}_K\cdot\mathfrak{f}_\psi$. Let $\xi$ be a  character as in the statement of Theorem~\ref{thm:BDP}, corresponding to a Hecke character of $K$ of infinity type $(n,-n)$ with $n\geq r$, and assume  that $\xi$ factors through $\Gamma$. Then the central $L$-value $L(\theta_\psi/K,\chi\xi,r)$ factors as
\[
L(\theta_\psi/K,\chi\xi,r)=L(\psi\chi\xi,r)\cdot L(\psi^\tau\chi\xi,r)=L(\psi\chi\mathbf{N}^{-r}\xi,0)\cdot L(\psi^\tau\chi\mathbf{N}^{-r}\xi,0).
\]
The Hecke characters $\psi\chi\mathbf{N}^{-r}\xi$ and $\psi^\tau\chi\mathbf{N}^{-r}\xi$ have infinity type $(1-r+n,r-n)$ and $(n+r,1-n-r)$, and one easily checks (see \cite[Lem.~3.16]{BDP-PJM}) that they are both self-dual, with prime-to-$p$ conductor $(c)\mathfrak{d}_K$ and $(cM)\mathfrak{d}_K$, respectively, where $\mathfrak{d}_K:=(\sqrt{-D_K})$ and $M:=\mathbf{N}(\mathfrak{f}_\psi)$; thus 
they are in the range of interpolation of $\mathscr{L}_{v,(c)\mathfrak{d}_K}$ and $\mathscr{L}_{v,(cM)\mathfrak{d}_K}$, respectively. From the self-duality of $\psi\chi\mathbf{N}^{-r}\xi$ and $\psi^\tau\chi\mathbf{N}^{-r}\xi$ we immediately get the relations
\[
(1-\psi^{-1}\chi^{-1}\xi^{-1}(v)p^{r-1})=(1-\psi\chi\xi(\overline{v})p^{-r}),\quad
(1-\psi^{-\st}\chi^{-1}\xi^{-1}(\overline{v})p^{r-1})=(1-\psi^\st\chi\xi(\overline{v})p^{-r}),
\]
and so from the interpolation formula in Theorem~\ref{thm:katz} we obtain
\begin{equation}\label{eq:prod-katz}
\begin{aligned}
\mathscr{L}_{v,(c)\mathfrak{d}_K}(\psi\chi\mathbf{N}^{-r}\xi)&\cdot\mathscr{L}_{v,(c)\mathfrak{d}_KM}(\psi^\tau\chi\mathbf{N}^{-r}\xi)
=\frac{\Omega_p^{2n-2r+1}}{\Omega^{2n-2r+1}}\cdot\frac{\Omega_p^{2n+2r-1}}{\Omega^{2n+2r-1}}\cdot\Gamma(n-r+1)\cdot\Gamma(n+r)\\
&\quad\times\biggl(\frac{2\pi}{\sqrt{D_K}}\biggr)^{n-r}\cdot\biggl(\frac{2\pi}{\sqrt{D_K}}\biggr)^{n+r-1}\cdot\bigl(1-\psi\chi\xi(\overline{v})p^{-r}\bigr)^2\cdot\bigl(1-\psi^\st\chi\xi(\overline{v})p^{-r}\bigr)^2\\
&\quad\times L(\psi\chi\mathbf{N}^{-r}\xi,0)\cdot L(\psi^\tau\chi\mathbf{N}^{-r}\xi,0),
\end{aligned}
\end{equation}
using that the prime-to-$p$ conductors of $\psi\chi\mathbf{N}^{-r}\xi$ and $\psi^\tau\chi\mathbf{N}^{-r}\xi$ are $(c)\mathfrak{d}_K$ and $(cM)\mathfrak{d}_K$, respectively, to equivalently write the primitive Hecke $L$-values. Noting that
\[
(1-\psi\chi\xi(\bar{v})p^{-r}\bigr)\cdot\bigl(1-\psi^\st\chi\xi(\bar{v})p^{-r}\bigr)=(1-a_p(g)\chi\xi(\overline{v})p^{-r}+\varepsilon_g(p)\chi\xi(\overline{v})^2p^{-1}),
\]
comparing \eqref{eq:prod-katz} with the interpolation property of Theorem~\ref{thm:BDP}, 
the result follows.
\end{proof}

\section{Selmer groups and Main Conjectures}\label{sec:Sel-char}


Throughout this section, we fix a prime $p>3$ and integer $N_g\geq 1$ with $p\nmid N_g$, and let 
\[
(g,\chi)\in S_{2r}(\Gamma_1(N_g))\times\Sigma_{\rm cc}(c,\mathfrak{N}_g,\varepsilon_g)
\]
be a self-dual pair defined over $\scO$ for some positive integer $c$ prime to $pN_g$ and an imaginary quadratic field $K$ satisfying \eqref{eq:Heeg} and \eqref{eq:spl}.
%

\subsection{Selmer groups for self-dual pairs}\label{subsec:Sel-sd}

Let $V_g$ denote the $p$-adic Galois representation associated to $g$ by Deligne, and put 
\[
V_{g,\chi}:=V_g(r)\vert_{G_K}\otimes\chi.
\]

\begin{defi}\label{def:Sel-sd}
Let $w\in\{v,\overline{v}\}$ be a prime of $K$ above $p$, and let $F/K$ be a finite extension. For $\eta$ a prime of $F$ above $w$, define the local condition $\rH^1_w(F_{\eta},V_{g,\chi})$ by
\[
\rH^1_w(F_{\eta},V_{g,\chi}):=\begin{cases}
\rH^1(F_{\eta},V_{g,\chi})&\textrm{if $\eta\mid w$,}\\[0.2em]
0&\textrm{if $\eta\mid\overline{w}$,}
\end{cases}
\] 
and let $\rH^1_f(F_{\eta},V_{g,\chi})$ denote the Bloch--Kato finite subspace, i.e.
\[
\rH^1_f(F_{\eta},V_{g,\chi}):=\ker\bigl\{\rH^1(F_{\eta},V_{g,\chi})\rightarrow\rH^1(F_{\eta},V_{g,\chi}\otimes\mathbf{B}_{\rm cris})\bigr\},
\]
where $\mathbf{B}_{\rm cris}$ is Fontaine's ring of crystalline periods. For a finite prime $\eta\nmid p$ of $F$, let
\[
\rH^1_{\rm ur}(F_{\eta},V_{g,\chi}):={\rm ker}\bigl\{\rH^1(F_\eta,V_{g,\chi})\rightarrow\rH^1(I_\eta,V_{g,\chi})\bigr\}
\] 
be the unramified subspace. 
\begin{itemize}
\item The \emph{$w$-Selmer group} of $V_{g,\chi}$ is 
\[
{\rm Sel}_w(F,V_{g,\chi}):=\ker\biggl\{\rH^1(F,V_{g,\chi})\rightarrow\prod_{\eta\mid p}\frac{\rH^1(F_\eta,V_{g,\chi})}{\rH_w^1(F_\eta,V_{g,\chi})}\times\prod_{\eta\nmid p}\frac{\rH^1(F_\eta,V_{g,\chi})}{\rH^1_{\rm ur}(F_\eta,V_{g,\chi})}\biggr\}.
\]
\item The \emph{Bloch--Kato Selmer group} of $V_{g,\chi}$ is
\[
{\rm Sel}_\BK(F,V_{g,\chi}):=\ker\biggl\{\rH^1(F,V_{g,\chi})\rightarrow\prod_{\eta\vert p}\frac{\rH^1(F_\eta,V_{g,\chi})}{\rH_f^1(F_\eta,V_{g,\chi})}\times\prod_{w\nmid p}\frac{\rH^1(F_\eta,V_{g,\chi})}{\rH^1_{\rm ur}(F_\eta,V_{g,\chi})}\biggr\}.
\]
(Note that since $p$ is odd, for $\eta\mid\infty$ the groups $\rH^1(F_\eta,V_{g,\chi})$ all vanish.)
\end{itemize}
\end{defi}

Let $T_{g,\chi}\subset V_{g,\chi}$ be a $G_K$-stable $\scO$-lattice, and define $W_{g,\chi}$ by the exact sequence
\begin{equation}\label{eq:TVA}
0\rightarrow T_{g,\chi}\rightarrow V_{g,\chi}\rightarrow W_{g,\chi}\rightarrow 0.
\end{equation}
Then we define the Selmer groups ${\rm Sel}_w(F,T_{g,\chi}), {\rm Sel}_w(F,W_{g,\chi}), {\rm Sel}_\BK(F,T_{g,\chi}), {\rm Sel}_\BK(F,W_{g,\chi})$ as above, with the corresponding local conditions obtained from those of $V_{g,\chi}$ by propagating via \eqref{eq:TVA}.

Finally, for $\star\in\{v,\overline{v},\BK\}$, we define the $\Lambda$-adic Selmer groups
\[
\check{\mathcal{S}}_\star(g,\chi):=\varprojlim_{F}{\rm Sel}_\star(F,T_{g,\chi}),\quad\quad\mathcal{S}_\star(g,\chi):=\varinjlim_{F}{\rm Sel}_\star(F,W_{g,\chi}),
\]
where $F$ runs over the finite extensions of $K$ contained in $K_\infty$, and the limits are with respect to the corestriction and restriction maps, respectively, and put
\[
\mathcal{X}_\star(g,\chi):={\rm Hom}_{\rm cts}(\mathcal{S}_\star(g,\chi),\Q_p/\Z_p)
\]
for the Pontryagin dual of $\mathcal{S}_\star(g,\chi)$.

\subsection{Generalized Heegner cycles}\label{subsec:GHC}

We recall 
the construction of $\Lambda_\scO$-adic classes $\mathbf{z}_{g,\chi}$ 
interpolating 
the generalized Heegner cycles of Bertolini--Darmon--Prasanna \cite{BDP} over the anticyclotomic tower $K_\infty/K$, as well as their link with the $p$-adic $L$-functions $\mathscr{L}_w(g,\chi)$.


We begin by recalling an important formula from \cite{BDP} for the  value of $\mathscr{L}_w(g,\chi)$ at the trivial character $\xi=\mathds{1}$ of $\Gamma$ (outside the range of $p$-adic interpolation in Theorem~\ref{thm:BDP}). For $w\in\{v,\overline{v}\}$ a prime of $K$ above $p$, denote by
\[
{\rm log}_{V_{g,\chi}}:\rH^1_f(K_{\overline{w}},V_{g,\chi})\rightarrow\frac{\mathbf{D}_{\rm dR}(V_{g,\chi})}{{\rm Fil}^0\mathbf{D}_{\rm dR}(V_{g,\chi})}\simeq
{\rm Fil}^0\mathbf{D}_{\rm dR}(V_{g,\chi}^*(1))^\vee
\]
the Bloch--Kato logarithm map,
and recall the CM elliptic curve $A_0/F$ used for the definition of the CM periods in $\S\ref{subsub:CM-periods}$.

\begin{thm}[Bertolini--Darmon--Prasanna]\label{thm:BDP-formula}
There exists a class $z_{g,\chi}\in{\rm Sel}_\BK(K,V_{g,\chi})$ 
such that
\begin{align*}
\mathscr{L}_w(g,\chi)(\mathds{1})&=\frac{c^{-r}}{\Gamma(r)}\cdot\bigl(1-a_p(g)\chi(\overline{w})p^{-r}+\varepsilon_g(p)\chi(\overline{w})^2p^{-1}\bigr)\cdot\left\langle{\rm log}_{V_{g,\chi}}({\rm loc}_{\overline{w}}(z_{g,\chi})),\omega_g'\wedge\omega_{A_0}^{r-1}\eta_{A_0}^{r-1}\right\rangle,
\end{align*}
where $\omega_g'\in\mathbf{D}_{\rm dR}(V_g^*)$ is a differential attached to $g$ as in \cite{KLZ-AJM}, and $\omega_{A_0}$ and $\eta_{A_0}$ are differentials attached to $A_0$ as in \cite[\S{1.4}]{BDP}. 
\end{thm}

\begin{proof}
This is \cite[Thm.~5.13]{BDP}, as reformulated in \cite[Thm.~7.2.4]{JLZ}.
\end{proof}

By interpolating $p$-power conductor variants of the classes $z_{g,\chi}$ into a $\Lambda_\scO$-adic class $\mathbf{z}_{g,\chi}\in\check{\mathcal{S}}_\BK(g,\chi)$, a generalization of Theorem~\ref{thm:BDP-formula} allows one to recover $\mathscr{L}_v(g,\chi)$ as the image of $\mathbf{z}_{g,\chi}$ under a generalized Coleman power series map. This was first done in a joint work of the author with M.-L.\,Hsieh \cite{cas-hsieh1} for $\chi=\mathds{1}$, and later in \cite{JLZ} in the level of generality required for this paper. 


To state the result, given a class $\mathbf{z}\in\check{\mathcal{S}}_\BK(g,\chi)$, we let $\mathbf{z}(\mathds{1})$ denote the image of $\mathbf{z}$ under the natural projection
\[
\check{\mathcal{S}}_\BK(g,\chi)\rightarrow{\rm Sel}_\BK(K,V_{g,\chi}).
\]
Let $\alpha$ be the $p$-adic unit root of $x^2-a_p(g)x+\varepsilon_g(p)p^{2r-1}$, and let $\varpi\in\scO$ be a uniformizer.

\begin{thm}\label{thm:ERL}
Let $w\in\{v,\overline{v}\}$ be a prime of $K$ above $p$. There exists a class $\mathbf{z}_{g,\chi}\in\check{\mathcal{S}}_\BK(g,\chi)$ with 
\[
\mathbf{z}_{g,\chi}(\mathds{1})
=\mathcal{E}_p(g,\chi)\cdot z_{g,\chi},
\]
where $\mathcal{E}_p(g,\chi)=\bigl(1-\frac{\chi(v)p^{r-1}}{\alpha}\bigr)\bigl(1-\frac{\chi(\overline{v})p^{r-1}}{\alpha}\bigr)$, and an injective $\Lambda_\scO^{\rm ur}$-module homomorphism
\[
{\rm Col}_{\overline{w}}:\varprojlim_{F\subset K_\infty}\prod_{\eta\vert \overline{w}}\rH^1_f(F_\eta,T_{g,\chi})\hat\otimes_{\scO}\Lambda_\scO^{\rm ur}\rightarrow\Lambda_\scO^{\rm ur}
\]
with finite cokernel 
for which we have the ``explicit reciprocity law''
\begin{equation}\label{eq:ERL-GHC}
{\rm Col}_{\overline{w}}\bigl({\rm loc}_{\overline{w}}(\mathbf{z}_{g,\chi})\bigr)=\mathscr{L}_w(g,\chi).\nonumber
\end{equation}
\end{thm}

\begin{proof}
The construction of $\mathbf{z}_{g,\chi}$ is given in \cite[\S{5.2}]{cas-hsieh1} in the case $\varepsilon_g=1$ and in \cite[Thm.~5.4.1]{JLZ}\footnote{Note that this theorem achieves more, also interpolating $z_{g,\chi}$ along a ``weight'' variable for $g$, vastly generalizing the results of \cite{howard-invmath,cas-higher, cas-variation,ota-JNT} in the $p$-ordinary case; cf. \cite{disegni-universal,BL-coleman}. Here we only need the interpolation result in the anticyclotomic direction.} in general. Similarly, the construction of ${\rm Col}_{\overline{w}}$ and the proof of the explicit reciprocity law is given in \cite[\S{5.3}]{cas-hsieh1} in the case $\varepsilon_g=1$, and in \cite[\S{8}]{JLZ} in general. 
\end{proof}

\subsection{Main Conjectures for self-dual pairs}\label{subsec:IMC-sd}

We now state the anticyclotomic Main Conjectures that will be studied in this paper.


\begin{conj}[Iwasawa--Greenberg Main Conjecture]\label{conj:BDP-IMC} Let $w\in\{v,\overline{v}\}$ be a prime of $K$ above $p$, and assume that $\mathscr{L}_w(g,\chi)\neq 0$. Then:
\begin{itemize}
\item[(i)] ${\rm rank}_{\Lambda_\scO}\big(\check{\mathcal{S}}_w(g,\chi)\bigr)={\rm rank}_{\Lambda_\scO}\bigl(\mathcal{X}_w(g,\chi)\bigr)=0$;
\item[(ii)] We have
\[
{\rm char}_{\Lambda_\scO}\bigl(\mathcal{X}_w(g,\chi)\bigr)=\bigl(\mathscr{L}_w(g,\chi)^2\bigr)
\]
as ideals in $\Lambda_{\scO}^{\rm ur}$.
\end{itemize}
\end{conj}

\begin{rem}\label{rem:isog-inv}
It follows from Perrin-Riou's results \cite{PR-isogeny} that the characteristic ideal of $\mathcal{X}_w(g,\chi)$ is independent of the choice of $G_K$-stable $\scO$-lattice $T_{g,\chi}$ in $V_{g,\chi}$ used in the definition of $\mathcal{X}_w(g,\chi)$ (see \cite[Prop.~2.9]{KO}). Note that this is consistent with Conjecture~\ref{conj:BDP-IMC}(ii), since the periods used in the construction of $\mathscr{L}_w(g,\chi)$ depend only on $K$.
\end{rem}

The following is a natural higher weight extension of the Heegner point Main Conjecture of \cite{BT-conv} and \cite{BCST} (see also \cite{LV-kyoto} for an analogous higher weight extension of the original Heegner point main conjecture by Perrin-Riou \cite{PR-HP}).

\begin{conj}[Heegner cycle Main Conjecture]\label{conj:HPMC}
The following hold:
\begin{itemize}
\item[(i)] $\mathbf{z}_{g,\chi}$ is not $\Lambda_\scO$-torsion; 
\item[(ii)] ${\rm rank}_{\Lambda_\scO}\bigl(\check{\mathcal{S}}_{\BK}(g,\chi)\bigr)={\rm rank}_{\Lambda_\scO}\bigl(\mathcal{X}_{\BK}(g,\chi)\bigr)=1$;
\item[(iii)] We have
\[
{\rm char}_{\Lambda_\scO}\bigl(\mathcal{X}_{\BK}(g,\chi)_{\rm tors}\bigr)={\rm char}_{\Lambda_\scO}\bigl(\check{\mathcal{S}}_{\BK}(g,\chi)/\Lambda_\scO\cdot\mathbf{z}_{g,\chi}\bigr)^2,
\]
where the subscript ${\rm tors}$ denotes the maximal $\Lambda_\scO$-torsion submodule.
\end{itemize}
\end{conj}

\begin{rem}\label{rem:iota}
It follows from Nekov\'{a}\v{r}'s results \cite{nekovar-310} that there is a $\Lambda_\scO$-module pseudo-isomorphism
\[
\mathcal{X}_\BK(g,\chi)_{\rm tors}\,\sim\,M\oplus M
\]
for a torsion $\Lambda_\scO$-module $M$ with ${\rm char}_{\Lambda_\scO}(M)^\iota={\rm char}_{\Lambda_\scO}(M)$, where $\iota$ denotes the involution on $\Lambda_\scO$ given by $\gamma\mapsto\gamma^{-1}$ for $\gamma\in\Gamma$ (see also \cite[p.~1464]{howard}). Thus we see that the equality of characteristic ideals in Conjecture~\ref{conj:HPMC} can alternatively be written as 
\[
{\rm char}_{\Lambda_\scO}\bigl(\mathcal{X}_{\BK}(g,\chi)_{\rm tors}\bigr)={\rm char}_{\Lambda_\scO}\bigl(\check{\mathcal{S}}_{\BK}(g,\chi)/\Lambda_\scO\cdot\mathbf{z}_{g,\chi}\bigr)\cdot{\rm char}_{\Lambda_\scO}\bigl(\check{\mathcal{S}}_{\BK}(g,\chi)/\Lambda_\scO\cdot\mathbf{z}_{g,\chi}\bigr)^\iota.
\]
\end{rem}


\subsection{Selmer groups for Hecke characters}\label{subsec:Sel-char}

In this subsection, the imaginary quadratic field $K$ is only required to satisfy \eqref{eq:spl} for our fixed prime $p>3$.

Let $\xi$ be a Hecke character of $K$ of infinity type $(a,b)\in\Z^2$ with $p$-adic avatar valued in the ring of integers $\scO$ of a finite extension $\Phi/\Q_p$. Put
\[
V_\xi=\Phi(\xi),
\]
where $\Phi(\xi)$ denotes the one-dimensional $\Phi$-vector space on which $G_K$ act via $\xi$. 

\begin{defi}\label{def:Sel-char}
Let $w\in\{v,\overline{v}\}$ be a prime of $K$ above $p$, and let $F/K$ be a finite extension. For $\eta$ a prime of $F$ above $p$, put 
\[
\rH^1_f(F_\eta,V_{\xi}):=\ker\big\{\rH^1(F_\eta,V_\xi)\rightarrow\rH^1(F_\eta,V_\xi\otimes\mathbf{B}_{\rm cris})\bigl\}
\] 
and 
\[
\rH^1_w(F_\eta,V_\xi):=\begin{cases}\rH^1(F_\eta,V_\xi)&\textrm{if $\eta\mid w$,}\\[0.2em]
0&\textrm{if $\eta\mid\overline{w}$};\end{cases}
\]
and for $\eta$ a finite prime of $F$ not dividing $p$, let $\rH^1_{\rm ur}(F_\eta,V_\xi)={\rm ker}\{\rH^1(F_\eta,V_{\xi})\rightarrow\rH^1(I_\eta,V_\xi)\}$ be the unramified subspace.
\begin{itemize}
\item The \emph{Bloch--Kato} Selmer group of $V_\xi$ is 
\[
{\rm Sel}_\BK(F,V_{\xi}):=\ker\biggl\{\rH^1(F,V_{\xi})\rightarrow\prod_{\eta\vert p}\frac{\rH^1(F_\eta,V_{\xi})}{\rH_f^1(F_\eta,V_{\xi})}\times\prod_{\eta\nmid p}\frac{\rH^1(F_\eta,V_{\xi})}{\rH^1_{\rm ur}(F_\eta,V_{\xi})}\biggr\}.
\]
\item The \emph{$w$-Selmer group} of $V_\xi$ is 
\[
{\rm Sel}_w(F,V_{\xi}):=\ker\biggl\{\rH^1(F,V_{\xi})\rightarrow\prod_{\eta\vert p}\frac{\rH^1(F_\eta,V_{\xi})}{\rH_w^1(F_\eta,V_{\xi})}\times\prod_{\eta\nmid p}\frac{\rH^1(F_\eta,V_{\xi})}{\rH^1_{\rm ur}(F_\eta,V_{\xi})}\biggr\}.
\]
\end{itemize}
\end{defi}

Let $T_\xi$ denote the free rank $1$ $\scO$-module on which $G_K$ acts via $\xi$, and put $W_\xi=V_\xi/T_\xi$. Similarly as above, for $\star\in\{\BK,v,\overline{v}\}$, we define ${\rm Sel}_\star(F,T_\xi)$ and ${\rm Sel}_\star(F,W_\xi)$ by propagation via $0\rightarrow T_\xi\rightarrow V_\xi\rightarrow W_\xi\rightarrow 0$, define the $\Lambda_\scO$-adic Selmer groups
\[
\check{\mathcal{S}}_\star(\xi):=\varprojlim_{F\subset K_\infty}{\rm Sel}_\star(F,T_\xi),\quad\quad
\mathcal{S}_\star(\xi):=\varinjlim_{F\subset K_\infty}{\rm Sel}_\star(F,W_\xi),
\]
with $F$ running over the finite extensions of $K$ contained in $K_\infty$, 
and let 
\[
\mathcal{X}_\star(\xi):={\rm Hom}_{\rm cts}(\mathcal{S}_\star(\xi),\Q_p/\Z_p)
\] 
be the Pontryagin dual of $\mathcal{S}_\star(\xi)$.

\begin{lemma}\label{lem:BK=v}
Suppose $\xi$ has infinity type $(a,b)$. Let $F/K$ be a finite extension and let $\eta$ be a prime of $F$ above $p$. Then
\[
\rH^1_f(F_\eta,V_\xi)=\begin{cases}
\rH^1(F_\eta,V_\xi)&\textrm{if $\eta\mid v$ and $a>0$, or $\eta\mid\overline{v}$ and $b>0$,} \\[0.2em]
0&\textrm{else.}
\end{cases}
\]
In particular, 
\[
\mathcal{X}_\BK(\xi)=\begin{cases}
\mathcal{X}_v(\xi)&\textrm{if $a>0$ and $b\leq 0$,}\\[0.2em]
\mathcal{X}_{\overline{v}}(\xi)& \textrm{if $b>0$ and $a\leq 0$.}
\end{cases}
\]
\end{lemma}

\begin{proof}
%
With the convention that the $p$-adic cyclotomic character $\varepsilon_{\rm cyc}:G_{\Q}\rightarrow\Z_p^\times$ has Hodge--Tate weight $+1$, our convention on infinity types implies that the $p$-adic avatar $\xi:G_K\rightarrow\scO$ has Hodge--Tate  weight $a$ (resp. $b$) at $v$ (resp. $\overline{v}$). In view of \cite[Thm.~4.1(ii)]{BK}, this implies the result.
\end{proof}

\subsection{Decompositions}

When $g$ is the theta series of a Hecke character $\psi$ of $K$, 
then $V_g\simeq{\rm Ind}_K^\Q V_\psi$ and so
\begin{equation}\label{eq:dec-V}
V_{g,\chi}\simeq V_{\psi\chi}(r)\oplus V_{\psi^\st\chi}(r).
\end{equation}
We fix an isomorphism as above, and let $T_{g,\chi}\subset V_{g,\chi}$ be the $G_K$-stable $\scO$-lattice with $T_{g,\chi}\simeq T_{\psi\chi}(r)\oplus T_{\psi^\st\chi}(r)$ under that isomorphism, so that
\[
W_{g,\chi}:=V_{g,\chi}/T_{g,\chi}\simeq W_{\psi\chi}(r)\oplus W_{\psi^\st\chi}(r)
\]
as $G_K$-modules.

\begin{prop}\label{prop:dec-Sel}
Suppose $g=\theta_\psi\in S_{2r}(\Gamma_1(N_g))$ is the theta series of a Hecke character $\psi$ of $K$, and $\chi$ is a finite order character with $\chi\in\Sigma_{\rm cc}(c,\mathfrak{N}_g,\varepsilon_g)$. 
Then there are $\Lambda_\scO$-module isomorphisms
\begin{align*}
\check{\mathcal{S}}_v(g,\chi)&\simeq\check{\mathcal{S}}_v(\psi\chi\nr)\oplus\check{\mathcal{S}}_v(\psi^\st\chi\nr),\\
\check{\mathcal{S}}_{\BK}(g,\chi)&\simeq\check{\mathcal{S}}_{\overline{v}}(\psi\chi\nr)\oplus\check{\mathcal{S}}_{v}(\psi^\st\chi\nr),
\end{align*}
and similarly
\begin{align*}
\mathcal{X}_{v}(g,\chi)&\simeq\mathcal{X}_v(\psi\chi\nr)\oplus\mathcal{X}_{v}(\psi^\st\chi\nr),\\
\mathcal{X}_{\BK}(g,\chi)&\simeq\mathcal{X}_{\overline{v}}(\psi\chi\nr)\oplus\mathcal{X}_{v}(\psi^\st\chi\nr).
\end{align*}
\end{prop}

\begin{proof}
The first and third isomorphisms are immediate from \eqref{eq:dec-V} and the definitions. On the other hand, note that the character $\psi\chi\varepsilon_{\rm cyc}^{r}$ has Hodge--Tate weight $1-r$ and $r$ above $v$ and $\overline{v}$, respectively, while $\psi^\st\chi\varepsilon_{\rm cyc}^{r}$ has Hodge--Tate weight $r$ and $1-r$ above $v$ and $\overline{v}$, respectively. In view of Lemma~\ref{lem:BK=v}, the second and fourth isomorphisms thus follow the decomposition
\[
\rH^1_f(F_\eta,V_{g,\chi})=\rH^1_f(F_\eta,V_{\psi\chi}(r))\oplus\rH^1_f(F_\eta,V_{\psi^\st\chi}(r))
\]
induced by \eqref{eq:dec-V} for any $\eta\mid p$.
\end{proof}

\begin{rem}\label{rem:reversed}
In the situation of Proposition~\ref{prop:dec-Sel}, the local conditions defining the Selmer groups $\check{\mathcal{S}}_v(\psi\chi\nr)$ and $\mathcal{X}_v(\psi\chi\nr)$ at the primes above $p$ are \emph{reversed} with respect to those defining the corresponding Bloch--Kato Selmer groups $\check{\mathcal{S}}_\BK(\psi\chi\nr)=\check{\mathcal{S}}_{\overline{v}}(\psi\chi\nr)$ and $\mathcal{X}_\BK(\psi\chi\nr)=\mathcal{X}_{\overline{v}}(\psi\chi\nr)$. 
\end{rem}

\subsection{Equivalence between the Main Conjectures}\label{subsec:equivalence}


In the case where $V_{g,\chi}\vert_{G_K}$ is irreducible, one can show that  $\mathscr{L}_w(g,\chi)\neq 0$ for both values of $w\in\{v,\overline{v}\}$ (see \cite[Thm.~3.9]{cas-hsieh1}), and Conjecture~\ref{conj:HPMC} can be shown to be equivalent to either version ($v$ or $\overline{v}$) of Conjecture~\ref{conj:BDP-IMC}. In the weight $2$ case, the argument for this appears in several places in the literature (see e.g. \cite[\S{3}]{wan-heegner} 
or \cite[\S{5}]{BCK}). 

In Theorem~\ref{thm:Heeg-nonzero} below, we shall see that when $g=\theta_\psi$ has CM by $K$ (and so $V_{g,\chi}\vert_{G_K}$ is reducible), the vanishing or not of $\mathscr{L}_w(g,\chi)$ \emph{depends} on the value of the root number 
\[
w(\psi\chi\nr)\in\{\pm{1}\}
\] 
in the functional equation relating $L(\psi\chi,s)$ and $L(\psi\chi,2r-s)$. In fact, in this case we shall see that $\mathscr{L}_w(g,\chi)$ vanishes for one value of $w(\psi\chi\nr)$, and it is nonzero for the other (see Remark~\ref{rem:v-vs-vbar}). Thus in this section we carefully explain the relation between Conjecture~\ref{conj:HPMC} and Conjecture~\ref{conj:BDP-IMC}. 


\begin{lemma}\label{lem:local-tors-free}
For any prime $w$ of $K$ above $p$, the quotient
\[
\mathcal{H}_{w}:=\varprojlim_{F\subset K_\infty}\prod_{\eta\vert w}\frac{\rH^1(F_\eta,T_{g,\chi})}{\rH^1_f(F_\eta,T_{g,\chi})}
\]
is $\Lambda_\scO$-torsion-free.
\end{lemma}

\begin{proof}
From Lemma~\ref{lem:BK=v} we immediately see that
\[
\mathcal{H}_v\simeq\varprojlim_{F\subset K_\infty}\prod_{\eta\vert v}\rH^1(F_\eta,T_{\psi\chi}(r)),\quad\quad
\mathcal{H}_{\overline{v}}\simeq\varprojlim_{F\subset K_\infty}\prod_{\eta\vert\overline{v}}\rH^1(F_\eta,T_{\psi^\st\chi}(r)).
\]
The result thus follows from \cite[Prop.~2.1.6]{PR-92} as in \cite[Lem.~2.8]{arnold}. (Note that \cite[Lem.~2.5]{arnold} also applies in our case, since neither of the characters $\psi\chi\mathbf{N}^{-r}$ and $\psi^\st\chi\mathbf{N}^{-r}$ has infinity type of the form $(a,b)$ with $a=-b$.)
\end{proof}


\begin{prop}\label{prop:equiv}
Suppose $\mathbf{z}_{g,\chi}$ is non-torsion, and the localization map
\begin{equation}\label{eq:Sel-locv}
{\rm loc}_{\overline{v}}:\check{\mathcal{S}}_\BK(g,\chi)\rightarrow\varprojlim_{F\subset K_\infty}\prod_{\eta\vert\overline{v}}\rH^1_f(F_\eta,T_{g,\chi})
\end{equation}
is nonzero. Then the following are equivalent:
\begin{itemize}
\item[$(i^+)$] 
${\rm rank}_{\Lambda_\scO}\bigl(\mathcal{X}_{{v}}(g,\chi)\bigr)={\rm rank}_{\Lambda_\scO}\bigl(\check{\mathcal{S}}_v(g,\chi)\bigr)=0$;
\item[$(i^-)$] 
${\rm rank}_{\Lambda_\scO}\bigl(\mathcal{X}_\BK(g,\chi)\bigr)={\rm rank}_{\Lambda_\scO}\bigl(\check{\mathcal{S}}_\BK(g,\chi)\bigr)=1$,
\end{itemize}
and in that case, the following are equivalent:
\begin{itemize}
\item[$(ii^+)$] ${\rm char}_{\Lambda_\scO}\bigl(\mathcal{X}_{v}(g,\chi)\bigr)\subset\bigl(\mathscr{L}_v(g,\chi)^2\bigr)$ in $\Lambda_\scO^{\rm ur}$;
\item[$(ii^-)$] ${\rm char}_{\Lambda_\scO}\bigl(\mathcal{X}_\BK(g,\chi)_{\rm tors}\bigr)\subset{\rm char}_{\Lambda_\scO}\bigl(\check{\mathcal{S}}_\BK(g,\chi)/{\Lambda_\scO}\cdot\mathbf{z}_{g,\chi}\bigr)^2$ in ${\Lambda_\scO}$,
\end{itemize}
and the same holds true for the opposite divisibilities. 
%
In particular, 
Conjecture~\ref{conj:BDP-IMC} for $w=v$ and Conjecture~\ref{conj:HPMC} are equivalent.
\end{prop}

\begin{proof}
For $\circ,\bullet\in\{{\rm str},f,{\rm rel}\}$, let $\check{\mathcal{S}}_{\circ,\bullet}(g,\chi)$ and $\mathcal{X}_{\circ,\bullet}(g,\chi)$ denote the Selmer groups defined as in $\S\ref{subsec:Sel-sd}$ by with the local conditions $\circ$ and $\bullet$ at the primes above $v$ and $\overline{v}$, respectively. Thus, for instance, $\mathcal{X}_{{\rm rel},{\rm str}}(g,\chi)$ and $\mathcal{X}_{f,f}(g,\chi)$ are the previously defined $\mathcal{X}_{{v}}(g,\chi)$ and $\mathcal{X}_{\BK}(g,\chi)$, respectively. 

By global duality, ${\rm coker}({\rm loc}_{\overline{v}})$ is the same\footnote{Using the fact that the representation $V_{g,\chi}$ is conjugate self-dual, so the Selmer group dual to $\check{\mathcal{S}}_{f,{\rm str}}(g,\chi)={\rm ker}({\rm loc}_{\overline{v}})$ is identified with $\mathcal{X}_{{\rm rel},f}(g,\chi)$.
} as the kernel of the  projection $\mathcal{X}_{{\rm rel},f}(g,\chi)\rightarrow\mathcal{X}_\BK(g,\chi)$. Similarly, by global duality the cokernel of the localization map
\[
{\rm loc}_{\overline{v}}':\check{\mathcal{S}}_{{\rm rel},f}(g,\chi)\rightarrow\varprojlim_{F\subset K_\infty}\prod_{\eta\vert\overline{v}}\rH^1_f(F_\eta,T_{g,\chi})
\]
is identified with the kernel of the projection $\mathcal{X}_{v}(g,\chi)\rightarrow\mathcal{X}_{{f},{\rm str}}(g,\chi)$. Since the same argument as in \cite[Lem.~2.3]{cas-JLMS} shows that ${\rm rank}_{\Lambda_\scO}(\mathcal{X}_{{\rm rel},f}(g,\chi))=1+{\rm rank}_{\Lambda_\scO}(\mathcal{X}_{f,{\rm str}}(g,\chi))$, under the assumption on ${\rm loc}_{\overline{v}}$ in the statement we thus conclude that
\[
{\rm rank}_{\Lambda_\scO}(\mathcal{X}_{\rm BK}(g,\chi))=1+{\rm rank}_{\Lambda_\scO}(\mathcal{X}_v(g,\chi)),
\]
whence the equivalence $(i^+)\Longleftrightarrow(i^-)$.

Now assume that $(i^+)$ (and therefore also $(i^-)$) holds. As above, from global duality we have the short exact sequences
\begin{equation}\label{eq:locv}
\begin{aligned}
0\rightarrow{\rm coker}({\rm loc}_{\overline{v}})&\rightarrow\mathcal{X}_{{\rm rel},f}(g,\chi)\rightarrow\mathcal{X}_\BK(g,\chi)\rightarrow 0,\\
0\rightarrow{\rm coker}({\rm loc}_{\overline{v}}')&\rightarrow\mathcal{X}_{v}(g,\chi)\rightarrow\mathcal{X}_{f,{\rm str}}(g,\chi)\rightarrow 0.
\end{aligned}
\end{equation}
Since the target of ${\rm loc}_{\overline{v}}'$ has $\Lambda_\scO$-rank $1$, and  ${\rm rank}_{\Lambda_\scO}(\check{\mathcal{S}}_{v}(g,\chi))=0$, the quotient $\check{\mathcal{S}}_{{\rm rel},f}(g,\chi)/\check{\mathcal{S}}_\BK(g,\chi)$ is $\Lambda_\scO$-torsion; since this quotient injects via ${\rm loc}_v$ into $\mathcal{H}_v$, which is $\Lambda_\scO$-torsion-free by Lemma~\ref{lem:local-tors-free}, we conclude that $\check{\mathcal{S}}_{{\rm rel},f}(g,\chi)=\check{\mathcal{S}}_\BK(g,\chi)$, so in particular ${\rm loc}_{\overline{v}}={\rm loc}_{\overline{v}}'$. From \eqref{eq:locv} we thus obtain
\begin{equation}\label{eq:take-char}
\begin{aligned}
{\rm char}_{\Lambda_\scO}(\mathcal{X}_v(g,\chi))&={\rm char}_{\Lambda_\scO}(\mathcal{X}_{f,{\rm str}}(g,\chi))\cdot{\rm char}_{\Lambda_\scO}({\rm coker}({\rm loc}_{\overline{v}}))\\
&={\rm char}_{\Lambda_\scO}(\mathcal{X}_{{\rm rel},f}(g,\chi)_{\rm tors})\cdot{\rm char}_{\Lambda_\scO}({\rm coker}({\rm loc}_{\overline{v}}))\\
&={\rm char}_{\Lambda_\scO}(\mathcal{X}_\BK(g,\chi)_{\rm tors})\cdot{\rm char}_{\Lambda_\scO}({\rm coker}({\rm loc}_{\overline{v}}))^2,
\end{aligned}
\end{equation}
using a variant of \cite[Lem.~2.3(3)]{cas-JLMS} for the second equality. On the other hand, since $\varprojlim_F\rH^1(F,T_{g,\chi})$ is $\Lambda_\scO$-torsion--free, the map ${\rm loc}_{\overline{v}}$ defines the short exact sequence
\[
0\rightarrow\check{\mathcal{S}}_\BK(g,\chi)/\Lambda_\scO\cdot\mathbf{z}_{g,\chi}\rightarrow\varprojlim_{F\subset K_\infty}\oplus_{\eta\vert\overline{v}}\rH^1_f(F_\eta,T_{g,\chi})/\Lambda_\scO\cdot{\rm loc}_{\overline{v}}(\mathbf{z}_{g,\chi})\rightarrow{\rm coker}({\rm loc}_{\overline{v}})\rightarrow 0,
\]
which together with Theorem~\ref{thm:ERL} yields the equality
\begin{equation}\label{eq:ERL-cor}
{\rm char}_{\Lambda_\scO}\bigl(\check{\mathcal{S}}_\BK(g,\chi)/\Lambda_\scO\cdot\mathbf{z}_{g,\chi}\bigr)\cdot{\rm char}_{\Lambda_\scO}({\rm coker}({\rm loc}_{\overline{v}}))=\bigl(\mathscr{L}_v(g,\chi)\bigr)
\end{equation}
in $\Lambda_\scO^{\rm ur}$. 
Combining \eqref{eq:take-char} and \eqref{eq:ERL-cor} we deduce
\[
{\rm char}_{\Lambda_\scO}(\mathcal{X}_{v}(g,\chi))\cdot{\rm char}_{\Lambda_\scO}\bigl(\check{\mathcal{S}}_\BK(g,\chi)/\Lambda_\scO\cdot\mathbf{z}_{g,\chi}\bigr)^2={\rm char}_{\Lambda_\scO}(\mathcal{X}_{\BK}(g,\chi)_{\rm tors})\cdot\bigl(\mathscr{L}_v(g,\chi)^2\bigr),
\]
which readily yields the equivalence $(ii^+)\Longleftrightarrow(ii^-)$.
\end{proof}


\begin{rem}\label{rem:sign-psichi}
Theorem~\ref{thm:ERL} also relates the image of $\mathbf{z}_{g,\chi}$ under the localization map
\[
{\rm loc}_{v}:\check{\mathcal{S}}_\BK(g,\chi)\rightarrow\varprojlim_{F\subset K_\infty}\prod_{\eta\vert v}\rH^1_f(F_\eta,T_{g,\chi})
\]
to the $p$-adic $L$-function $\mathscr{L}_{\overline{v}}(g,\chi)$. 
The same argument as in Proposition~\ref{prop:equiv} shows that, assuming the non-vanishing of both $\mathbf{z}_{g,\chi}$ and ${\rm loc}_{v}(\check{\mathcal{S}}_\BK(g,\chi))$,  Conjecture~\ref{conj:BDP-IMC} for $w=\overline{v}$ and Conjecture~\ref{conj:HPMC} are equivalent.
%
\end{rem}

\section{Proof of the Iwasawa Main Conjectures}

In this section we prove Conjecture~\ref{conj:BDP-IMC} and Conjecture~\ref{conj:HPMC} for self-dual pairs $(g,\chi)$ in the case where $g=\theta_\psi$ is the theta series  of a Hecke character $\psi$. In the course of the proof we shall establish the non-triviality of $\mathbf{z}_{g,\chi}$ in this case (which in weights $2r>2$ appears to be new).

\subsection{Statement of the main result}

Let $K$ be an imaginary quadratic field satisfying \eqref{eq:spl} for a prime $p>3$. For a self-dual character $\phi$, denote by $w(\phi)\in\{\pm{1}\}$ the sign in the functional equation
\[
L(\phi,s)=w(\phi)L(\phi,-s).
\]
The main result of this section is the following.

\begin{thm}\label{thm:BDP-IMC}
Let $\psi$ be a Hecke character of infinity type $(1-2r,0)$ for some $r\geq 1$ and conductor a cyclic ideal $\mathfrak{f}_\psi$ of norm prime to $pD_K$. Set 
\[
g=\theta_\psi,
\quad 
N_g=D_K\cdot\mathbf{N}(\mathfrak{f}_\psi),\quad\mathfrak{N}_g=\mathfrak{d}_K\cdot\mathfrak{f}_\psi.
\]
Let $c$ be a positive integer prime to $pN_g$, let 
$\chi$ be a finite order character in  
$\Sigma_{\rm cc}(c,\mathfrak{N}_g,\varepsilon_g)$, and suppose $(g,\chi)$ is defined over $\scO$. %
\begin{itemize}
\item[(i)] If $w(\psi\chi\nr)=-1$, then 
$\check{\mathcal{S}}_v(g,\chi)=0$ and $\mathcal{X}_{v}(g,\chi)$ is $\Lambda_{\scO}$-torsion, with
\[
{\rm char}_{\Lambda_\scO}\bigl(\mathcal{X}_{v}(g,\chi)\bigr)={\rm char}_{\Lambda_\scO}\bigl(\mathscr{L}_v(g,\chi)^2\bigr)
\]
in $\Lambda_\scO^{\rm ur}$.
\item[(ii)] If $w(\psi\chi\nr)=+1$, the same results hold with $v$ replaced by $\overline{v}$.
\end{itemize}
Hence Conjecture~\ref{conj:BDP-IMC} holds for $(g,\chi)$.
\end{thm}

\begin{rem}
The above assumption on $\mathfrak{f}_\psi$ implies that $\cO_K/\mathfrak{N}_g\simeq\Z/N_g\Z$, i.e. $K$ satisfies \eqref{eq:Heeg} relative to $N_g$. Moreover, it is easy to see that in the setting of Theorem~\ref{thm:BDP-IMC}, the character $\psi\chi\nr$ (and hence also $\psi^\st\chi\nr$) is self-dual (see \cite[Lem.~3.16]{BDP-PJM}).
\end{rem}


By Proposition~\ref{prop:factor-BDP} and Proposition~\ref{prop:dec-Sel}, the proof of Theorem~\ref{thm:BDP-IMC} is reduced to the study of the  
relation between the $\Lambda_\scO$-adic Selmer groups attached to the self-dual characters $\psi\chi\nr$ and $\psi^\st\chi\nr$ and the $p$-adic $L$-functions $\mathscr{L}_w(\psi\chi\nr)$ and $\mathscr{L}_w(\psi^\st\chi\nr)$, respectively. The main difficulty arises from the fact that the Selmer group $\mathcal{X}_v(\psi\chi\nr)$ is different from the Bloch--Kato Selmer group over $K_\infty/K$ (see Remark~\ref{rem:reversed}), and so a relation between the characteristic ideal of $\mathcal{X}_v(\psi\chi\nr)$  and the $p$-adic $l$-function $\mathscr{L}_v(\psi\chi\nr)$ is not immediate from the Main Conjecture. 

The proof of Theorem~\ref{thm:BDP-IMC} is concluded in $\S\ref{subsec:IMC-proof-end}$, where we also deduce a similar result on 
Conjecture~\ref{conj:HPMC}.

\subsection{Explicit reciprocity law}\label{subsec:ERL-katz}

For an ideal $\cc\subset\cO_K$ prime to $p$ and a non-trivial ideal $\mathfrak{a}$ prime to $6\cc p$, let $c_\mathfrak{a}(K(\cc p^k))\in\rH^1(K(\cc p^k),\Z_p(1))$ be the elliptic unit denoted $\vartheta_\mathfrak{a}(\mathfrak{\cc} p^k)$ in \cite[\S{2.3}]{AH-ord} and ${}_\mathfrak{a}z_{\cc p^k}$ in \cite[\S{15.5}]{kato-295}; these are norm-compatible as $k$ varies. 

As a piece of notation, for an infinite abelian extension $K'/K$ and a $\Z_p$-module $T$ with a continuous linear $G_K$-action, put
\[
\rH^1_{\rm Iw}(K',T):=\varprojlim_{K''\subset K'}\rH^1(K'',T),
\]
with $K''$ running over the finite extensions of $K$ contained in $K'$, and limit with respect to corestriction. For a character $\phi:{\rm Gal}(K(\cc p^\infty)/K)\rightarrow\scO^\times$, we put $\phi^*=\varepsilon_{\rm cyc}\phi^{-1}$ and let $c_\mathfrak{a}(K(\cc p^\infty))^{\phi^{-1}}$ denote the image of $\varprojlim_k c_{\mathfrak{a}}(K(\cc p^k))$ under the twisting homomorphism
\[
\rH^1_{\rm Iw}(K(\cc p^\infty),\Z_p(1))\xrightarrow{\otimes\phi^{-1}}\rH^1_{\rm Iw}(K(\cc p^\infty),T_{\phi^*}).
\]
For any subextension $L$ of $K(\cc p^\infty)$ (not necessarily finite over $K$), we define $c_{\mathfrak{a}}(L)^{\phi^{-1}}$ to be the image of $c_\mathfrak{a}(K(\cc p^\infty))^{\phi^{-1}}$ under the corestriction map $\rH^1_{\rm Iw}(K(\cc p^\infty),T_{\phi^*})\rightarrow\rH^1_{\rm Iw}(L,T_{\phi^*})$. 

%
%
Here we give a refinement of Yager's work \cite{yager} (restricted to the anticyclotomic line) building on Kato's explicit reciprocity law \cite{kato-295}.

\begin{thm}\label{thm:ERL-katz}
Let $w\in\{v,\overline{v}\}$ be a prime of $K$ above $p$, and let $\phi$ be a self-dual Hecke character of $K$ with values in $\scO$. Then there is an injective $\Lambda_\scO^{\rm ur}$-module homomorphism
\[
{\rm Col}_w:\varprojlim_{F\subset K_\infty}\prod_{\eta\vert w}\rH^1(F_\eta,T_{\phi^\st})\hat\otimes_{\Z_p}\Z_p^{\rm ur}\rightarrow\Lambda_\scO^{\rm ur}
\]
with finite cokernel, and a ``twisted elliptic unit'' $\mathbf{c}_{\phi^\st}\in\rH^1_{\rm Iw}(K_\infty,T_{\phi^\st})$, such that
\[
{\rm Col}_w\bigl({\rm loc}_w(\mathbf{c}_{\phi^\st})\bigr)=\mathscr{L}_w(\phi).
\]
\end{thm}

\begin{proof}
We first explain the construction of the map ${\rm Col}_w$ for $w=v$.   
For any field $F\subset\overline{\Q}$ over $K$, let $F_v$ denote the completion of $F$ at the prime above $p$ induced by $\imath_p$.  Then $H_{p^\infty,v}$ contains $K_{\infty,v}$, and it follows by local Class Field Theory that $H_{p^\infty,v}$ is obtained by adjoining to $H_v$ the torsion points of a height $1$ Lubin--Tate formal group relative to the extension $H_v/K_v$ (see \cite[Prop.~39]{shnidman-AIF}). 
Thus, as explained in \cite[\S{3}]{cas-hsieh-GKC}, from the Perrin-Riou big exponential map for $H_{p^\infty}/H_v$ (see \cite{PR-94,kobayashi-AMQ}) we have a $\Z_p[\![\Gamma_v]\!]$-linear map
\[
{\rm Col}_v:\varprojlim_{F\subset K_\infty}\rH^1(F_v,T_{\phi^\st})\hat\otimes_{\Z_p}\Z_p^{\rm ur}\rightarrow\scO[\![\Gamma_v]\!]\hat\otimes_{\Z_p}\Z_p^{\rm ur}
\]
interpolating the Bloch--Kato logarithm and dual exponential maps for varying specializations (see \cite[Thm.~3.4]{cas-hsieh-GKC}) which gives the map ${\rm Col}_v$ in the statement after tensoring with $\Z_p[\![\Gamma]\!]$ over $\Z_p[\![\Gamma_v]\!]$. The maps ${\rm Col}_{\overline{v}}$ is constructed in the same manner, replacing $\imath_p$ by $\imath_p\circ\st$, where $\st\in G_\Q$ is the complex conjugation induced by $\imath_\infty$;  the claim that both maps ${\rm Col}_w$ are injective with finite cokernel follows form the theory of Coleman power series as in \cite[\S{17.10}]{kato-295}.

On the other hand, letting $\cc$ be the prime-to-$p$ part of the conductor of $\phi$, we can find $\mathfrak{a}$ coprime to $6\cc p$ with $\sigma_{\mathfrak{a}}-\phi^\st(\sigma_\mathfrak{a})$ invertible in $\Lambda_{\scO}^{\rm ur}$ (cf. \cite[p.\,77]{deshalit}), and setting
\[
\mathbf{c}_{\phi^\st}:=(\sigma_{\mathfrak{a}}-\phi^\tau(\sigma_\mathfrak{a}))^{-1}\cdot c_{\mathfrak{a}}(K_\infty)^{\phi^{-1}}
\] 
(which is independent of $\mathfrak{a}$ by \cite[(15.4.4]{kato-295} and belongs to $\rH^1_{\rm Iw}(K_\infty,T_{\phi^*})=\rH^1_{\rm Iw}(K_\infty,T_{\phi^\st})$ by the self-duality of $\phi$), the last assertion follows directly from the interpolations properties of $\mathscr{L}_w(\phi)$ and ${\rm Col}_w$ together with Kato's explicit reciprocity law \cite[Prop.~15.9]{kato-295}.
\end{proof}

\begin{rem} 
Specializing Yager's result \cite{yager} to the anticyclotomic line as in \cite[Prop.~2.3.4]{AH-ord} and \cite[Prop.~2.6]{arnold} yields an analogue of Theorem~\ref{thm:ERL-katz} under the assumption that $p\nmid [K(\cc):K]$ and $\phi\vert_{G_{K_w}}\not\equiv 1\pmod{\varpi}$ 
(i.e. $\phi$ is ``non-anomalous'' at $w$). Indeed, following the argument in \emph{loc.\,cit.}, by local Tate duality one can show that under these additional hypotheses the corestriction map 
\[
\biggl(\varprojlim_{F\subset K(\cc p^\infty)}\prod_{\eta\vert w}\rH^1(F_\eta,T_{\phi^*})\biggr)\otimes_{\Lambda_\scO(K(\cc p^\infty))}\Lambda_{\scO}\rightarrow\varprojlim_{F\subset K_\infty}\prod_{\eta\vert w}\rH^1(F_\eta,T_{\phi^*})
\]
is an isomorphism. 
\end{rem}

\subsection{Main Conjectures for characters}\label{subsec:IMC-char}

The Iwasawa Main Conjecture for $K$ was proved by Rubin \cite{rubin-IMC} under mild hypotheses on the prime $p$, and by Johnson-Leung--Kings \cite{kings-JL} in general. 

%
The problem of deducing from the $2$-variable Main Conjecture a proof of the anticyclotomic Main Conjecture for self-dual Hecke characters $\phi$ 
was first studied in detail by Agboola--Howard \cite{AH-ord} in the case of CM elliptic curves $E/\Q$, and by Arnold \cite{arnold} for higher weight CM forms. The main result of this section is Theorem~\ref{thm:AH-lambda} below, which we deduce from an adaptation of their methods.
%

\subsubsection{The case of $\lambda=\psi\chi$} 

\begin{thm}\label{thm:AH-lambda}
Let $\lambda$ be a Hecke character of infinity type $(1-2r,0)$ for some $r\geq 1$ with $\varepsilon_\lambda=\eta_K$, conductor $\cc$ prime to $p$. Suppose 
\[
w(\lambda\nr)=-1.
\] 
Then $\check{\mathcal{S}}_v(\lambda\nr)=0$ and $\mathcal{X}_v(\lambda\nr)$ is $\Lambda_\scO$-torsion, with
\[
{\rm char}_{\Lambda_\scO}\bigl(\mathcal{X}_v(\lambda\nr)\bigr)=\bigl(\mathscr{L}_v(\lambda\mathbf{N}^{-r})\bigr)
\]
as ideals in $\Lambda^{\rm ur}$. 
\end{thm}


%
Following the method of \cite{AH-ord} and \cite{arnold}, by anticyclotomic descent, the proof of Theorem~\ref{thm:AH-lambda} will be based on a twisted variant of the results by Rubin \cite{rubin-IMC} (and Johnson-Leung--Kings \cite{kings-JL} more generally) on the $2$-variable Iwasawa Main Conjecture for $K$.
%
The starting point is the following key consequence of Greenberg's nonvanishing results \cite{greenberg-critical}. 

\begin{prop}\label{prop:Greenberg-nonv}
Let $w(\lambda\nr)\in\{\pm{1}\}$ be the root number of $\lambda\nr$. Then
\begin{align*}
\mathscr{L}_v(\lambda\nr)\neq 0\quad&\Longleftrightarrow\quad w(\lambda\nr)=-1.
\end{align*}
\end{prop}

\begin{proof}
We adapt the argument in the proof of \cite[Prop.~2.3]{arnold}. Put $\phi_0=\overline{\lambda}/\lambda$, and take $m>0$ (to be possibly enlarged later) large enough so that $\phi:=\phi_0^m$ factors through $\Gamma$ and has trivial conductor. Then for $n\gg 0$, the character $\lambda\mathbf{N}^{-r}\phi^{n}$ is in the range of interpolation of $\mathscr{L}_{v,\cc}$ (indeed, its infinity type $(1-r+mn(2r-1),r+mn(1-2r))$) and so from Theorem~\ref{thm:katz} and the functional equation for Hecke $L$-functions\footnote{See e.g. \cite[p.\,37]{deshalit}, whose conventions on infinity type are \emph{opposite} to ours.} we immediately obtain
\[
\mathscr{L}_v(\lambda\mathbf{N}^{-r})(\phi^{n})\doteq L(\lambda^{2mn-1},c),
\]
where $c=mn(2r-1)-r+1$ is the center of the functional equation. Applying Weil's formula for root numbers as stated in \cite[Prop.~2.1.6]{AH-ord}, arguing as in \cite[p.\,51]{arnold} we see that (after possibly replacing $m$ by $2m$) $w(\lambda^{2mn-1})=-w(\lambda\nr)$, and the result follows from \cite[Thm.~1]{greenberg-critical}. 
\end{proof}

\begin{prop}\label{prop:Xv-tors}
If $w(\lambda\nr)=-1$, then $\mathcal{X}_v(\lambda\nr)$ is $\Lambda_\scO$-torsion.
\end{prop}

\begin{proof}
By Theorem~\ref{thm:ERL-katz} (with $\phi=\lambda\mathbf{N}^{-r}$) and Proposition~\ref{prop:Greenberg-nonv}, if $w(\lambda\nr)=-1$ then the map
\begin{equation}\label{eq:loc-vbar-BK}
{\rm loc}_{v}:\check{\mathcal{S}}_{\rm rel}(\lambda^\st\nr)\rightarrow\varprojlim_{F\subset K_\infty}\prod_{\eta\vert v}\rH^1(F_\eta,T_{\lambda^\st}(r))\nonumber
\end{equation}
is nonzero\footnote{Note that the implicit inclusion $\mathbf{c}_{\lambda^\tau\mathbf{N}^{-r}}\in\check{\mathcal{S}}_{\rm rel}(\lambda^\st\nr)$ follows from 
\cite[Lem.~2.4.2]{AH-ord}.}. On the other hand, by \cite[Thm.~2.14]{arnold} (an application of the Euler system machinery, \cite[Thm.~2.3.3]{rubin-ES}), the $\Lambda_\scO$-module $\mathcal{X}_{\rm str}(\lambda\nr)$ is torsion. By the global duality exact sequence
\[
\check{\mathcal{S}}_{\rm rel}(\lambda^\tau\nr)\rightarrow\varprojlim_{F\subset K_\infty}\prod_{\eta\vert v}\rH^1(F_\eta,T_{\lambda^\tau}(r))\rightarrow\mathcal{X}_{v}(\lambda\nr)\rightarrow\mathcal{X}_{\rm str}(\lambda\nr)\rightarrow 0,
\]
this yields the result.
\end{proof}




\begin{prop}\label{prop:char-v}
Suppose $w(\lambda\nr)=-1$. Then 
\[
{\rm char}_{\Lambda_\scO}\bigl(\mathcal{X}_v(\lambda\nr)\bigr)=\bigl(\mathscr{L}_v(\lambda\mathbf{N}^{-r})\bigr)
\]
as ideals in $\Lambda_\scO^{\rm ur}$.
\end{prop}

\begin{proof}
Denote by $\mathcal{C}_{\lambda^\st\nr}(K_\infty)$ the $\Lambda_\scO$-submodule of $\check{\mathcal{S}}_{\rm rel}(\lambda^\tau\nr)$ generated by $\mathbf{c}_{\lambda^{\st}\mathbf{N}^{-r}}$. By Theorem~\ref{thm:ERL-katz} (with $\phi=\lambda\mathbf{N}^{-r}$) and Proposition~\ref{prop:Greenberg-nonv}, the module $\mathcal{C}_{\lambda^\st\nr}(K_\infty)$ is not $\Lambda_\scO$-torsion, and by Theorem~2.14 in \cite{arnold} we know that $\check{\mathcal{S}}_{\rm rel}(\lambda^\st\nr)$ is torsion-free of $\Lambda_\scO$-rank $1$, $\mathcal{X}_{\rm str}(\lambda\nr)$ is $\Lambda_\scO$-torsion, and we have the divisibility
\begin{equation}\label{eq:ac-div}
{\rm char}_{\Lambda_\scO}\bigl(\mathcal{X}_{\rm str}(\lambda\nr)\bigr)\supset{\rm char}_{\Lambda_\scO}\bigl(\check{\mathcal{S}}_{\rm rel}(\lambda^\st\nr)/\mathcal{C}_{\lambda^\st\nr}(K_\infty)\bigr).\nonumber
\end{equation}
Note that this divisibility is independent of the value of $w(\lambda\nr)$, and in \emph{loc.\,cit.} the result is stated for $T_\lambda(r)$ and $W_{\lambda^\st}(r)$ in our notations (corresponding to $T$ and $W^*$ in \emph{loc.\,cit.}), rather than for our $T_{\lambda^\st}(r)$ and $W_\lambda(r)$ as above.

On the other hand, if $w(\lambda\nr)=-1$, in view of Proposition~\ref{prop:Xv-tors} the arguments in \cite[\S{3}]{arnold}\footnote{Using  \cite[Thm.~5.2]{kings-JL}, which 
proves an extension of Rubin's result on the $2$-variable Iwasawa Main Conjecture for $K$ used in \cite[Thm.~3.2]{arnold} under weaker hypotheses on $p$.} leading to the proof of Proposition~3.8 in \emph{loc.\,cit.} apply \emph{verbatim} with $T_\lambda(r)$ and $W_{\lambda^\st}(r)$ replaced by $T_{\lambda^\tau}(r)$ and $W_{\lambda}(r)$, hence yielding a proof of the converse divisibility, and so
\begin{equation}\label{eq:ac-eq}
{\rm char}_{\Lambda_\scO}\bigl(\mathcal{X}_{\rm str}(\lambda\nr)\bigr)={\rm char}_{\Lambda_\scO}\bigl(\check{\mathcal{S}}_{\rm rel}(\lambda^\st\nr)/\mathcal{C}_{\lambda^\st\nr}(K_\infty)\bigr).
\end{equation}
Now, from Poitou--Tate duality we have the exact sequence
\begin{equation}\label{eq:PT-compact}
0\rightarrow\check{\mathcal{S}}_{\overline{v}}(\lambda^\st\nr)\rightarrow\check{\mathcal{S}}_{\rm rel}(\lambda^\st\nr)\xrightarrow{{\rm loc}_{v}}\varprojlim_{F\subset K_\infty}\prod_{\eta\vert v}\rH^1(F_\eta,T_{\lambda^\st}(r))\rightarrow\mathcal{X}_v(\lambda\nr)\rightarrow\mathcal{X}_{\rm str}(\lambda\nr)\rightarrow 0,
\end{equation}
which, by the combination of Theorem~\ref{thm:ERL-katz} (with $w=v$ and $\phi=\lambda\mathbf{N}^{-r}$, as above),
Proposition~\ref{prop:Greenberg-nonv}, and Proposition~\ref{prop:Xv-tors}, implies that $\check{\mathcal{S}}_{\overline{v}}(\lambda^\st\nr)$ is $\Lambda_\scO$-torsion; being also $\Lambda_\scO$-torsion-free, it follows that $\check{\mathcal{S}}_{\overline{v}}(\lambda^\st\nr)=0$. Thus from \eqref{eq:PT-compact} we obtain the exact sequence
\begin{align*}
0\rightarrow\check{\mathcal{S}}_{\rm rel}(\lambda^\st\nr)/\mathcal{C}_{\lambda^\st\nr}(K_\infty)&\xrightarrow{{\rm loc}_v}\varprojlim_{F\subset K_\infty}\prod_{\eta\vert v}\rH^1(F_\eta,T_{\lambda^\st}(r))/{\rm loc}_v(\mathcal{C}_{\lambda^\st\nr}(K_\infty))\\
&\quad\rightarrow\mathcal{X}_v(\lambda\nr)\rightarrow\mathcal{X}_{\rm str}(\lambda\nr)\rightarrow 0.
\end{align*}
Since by Theorem~\ref{thm:ERL-katz} the map ${\rm Col}_v$ defines a $\Lambda_\scO^{\rm ur}$-module pseudo-isomorphism
\[
\biggl(\varprojlim_{F\subset K_\infty}\prod_{\eta\vert v}\rH^1(F_\eta,T_{\lambda^\st}(r))/{\rm loc}_v(\mathcal{C}_{\lambda^\st\nr}(K_\infty))\biggr)\hat\otimes_{\scO}\Z_p^{\rm ur}\rightarrow\Lambda_\scO^{\rm ur}/(\mathscr{L}_v(\lambda\nr))
\]
together with \eqref{eq:ac-div} this concludes the proof.
\end{proof}




\begin{proof}[Proof of Theorem~\ref{thm:AH-lambda}] This is the combination of Proposition~\ref{prop:Xv-tors} and Proposition~\ref{prop:char-v}.
\end{proof}

\subsubsection{The case of $\psi^\st\chi$}

\begin{thm}\label{thm:AH}
Let $\psi$ be a Hecke character of infinity type $(1-2r,0)$ for some $r\geq 1$ and let $\chi$ be a finite order character such that $\psi^\st\chi\mathbf{N}^{-r}$ is self-dual with root number 
\[
w(\psi^\st\chi\nr)=+1
\] 
and conductor prime to $p$. Then $\check{\mathcal{S}}_v(\psi^\st\chi\nr)=0$ and $\mathcal{X}_{v}(\psi^\st\chi\nr)$ is $\Lambda_\scO$-torsion, with
\[
{\rm char}_{\Lambda_\scO}\bigl(\mathcal{X}_v(\psi^\st\chi\nr)\bigr)=\bigl(\mathscr{L}_v(\psi^\st\chi\mathbf{N}^{-r})\bigr).
\]
as ideals in $\Lambda_\scO^{\rm ur}$.
\end{thm}

\begin{proof}
We begin by noting that $L(\psi^\st\chi,s)=L(\psi\chi^\st,s)$,  and by our assumption the self-dual character $\psi\chi^\st\nr$ has root number $w(\psi\chi^\st\nr)=+1$. Since $\mathcal{X}_{\BK}(\psi\chi^\st\nr)=\mathcal{X}_{\overline{v}}(\psi\chi^\st\nr)$ by Lemma~\ref{lem:BK=v} (indeed, the infinity type of $\psi\chi^\st\nr$ is $(1-r,r)$), by the same argument as in the proof of \cite[Thm.~2.4.17(1)]{AH-ord} and \cite[Thm.~3.9]{arnold}, but replacing the use of \cite[Prop.~2.3.4]{AH-ord} and \cite[Prop.~2.6]{arnold} by an appeal to Theorem~\ref{thm:ERL-katz} above, 
we deduce that $\check{\mathcal{S}}_{\overline{v}}(\psi^\st\chi\nr)=0$ and $\mathcal{X}_{\overline{v}}(\psi\chi^\st\nr)$ is $\Lambda_\scO$-torsion, with
\[
{\rm char}_{\Lambda_\scO}\bigl(\mathcal{X}_{\overline{v}}(\psi\chi^\st\nr)\bigr)=\bigl(\mathscr{L}_{\overline{v}}(\psi\chi^\st\nr)\bigr)
\]
as ideals in $\Lambda_\scO^{\rm ur}$.

Let $\iota$ denote the involution of $\Lambda_\scO$ given by $\gamma\mapsto\gamma^{-1}$ for $\gamma\in\Gamma$. Since by \cite[Prop.~4.1]{arnold}, and as a direct consequence of the interpolation property of Theorem~\ref{thm:katz} (see e.g. \cite[Lem.~3.3.2(a)]{7-author}), we have the equalities
\begin{align*}
{\rm char}_{\Lambda_\scO}\bigl(\mathcal{X}_{\overline{v}}(\psi\chi^\st\nr)\bigr)&={\rm char}_{\Lambda_\scO}\bigl(\mathcal{X}_{v}(\psi^\st\chi\nr)^\iota\bigr),\\
\bigl(\mathscr{L}_{\overline{v}}(\psi\chi^\st\mathbf{N}^{-r})\bigr)&=\bigl(\mathscr{L}_{v}(\psi^\st\chi\mathbf{N}^{-r})^\iota\bigr),
\end{align*}
the result follows.
\end{proof}


\subsection{$p$-parity conjecture}

By Nekov\'{a}\v{r}'s methods, we can deduce 
a proof of the $p$-parity conjecture for CM forms of higher weight (see \cite{guo-parity} for earlier results in this context). 

\begin{cor}\label{cor:parity}
Let $\lambda$ be a Hecke character of infinity type $(1-2r,0)$ for some $r\geq 1$ with $\varepsilon_\lambda=\eta_K$ and conductor $\cc$ coprime to $p$ and such that $\mathfrak{d}_K\Vert\cc$. Then
\[
{\rm ord}_{s=r}\,L(\lambda,s)\equiv{\rm dim}_\Phi\,{\rm Sel}_{\rm BK}(K,V_\lambda(r))\pmod{2},
\]
and hence the $p$-parity conjecture holds for $\lambda$. 
\end{cor}

\begin{proof}
By an application of Nekov\'{a}\v{r}'s general result  \cite[Cor.~(5.3.2)]{nekovar-parity-III}  (see also \cite{nekovar-parity-III-correction}), it suffices to show that if $w(\lambda\nr)=-1$ (resp. $w(\lambda\nr)=+1$), there exists a finite order character $\phi$ of $\Gamma$ such that $\dim_\Phi\,{\rm Sel}_\BK(K,V_{\lambda\phi}(r))=1$ (resp. ${\rm Sel}_\BK(K,V_{\lambda\phi}(r))=0$).

Suppose first that $w(\lambda\nr)=-1$. 
In view of Proposition~\ref{prop:Greenberg-nonv}, we can take a finite order character $\phi$ of $\Gamma$ such that $\mathscr{L}_v(\lambda\nr)(\phi)\neq 0$. By  Theorem~\ref{thm:AH-lambda} and a variant of Mazur's control theorem for $\mathcal{X}_v(\lambda\phi\nr)$ (see \cite[Prop.~6.2.1]{cas-Kato-Schoen}), it follows that ${\rm Sel}_{v}(K,V_{\lambda\phi}(r))=0$; while by Theorem~\ref{thm:ERL-katz} it follows that the localization map
\begin{equation}\label{eq:locv-V}
{\rm loc}_v:{\rm Sel}_v(K,V_{\lambda^\st\phi^\st}(r))\rightarrow\rH^1(K_v,V_{\lambda^\st\phi^\st}(r))
\end{equation}
is nonzero\footnote{Indeed, letting $\mathbf{c}_{\lambda^\st\mathbf{N}^{-r}}(\phi^\st)$ denote the specialization of $\mathbf{c}_{\lambda^\st\mathbf{N}^{-r}}$ at $\phi^\st$, the inclusion $\mathbf{c}_{\lambda^\st\mathbf{N}^{-r}}(\phi^\st)\in{\rm Sel}_v(K,V_{\lambda^\st\phi^\st}(r))$ follows from Theorem~\ref{thm:ERL-katz}, the interpolation property in Theorem~\ref{thm:katz}, and the fact that $L(\lambda,r)=0$.}. 
Since $\rH^1(K_v,V_{\lambda^\st\phi^\st}(r))$ is $1$-dimensional as a consequence of Tate's local Euler characteristic, 
 from the Poitou--Tate exact sequence
\[
{\rm Sel}_v(K,V_{\lambda^\st\phi^\st}(r))\rightarrow\rH^1(K_v,V_{\lambda^\st\phi^\st}(r))\rightarrow{\rm Sel}_{\rm rel}(K,V_{\lambda\phi}(r))^*\rightarrow{\rm Sel}_{\overline{v}}(K,V_{\lambda\phi}(r))^*\rightarrow 0,
\]
we conclude that 
\begin{equation}\label{eq:BK=rel}
{\rm Sel}_{\rm rel}(K,V_{\lambda\phi}(r))={\rm Sel}_{\overline{v}}(K,V_{\lambda\phi}(r))={\rm Sel}_{\BK}(K,V_{\lambda\phi}(r)),
\end{equation}
using Lemma~\ref{lem:BK=v} for the last equality. Since the nonvanishing of \eqref{eq:locv-V} gives ${\rm loc}_{\overline{v}}({\rm Sel}_{\overline{v}}(K,V_{\lambda\phi}(r)))\neq 0$, from the tautological exact sequence
\[
{\rm Sel}_v(K,V_{\lambda\phi}(r))\rightarrow{\rm Sel}_{\rm rel}(K,V_{\lambda\phi}(r))\rightarrow\rH^1(K_{\overline{v}},V_{\lambda\phi}(r))
\]
and \eqref{eq:BK=rel}, together with the $1$-dimensionality of $\rH^1(K_{\overline{v}},V_{\lambda\phi}(r))$, 
we obtain 
${\rm dim}_\Phi\,{\rm Sel}_\BK(K,V_{\lambda\phi}(r))={\rm dim}_\Phi\,{\rm Sel}_v(K,V_{\lambda\phi}(r))+1
=1$, as desired.

The case $w(\lambda\nr)=+1$ is easier: By \cite[Prop.~2.3]{arnold}, we can take a finite order character $\phi$ of $\Gamma$ such that $\mathscr{L}_{\overline{v}}(\lambda\mathbf{N}^{-r})(\phi)\neq 0$, and then by Theorem~3.9 in \emph{op.\,cit.} it follows that ${\rm Sel}_{\overline{v}}(K,V_{\lambda\phi}(r))=0$, so by Lemma~\ref{lem:BK=v} the Bloch--Kato Selmer group ${\rm Sel}_{\BK}(K,V_{\lambda\phi}(r))$ vanishes as desired. 
\end{proof}

\subsection{Non-triviality of $\mathbf{z}_{g,\chi}$}

For $g$ of weight $k=2$, the non-triviality of $\mathbf{z}_{g,\chi}$ follows from the work of Cornut--Vatsal \cite{CV-docmath,CV-durham}. For $g$ of even weight $k>2$, assuming that 
the residual representation attached to $V_g\vert_{G_K}$ is absolutely \emph{irreducible}, the non-triviality of $\mathbf{z}_{g,\chi}$ follows from the combination of Theorem~3.9 and Theorem~5.7 in \cite{cas-hsieh1}, and also from \cite[Thm.~4.3]{bur-JAG}; 
in both cases, the result is deduced from Hida's methods \cite{hidamu=0,hida-dwork}, showing the non-vanishing of $\mathscr{L}_v(g,\chi)$, and a form of Theorem~\ref{thm:ERL}. 

Here we are interested in the case where $g=\theta_\psi$ has CM by $K$, so in particular $V_g\vert_{G_K}$ is reducible. In the weight $2$ case,  an alternative proof of Cornut--Vatsal's nonvanishing result in this setting  
was given in \cite{bur-disegni} building on the $\Lambda_\scO$-adic Gross--Zagier formula of \cite{disegni-compositio} and the nonvanishing result of \cite{bur-katz}. 
Here we prove the non-triviality of $\mathbf{z}_{g,\chi}$ for $g=\theta_\psi$ of even weight by a different approach.


\begin{thm}\label{thm:Heeg-nonzero}
Let $\psi$ be a Hecke character of infinity type $(1-2r,0)$ for some $r\geq 1$ and conductor a cyclic ideal $\mathfrak{f}_\psi$ of norm prime to $pD_K$. Let 
\[
(g,\chi)=(\theta_\psi,\chi)\in S_{2r}(\Gamma_1(N_g))\times\Sigma_{\rm cc}(c,\mathfrak{N}_g,\varepsilon_g)
\] 
be a self-dual pair for some positive integer $c$ prime to $pN_g$. 
Then: 
\begin{enumerate}
\item[(i)] If $w(\psi\chi\nr)=-1$, then $\mathscr{L}_v(g,\chi)\neq 0$ and ${\rm loc}_{\overline{v}}(\check{\mathcal{S}}_\BK(g,\chi))\neq 0$.
\item[(ii)] If $w(\psi\chi\nr)=+1$, then $\mathscr{L}_{\overline{v}}(g,\chi)\neq 0$ and ${\rm loc}_{v}(\check{\mathcal{S}}_\BK(g,\chi))\neq 0$.
\end{enumerate}
In particular, $\mathbf{z}_{g,\chi}\neq 0$ regardless of the sign of $w(\psi\chi\nr)$.
\end{thm}

\begin{proof}
Since $\mathscr{L}_v(g,\chi)$ is a unit multiple of $\mathscr{L}_v(\psi\chi\mathbf{N}^{-r})\mathscr{L}_v(\psi^\st\chi\mathbf{N}^{-r})$ by Proposition~\ref{prop:factor-BDP}, in the case $w(\psi\chi\nr)=-1$ the nonvanishing of $\mathscr{L}_v(g,\chi)$ follows from Proposition~\ref{prop:Greenberg-nonv} and Theorem~\ref{thm:AH}; the nonvanishing of ${\rm loc}_{\overline{v}}(\check{\mathcal{S}}_\BK(g,\chi))$ then follows from Theorem~\ref{thm:ERL}. 

In the case $w(\psi\chi\nr)=+1$, the same argument as in Proposition~\ref{prop:factor-BDP} shows that $\mathscr{L}_{\overline{v}}(g,\chi)$ is a unit multiple of $\mathscr{L}_{\overline{v}}(\psi\chi\mathbf{N}^{-r})\mathscr{L}_{\overline{v}}(\psi^\st\chi\mathbf{N}^{-r})$. 
Directly from Theorem~\ref{thm:katz} (see \cite[Lem.~3.3.2(a)]{7-author}), we have the relations
\[
\bigr(\mathscr{L}_{\overline{v}}(\psi\chi\mathbf{N}^{-r})\bigr)=\big(\mathscr{L}_{v}(\psi^\st\chi^\st\mathbf{N}^{-r})^\iota\bigr),\quad
\bigr(\mathscr{L}_{\overline{v}}(\psi^\st\chi\mathbf{N}^{-r})\bigr)=\big(\mathscr{L}_{v}(\psi\chi^\st\mathbf{N}^{-r})^\iota\bigr).
\] 
Note that the character $\psi^\st\chi^\st\mathbf{N}^{-r}$ lies inside the range of $p$-adic interpolation for $\mathscr{L}_{v,\cc}$, while $\psi\chi^\st\mathbf{N}^{-r}$ lies outside this range. As in Theorem~\ref{thm:AH} and Proposition~\ref{prop:Greenberg-nonv}, we see that 
\begin{align*}
w(\psi^\st\chi^\st\nr)=+1\quad&\Longrightarrow\quad\mathscr{L}_{v}(\psi^\st\chi^\st\mathbf{N}^{-r})\neq 0,\\
w(\psi\chi^\st\nr)=-1\quad&\Longrightarrow\quad\mathscr{L}_{v}(\psi\chi^\st\mathbf{N}^{-r})\neq 0,
\end{align*}
respectively, which together with Theorem~\ref{thm:ERL} yields the result.
\end{proof}

\begin{rem}\label{rem:v-vs-vbar}
It is interesting to note that the preceding results show in fact the equivalences
\begin{align*}
{\rm loc}_{\overline{v}}(\check{\mathcal{S}}_\BK(g,\chi))\neq 0\quad\Longleftrightarrow\quad\mathscr{L}_v(g,\chi)\neq 0&\quad\Longleftrightarrow\quad  w(\psi\chi\nr)=-1,\\
{\rm loc}_{v}(\check{\mathcal{S}}_\BK(g,\chi))\neq 0\quad \Longleftrightarrow\quad\mathscr{L}_{\overline{v}}(g,\chi)\neq 0&\quad\Longleftrightarrow\quad w(\psi\chi\nr)=+1.
\end{align*}
Indeed, the right pair of equivalences and the backward direction for the left pair directly follow from the proof of Theorem~\ref{thm:Heeg-nonzero}.  For the left pair of implications $\Longrightarrow$, note that by \cite[Thm.~2.4.17]{AH-ord} (as extended in \cite[Thm.~3.9]{arnold}) we have the implications
\begin{align*}
w(\psi\chi\nr)=+1&\quad\Longrightarrow\quad\check{\mathcal{S}}_\BK(\psi\chi\nr)=0,\\\quad
w(\psi\chi\nr)=-1&\quad\Longrightarrow\quad\check{\mathcal{S}}_\BK(\psi^\st\chi\nr)=0,
\end{align*}
using $w(\psi^\st\chi\nr)=-w(\psi\chi\nr)$ for the second implication. Since by Proposition~\ref{prop:dec-Sel} we have 
\begin{equation}\label{eq:loc-dec}
{\rm loc}_{\overline{v}}(\check{\mathcal{S}}_\BK(\psi\chi\nr))={\rm loc}_{\overline{v}}(\check{\mathcal{S}}_\BK(g,\chi)),\quad
{\rm loc}_{v}(\check{\mathcal{S}}_\BK(\psi^\st\chi\nr))={\rm loc}_{v}(\check{\mathcal{S}}_\BK(g,\chi)),\nonumber
\end{equation}
this yields the claim.
\end{rem}

\subsection{Proof of Theorem~\ref{thm:BDP-IMC} and Heegner cycle Main Conjecture}\label{subsec:IMC-proof-end}

\begin{proof}[Proof of Theorem~\ref{thm:BDP-IMC}]
Since $(g,\chi)=(\theta_\psi,\chi)$ is a self-dual pair, the $L$-function $L(g/K,\chi,s)$ is self-dual with sign $-1$ and center at $s=r$,  
and the decomposition \eqref{eq:dec-V} 
gives 
\[
L(g/K,\chi,s)=L(\psi\chi,s)\cdot L(\psi^\st\chi,s).
\]
The $L$-functions in the right-hand side of this factorization are self-dual with opposite signs, 
\[
w(\psi\chi\nr)=-w(\psi^\st\chi\nr).
\] 
Suppose first that $w(\psi\chi\nr)=-1$. By Proposition~\ref{prop:dec-Sel}, 
the assertions that $\check{\mathcal{S}}_v(g,\chi)$ vanishes and $\mathcal{X}_{v}(g,\chi)$ is $\Lambda_\scO$-torsion follow from Theorem~\ref{thm:AH-lambda} and Theorem~\ref{thm:AH}. Together with the factorization in Proposition~\ref{prop:factor-BDP}, the same theorems give the equalities 
\begin{equation}\label{eq:put-together}
\begin{aligned}
{\rm char}_{\Lambda_\scO}\bigl(\mathcal{X}_{v}(g,\chi)\bigr)&=
{\rm char}_{\Lambda_\scO}\bigl(\mathcal{X}_{v}(\psi\chi\nr)\bigr)\cdot
{\rm char}_{\Lambda_\scO}\bigl(\mathcal{X}_{v}(\psi^\st\chi\nr)\bigr)\\
&=\bigl(\mathscr{L}_v(\psi\chi\mathbf{N}^{-r})\cdot\mathscr{L}_v(\psi^\st\chi\mathbf{N}^{-r})\bigr)\\
&=\bigl(\mathscr{L}_v(g,\chi)^2\bigr)
\end{aligned}
\end{equation}
in $\Lambda_\scO^{\rm ur}$, yielding the result in this case.
%

In the case $w(\psi\chi\nr)=+1$, we apply the result of Theorem~\ref{thm:AH-lambda} and Theorem~\ref{thm:AH} to $\psi\chi^\st\nr$ (which has $w(\psi\chi^\st\nr)=-1$) and $\psi^\st\chi^\st\nr$ (which has $w(\psi^\st\chi^\st\nr)=+1$), respectively, to obtain 
\[
\check{\mathcal{S}}_v(\psi\chi^\st\nr)=\check{\mathcal{S}}_v(\psi^\st\chi^\st\nr)=0
\] 
and that $\mathcal{X}_v(\psi\chi^\st\nr)$ and $\mathcal{X}_v(\psi^\st\chi^\st\nr)$ are both $\Lambda_\scO$-torsion, with
\[
{\rm char}_{\Lambda_\scO}\bigl(\mathcal{X}_v(\psi\chi^\st\nr)\bigr)=\bigl(\mathscr{L}_v(\psi\chi^\st\nr)\bigr),\quad
{\rm char}_{\Lambda_\scO}\bigl(\mathcal{X}_v(\psi^\st\chi^\st\nr)\bigr)=\bigl(\mathscr{L}_v(\psi^\st\chi^\st\nr)\bigr)
\]
as ideals in $\Lambda_\scO^{\rm ur}$. 
As in the proof of Theorem~\ref{thm:AH}, by the action of complex conjugation this gives $\check{\mathcal{S}}_v(\psi^\st\chi\nr)=\check{\mathcal{S}}_v(\psi\chi\nr)=0$ and that $\mathcal{X}_{\overline{v}}(\psi^\st\chi\nr)$ and $\mathcal{X}_{\overline{v}}(\psi\chi\nr)$ are both $\Lambda_\scO$-torsion, with
\[
{\rm char}_{\Lambda_\scO}\bigl(\mathcal{X}_{\overline{v}}(\psi^\st\chi\nr)\bigr)=\bigl(\mathscr{L}_{\overline{v}}(\psi^\st\chi\nr)\bigr),\quad
{\rm char}_{\Lambda_\scO}\bigl(\mathcal{X}_{\overline{v}}(\psi\chi\nr)\bigr)=\bigl(\mathscr{L}_{\overline{v}}(\psi\chi\nr)\bigr),
\]
which by the $\overline{v}$-versions of Proposition~\ref{prop:factor-BDP} and Proposition~\ref{prop:dec-Sel} yields the result as above.
\end{proof}

By the nonvanishing results established in the course of proving Theorem~\ref{thm:BDP-IMC}, we can deduce the following result on Conjecture~\ref{conj:HPMC}.

\begin{cor}\label{cor:HP-IMC}
Let $(g,\chi)$ be a self-dual pair as in Theorem~\ref{thm:BDP-IMC}. Then $\check{\mathcal{S}}_{\BK}(g,\chi)$ and $\mathcal{X}_{\BK}(g,\chi)$ both have $\Lambda_\scO$-rank one, and
\[
{\rm char}_{\Lambda_\scO}\bigl(\mathcal{X}_{\BK}(g,\chi)_{\rm tors}\bigr)={\rm char}_{\Lambda_\scO}\bigl(\check{\mathcal{S}}_{\BK}(g,\chi)/\Lambda_\scO\cdot\mathbf{z}_{g,\chi}\bigr)^2.
\]
In other words, Conjecture~\ref{conj:HPMC} holds for $(g,\chi)$.
\end{cor}

\begin{proof}
By Theorem~\ref{thm:Heeg-nonzero}, $\mathbf{z}_{g,\chi}$ is not $\Lambda_\scO$-torsion, and ${\rm loc}_{\overline{v}}(\check{\mathcal{S}}_\BK(g,\chi))\neq 0$ (resp. ${\rm loc}_{v}(\check{\mathcal{S}}_\BK(g,\chi))\neq 0$) when $w(\psi\chi\nr)=-1$ (resp. $w(\psi\chi\nr)=+1$). Hence by Proposition~\ref{prop:equiv} (see also Remark~\ref{rem:sign-psichi}) the result follows from Theorem~\ref{thm:BDP-IMC}.
\end{proof}

%

\section{Formula for the Bloch--Kato $\sha_{\BK}(W_{g,\chi}/K)$}\label{sec:TNC}

Let $K$ be an imaginary quadratic field satisfying 
\eqref{eq:spl} for a prime $p>3$. Let $\lambda$ be a Hecke character of $K$ of infinity type $(1-2r,0)$ for some $r\geq 1$ and central character 
\begin{equation}\label{eq:sd}
\varepsilon_\lambda=\eta_K.\tag{sd}
\end{equation} 
In particular, $\lambda\mathbf{N}^{-r}$ is self-dual in the sense of Definition~\ref{def:sd-char}, and the conductor $\cc$ of $\lambda$ is divisible by $\mathfrak{d}_K:=(\sqrt{-D_K})$. Throughout this section we assume that 
\begin{equation}\label{eq:div}
p\nmid\cc\quad\textrm{and}\quad\mathfrak{d}_K\Vert\cc.\tag{div}
\end{equation}
By self-duality, $\cc$ is invariant under complex conjugation, and so by \eqref{eq:div} can write 
\begin{equation}\label{eq:cond-sd}
\cc=(c)\mathfrak{d}_K
\end{equation}
for a unique positive integer $c$ prime to $pD_K$.

\subsection{Strong good pairs}

The following definition is a natural higher weight extension (and strengthening) of the notion of ``good pair'' for $\lambda$ introduced in \cite{BDP-PJM}. 

\begin{defi}\label{def:good}
Suppose the self-dual character $\lambda\nr$ has root number $w(\lambda\nr)=-1$. 
We say that a pair of Hecke characters $(\psi,\chi)$ is a \emph{good pair for $\lambda$} if it satisfies
\begin{enumerate}
\item[(S1)] $\psi$ has infinity type $(1-2r,0)$ and conductor a cyclic ideal $\mathfrak{f}_\psi$ of norm prime to $pD_K$. 
\item[(S2)] $\chi$ is a finite order character such that 
\[
(g,\chi)=(\theta_\psi,\chi)\in S_{2r}(\Gamma_1(N_g))\times\Sigma_{\rm cc}(c,\mathfrak{N}_g,\varepsilon_g),
\] 
where $\mathfrak{N}_g:=\mathfrak{d}_K\cdot\mathfrak{f}_\psi$ and $N_g=D_K\mathbf{N}(\mathfrak{f}_\psi)$.
\item[(S3)] $\lambda=\psi\chi$.
\item[(S4)] $L(\psi^\st\chi,r)\neq 0$.
\end{enumerate}
If $r=1$, we say that $(\psi,\chi)$ is a \emph{strong good pair for $\lambda$} if in addition
\begin{enumerate}
\item[(S5)] $\psi^-:=\psi^\st/\psi$ has order prime to $p$.
\item[(S6)] $\psi^-\vert_{G_v}\neq\mathds{1}$, where $G_v\subset G_K$ is a decomposition group at $v$.
\item[(S7)] $\psi^-$ has order $\geq 3$.
\item[(S8)] The value
\[
L^{\rm alg}(\psi^-,1):=
\frac{L(\psi^-,1)}{\Omega^2}
\]
is a $p$-adic unit.
\end{enumerate}
\end{defi}

For $r=1$, the existence of good pairs for $\lambda$ is shown in \cite[Prop.~3.28]{BDP-PJM} building on Greenberg \cite{greenberg-critical} and Rohrlich \cite{rohrlich-ac} nonvanishing results. Adapting their argument, and building on Hida's mod~$p$ non-vanishing results \cite{hida-durham,hida-ICTS}, we can prove the following.


\begin{lemma}\label{lem:good}
Suppose $\lambda\mathbf{N}^{-r}$ has root number $-1$. Then there exist good pairs for $\lambda$. Moreover, if $r=1$, then there exist strong good pairs  for $\lambda$.
%
\end{lemma}

\begin{proof}
Take $\psi$ a Hecke character of infinity type $(1-2r,0)$ and conductor a cyclic ideal $\mathfrak{f}_\psi$ of norm coprime to $p\cc$, and put 
\[
\chi:=\psi^{-1}\lambda.
\]
By \eqref{eq:cond-sd}, $\psi$ satisfies (S1), and by construction, the pair $(\psi,\chi)$ satisfies (S3). By \cite[Rem.~3.20]{BDP-PJM}, it follows that $\chi$ satisfies (S2), so it remains to verify the much subtler condition (S4). 

We begin by noting 
that the functional equation for $L(\psi^\st\chi,s)$ relating its values at $s$ and $2r-s$ has sign $w(\psi^\st\chi\nr)=+1$.  Indeed, by the inclusion $\chi\in\Sigma_{\rm cc}(c,\mathfrak{N}_g,\varepsilon_g)$, the sign in the functional equation for $L(g/K,\chi,s)=L(\theta_\psi/K,\chi,s)$ 
is $\epsilon(g,\chi)=-1$, and so the claim follow from the factorization
\begin{equation}\label{eq:factor-good}
\begin{aligned}
L(\theta_\psi/K,\chi,s)&=L(\lambda,s)\cdot L(\psi^\st\chi,s)
\end{aligned}
\end{equation}
and the assumption that $w(\lambda\mathbf{N}^{-r})=-1$. Let $\ell=\lambda\overline{\lambda}$ be a prime split in $K$, with $\ell\neq p$ and prime to  the conductors of $\psi$ and $\chi$. 
Given a pair $(\psi,\chi)$ satisfying conditions (S1)--(S3), 
for any finite order character $\alpha$ of conductor dividing $\lambda^\infty$, the pair $(\psi',\chi'):=(\psi\alpha,\chi\alpha^{-1})$  satisfies the same conditions (with $\mathfrak{f}_\psi$ replaced by $\mathfrak{f}_\psi\lambda^m$ for some $m\geq 0$), while \eqref{eq:factor-good} becomes
\[
L(\theta_{\psi'}/K,\chi',s)=L(\lambda,s)\cdot L(\psi^\st\chi\alpha^-,s),
\]  
where $\alpha^-:=\alpha^\st/\alpha$.  
Since $\alpha^-$ is anticyclotomic of $\ell$-power conductor, by \cite[Thm.~1.1]{finis-modp} all but finitely many of the values $L(\psi^\st\chi\alpha^-,r)=L(\psi^\st\chi\alpha^-\mathbf{N}^{-r},0)$ are nonzero\footnote{Reversing the roles of $\ell$ and $p$. Note that $\psi^\st\chi\alpha^-\mathbf{N}^{-r}$ has infinity type $(r,1-r)$, and our conventions are \emph{opposite} to those in \cite{finis-modp}.} (even nonzero modulo $\varpi$) as $\alpha$ varies, whence the first assertion. Moreover, when $r=1$ noting that 
\[
L((\psi')^-,1)=L(\psi^-\alpha^-,1),
\]
by \cite[Thm.~4.3]{hida-durham} (see also \cite{hida-ICTS}) we can find $\ell$ and $\alpha$ as above so that 
further $L^{\rm alg}((\psi')^-,1)$ is a $p$-adic unit and conditions (S5)--(S7) hold for $(\psi')^-$, whence the result.
\end{proof}

\begin{rem}
In the proof of Lemma~\ref{lem:good}, the use of the main result of \cite{finis-modp} could alternatively be replaced by an appeal to Hsieh's mod $p$ nonvanishing result \cite[Thm.~A]{hsieh-AJM} or the nonvanishing results of Greenberg and Rohrlich \cite{greenberg-critical,rohrlich-ac}. 
\end{rem}


\subsection{Proof of Theorem~\ref{thmintro:sha}}\label{subsec:proof-thmC}

We begin by noting that by the assumption that $\lambda\mathbf{N}^{-r}$ has root number $-1$, Lemma~\ref{lem:good} ensures the existence of good pairs  for $\lambda$. We fix once and for all a good pair $(\psi,\chi)$ for $\lambda$, and put $g=\theta_\psi$. 
%
%
Exploiting the decomposition 
\[
\mathcal{X}_v(g,\chi)\simeq\mathcal{X}_v(\lambda\nr)\oplus\mathcal{X}_v(\psi^\st\chi\nr)
\] 
from Proposition~\ref{prop:dec-Sel}, the formula for $\#\sha_\BK(W_{g,\chi}/K)$ of Theorem~\ref{thmintro:sha} will be obtained by computing the $\Gamma$-Euler characteristic for the constituent Selmer groups in this decomposition and combining the resulting formulae. Note that some of the intermediate results will be obtained in greater generality than needed for the proof of Theorem~\ref{thmintro:sha}.

For a finite extension $\Phi$ of $\Q_p$ with ring of integers $\scO$, and a $\Phi$-linear $G_K$-representation $V$ with a fixed $G_K$-stable $\scO$-lattice $T\subset V$, recall that the Bloch--Kato Tate--Shafarevich group $\sha_\BK(W/K)$, where $W=V/T$, is defined by
\[
\sha_\BK(W/K):={\rm Sel}_\BK(K,W)/{\rm Sel}_\BK(K,W)_{\rm div},
\]
where ${\rm Sel}_\BK(K,W)_{\rm div}$ denotes the maximal divisible submodule of ${\rm Sel}_\BK(K,W)$; and for a finite prime $w\nmid p$ of $K$, the Tamagawa number of $W$ at $w$ is defined by
\[
c_w(W/K):=\#\bigl(\rH^1_{\rm ur}(K_w,W)/\rH^1_f(K_w,W)\bigr),
\]
where $\rH^1_{\rm ur}(K_w,W):={\rm ker}\{\rH^1(K_w,W)\rightarrow\rH^1(I_w,W)\}$ is the unramified submodule, and $\rH^1_f(K_w,W)$ is the image of the unramified subspace $\rH^1_{\rm ur}(K_w,V)$ under the natural map $\rH^1(K_w,V)\rightarrow\rH^1(K_w,W)$ (which is thus contained in $\rH^1_{\rm ur}(K_w,W)$).

\subsubsection{Formula for $\psi^\st\chi$}\hfill 


\begin{prop}\label{prop:control-BK}
Let $(\psi,\chi)$ be as in Theorem~\ref{thm:AH}, and let $\mathscr{F}_v(\psi^\st\chi\nr)\in\Lambda_\scO$ be a generator of the characteristic ideal of $\mathcal{X}_v(\psi^\st\chi\nr)$. If ${\rm Sel}_\BK(K,W_{\psi^\st\chi}(r))$ is finite, then $\mathscr{F}_v(\psi^\st\chi\nr)(0)\neq 0$ with
\[
\mathscr{F}_v(\psi^\st\chi\nr)(0)\;\sim_p\;\#\rH^0(K_{\overline{v}},W_{\psi^\st\chi}(r))^2\cdot
\frac{\#\sha_{\BK}(W_{\psi^\st\chi}(r)/K)}{\#\rH^0(K,W_{\psi^\st\chi}(r))^2}
\cdot\prod_{\substack{w\in\Sigma,w\nmid p}}c_w(W_{\psi^\st\chi}(r)/K).
\]
\end{prop}

\begin{proof}
By Lemma~\ref{lem:BK=v}, 
$\mathcal{X}_v(\psi^\st\chi\nr)$ interpolates 
Selmer groups ${\rm Sel}_\BK(F,W_{\psi^\st\chi}(r))$  as $F$ varies over the finite extensions of $K$ contained in $K_\infty$. Therefore, by the control theorem of  \cite[Prop.~4.3]{arnold} the natural restriction map
\[
{\rm Sel}_{\BK}(K,W_{\psi^\st\chi}(r))={\rm Sel}_v(K,W_{\psi^\st\chi}(r))\rightarrow\mathcal{S}_v(\psi^\st\chi\nr)^\Gamma
\]
has finite kernel and cokernel, and so the finiteness of ${\rm Sel}_{\BK}(K,W_{\psi^\st\chi}(r))$ gives $\#\mathcal{S}_v(\psi^\st\chi\nr)^\Gamma<\infty$.  
Letting $\mathscr{F}_{v}(\psi^\st\chi\nr)\in\Lambda_\scO$ be any  generator of ${\rm char}_{\Lambda_\scO}(\mathcal{X}_v(\psi^\st\chi\nr))$, 
it is then easily checked that the $\Gamma$-coinvariants $\mathcal{S}_v(\psi^\st\chi\nr)_\Gamma$ are also finite, that $\mathscr{F}_v(\psi^\st\chi\nr)(0)\neq 0$, and that
\begin{equation}\label{eq:val-0}
\mathscr{F}_v(\psi^\st\chi\nr)(0)\,\sim_p\,\frac{\#\mathcal{S}_v(\psi^\st\chi\nr)^\Gamma}{\#\mathcal{S}_v(\psi^\st\chi\nr)_\Gamma}
\end{equation}
(see \cite[Lem.~4.2]{greenberg-cetraro} for example). For the proof of the formula in the Proposition, we shall adapt the methods of \cite[\S{4}]{greenberg-cetraro} and \cite[\S{3}]{jsw}. 

First we set up some notations. Fix $\Sigma$ a finite set of primes of $K$ containing the archimedean prime, the primes above $p$, the primes dividing the conductor of $\psi^\st$ or $\chi$ and their complex conjugates, and let $K^\Sigma$ denote the maximal extension of $K$ unramified outside $\Sigma$. 
Put
\[
\mathcal{P}_{}^\Sigma(K):=\frac{\rH^1(K_v,W_{\psi^\st\chi}(r))}{\rH^1(K_v,W_{\psi^\st\chi}(r))_{\rm div}}\times\rH^1(K_{\overline{v}},W_{\psi^\st\chi}(r))\times\prod_{w\in\Sigma,w\nmid p}\rH^1(K_w,W_{\psi^\st\chi}(r)).
\]
Similarly, consider the $\Lambda_\scO$-module $\mathbf{W}_{\psi^\st\chi}(r):=W_{\psi^\st\chi}(r)\otimes_{\scO}\Lambda_\scO^\vee$ with diagonal $G_K$-action, letting $G_K$ act on the second factor through the inverse of the tautological character $G_K\twoheadrightarrow\Gamma\hookrightarrow\Lambda_\scO^\times$, and put 
\[
\mathcal{P}^\Sigma(K_\infty):=\{0\}\times\rH^1(K_{\overline{v}},\mathbf{W}_{\psi^\st\chi}(r))\times\prod_{w\in\Sigma,w\nmid p}\rH^1(K_{w},\mathbf{W}_{\psi^\st\chi}(r)),
\]
thinking of $\{0\}$ as a subspace of $\rH^1(K_{v},\mathbf{W}_{\psi^\st\chi}(r))$. Then from the definitions and Shapiro's lemma we have  exact sequences
\begin{equation}\label{eq:glob-loc-K}
\begin{aligned}
0&\rightarrow{\rm Sel}_\BK(K,W_{\psi^\st\chi}(r))\rightarrow\rH^1(K^\Sigma/K,W_{\psi^\st\chi}(r))\rightarrow\mathcal{P}^\Sigma(K),
\end{aligned}
\end{equation}
\begin{equation}\label{eq:glob-loc-Kinfty}
\begin{aligned}
0&\rightarrow\mathcal{S}_v(\psi^\st\chi\nr)\rightarrow\rH^1(K^\Sigma/K,\mathbf{W}_{\psi^\st\chi}(r))\rightarrow\mathcal{P}^\Sigma(K_\infty).
\end{aligned}
\end{equation}
Denote by $\mathcal{G}^\Sigma(K)\subset\mathcal{P}^\Sigma(K)$ and $\mathcal{G}^\Sigma(K_\infty)\subset\mathcal{P}^\Sigma(K_\infty)$ the image of the right maps in the above exact sequence, whose $\Gamma$-invariants then fit into the commutative diagram
\begin{equation}\label{eq:comm-diag}
\begin{aligned}
\xymatrix{
0\ar[r]&{\rm Sel}_{\BK}(K,W_{\psi^\st\chi}(r))\ar[r]\ar[d]^{r^*}&\rH^1(K^\Sigma/K,W_{\psi^\st\chi}(r))\ar[r]\ar[d]^{s^*}&\mathcal{G}^\Sigma(K)\ar[r]\ar[d]^{t^*}&0\\
0\ar[r]&\mathcal{S}_v(\psi^\st\chi\nr)^\Gamma\ar[r]&\rH^1(K^\Sigma/K,\mathbf{W}_{\psi^\st\chi}(r))^\Gamma\ar[r]&\mathcal{G}^\Sigma(K_\infty)^\Gamma.&
}
\end{aligned}
\end{equation}
The surjectivity of $s^*$ follows from the fact that $\Gamma$ has $p$-cohomological dimension $1$, while the kernel of $s^*$ is given by 
\[
\rH^1(\Gamma,\rH^0(K,\mathbf{W}_{\psi^\st\chi}(r)))=\rH^1(\Gamma,\rH^0(K_\infty,W_{\psi^\st\chi}(r))),
\] 
whose order is the same as $\#\rH^0(K,W_{\psi^\st\chi}(r))$ (using that $\#\rH^0(K_\infty,W_{\psi^\st\chi}(r))<\infty$, as follows from \cite[Lem.~2.5]{arnold}). Hence from the Snake Lemma applied to \eqref{eq:comm-diag} we immediately get the relation
\begin{equation}\label{eq:gamma-inv}
\#\mathcal{S}_v(\psi^\st\chi\nr)^\Gamma=\#{\rm Sel}_{\rm BK}(K,W_{\psi^\st\chi}(r))\cdot\frac{\#{\rm ker}(t^*)}{\#\rH^0(K,W_{\psi^\st\chi}(r))}.
\end{equation}
To compute the order of ${\rm ker}(t^*)$, we first compute the order of the kernel of the map
\[
\tau^*=(\tau_w^*)_{w\in\Sigma,w\nmid p}:\mathcal{P}^\Sigma(K)\rightarrow\mathcal{P}^\Sigma(K_\infty)^\Gamma.
\]
For $w=v$ we find
\begin{equation}\label{eq:tau-v}
\begin{aligned}
{\rm ker}(\tau_v^*)=\frac{\rH^1(K_v,W_{\psi^\st\chi}(r))}{\rH^1(K_v,W_{\psi^\st\chi}(r))_{\rm div}}
&\simeq\rH^1(K_v,T_{\psi\chi^\st}(r))_{\rm tors}\\
&={\rm ker}\bigl(\rH^1(K_v,T_{\psi\chi^\st}(r))\rightarrow\rH^1(K_v,V_{\psi\chi^\st}(r))\bigr)\\
&={\rm coker}\bigl(\rH^0(K_v,V_{\psi\chi^\st}(r))\rightarrow\rH^0(K_v,W_{\psi\chi^\st}(r))\bigr)\\
&=\rH^0(K_v,W_{\psi\chi^\st}(r)),
\end{aligned}
\end{equation}
using local Tate duality, the cohomology long exact sequence associated to $0\rightarrow T_{\psi\chi^\st}(r)\rightarrow V_{\psi\chi^\st}(r)\rightarrow W_{\psi\chi^\st}(r)\rightarrow 0$, and the vanishing of $\rH^0(K_v,V_{\psi\chi^\st}(r))$ (as follows from \cite[Lem.~2.5]{arnold}). Similarly, we find
\begin{equation}\label{eq:tau-vbar}
\#{\rm ker}(\tau_{\overline{v}}^*)=\#\rH^1(\Gamma,\rH^0(K_{\overline{v}},\mathbf{W}_{\psi^\st\chi}(r)))=\#\rH^0(K_{\overline{v}},W_{\psi^\st\chi}(r)),
\end{equation}
using the finiteness of $\rH^0(K_{\overline{v}},\mathbf{W}_{\psi^\st\chi}(r))$ (a consequence of \cite[Lem.~2.5]{arnold}) for the last equality. 
On the other hand, consider a prime $w\in\Sigma$ with $w\nmid p$. If $W_{\psi^\st\chi}$ is unramified at $w$, then ${\rm ker}(\tau_w^*)=0$ by the same argument as in \cite[p.\,389]{jsw}; if $W_{\psi^\st\chi}$ is ramified at $w$, then the group $\rH^0(K_w,\mathbf{W}_{\psi^\st\chi}(r))$ is finite (see \cite[Lem.~3.4]{arnold}), and as in \eqref{eq:tau-vbar} we find $\#{\rm ker}(\tau_{w}^*)=\#\rH^0(K_{w},W_{\psi^\st\chi}(r))$. Together with \cite[Lem.~5.10]{BK} and the discussion in [\emph{op.\,cit.}, p.\,373], this shows
\begin{equation}\label{eq:tau-notp}
\#{\rm ker}(\tau_w^*)=c_w(W_{\psi^\st\chi}(r)/K)
\end{equation}
for all $w\in\Sigma$ with $w\nmid p$.

Next, the relation between $\#{\rm ker}(t^*)$ and $\#{\rm ker}(\tau^*)$ can be found similarly as in \cite[Lem.~4.7]{greenberg-cetraro}. Indeed, by 
\cite[Prop.~A.2]{pollack-weston} (extending and generalizing \cite[Prop.~2.1]{greenvats} to the anticyclotomic setting), the $\Lambda_\scO$-torsionness of $\mathcal{X}_v(\psi^\st\chi\nr)$  (as shown in Theorem~\ref{thm:AH}) implies 
surjectivity of the right arrow in \eqref{eq:glob-loc-Kinfty}, i.e. 
$\mathcal{G}^\Sigma(K_\infty)=\mathcal{P}^\Sigma(K_\infty)$. On the other hand, from Poitou--Tate duality we have the exact sequence
\[
0\rightarrow{\rm Sel}_\BK(K,W_{\psi^\st\chi}(r))\rightarrow\rH^1(K,W_{\psi^\st\chi}(r))\rightarrow\mathcal{P}^\Sigma(K)\rightarrow{\rm Sel}_{\overline{v}}(K,T_{\psi\chi^\st}(r))^\vee\rightarrow\rH^2(K,W_{\psi^\st\chi}(r))^\vee\rightarrow 0,
\]
using that ${\rm Hom}_{}(W_{\psi^\st\chi}(r),\mu_{p^\infty})\simeq T_{\psi\chi^\st}(r)$ by the self-duality of $\psi^\st\chi\mathbf{N}^{-r}$. Since $\#{\rm Sel}_\BK(K,W_{\psi^\st\chi}(r))<\infty$ by assumption, the generalization of Cassels' theorem in \cite[Prop.~4.13]{greenberg-cetraro} gives\footnote{Taking $M=W_{\psi^\st\chi}(r)$ 
in the notations of \emph{loc.\,cit.}. Note also that the hypothesis in \emph{loc.\,cit.} that $\rH^0(K_w,M)$ is finite for some finite $w\in\Sigma$ holds in our case taking $w=v$ thanks to \cite[Lem.~2.5]{arnold}.} 
\[
\mathcal{P}^\Sigma(K)/\mathcal{G}^\Sigma(K)\simeq\rH^0(K,W_{\psi\chi^\st}(r))^\vee,
\]
and with this isomorphism, the argument in \cite[Prop.~4.7]{greenberg-cetraro} gives
\[
\#{\rm ker}(t^*)=\#{\rm ker}(\tau^*)\cdot\frac{\#\mathcal{S}_v(\psi^\st\chi\nr)_\Gamma}{\#\rH^0(K,W_{\psi\chi^\st}(r))}.
\]
Substituting \eqref{eq:tau-v}, \eqref{eq:tau-vbar}, and \eqref{eq:tau-notp} 
into this, and the resulting formula for $\#{\rm ker}(t^*)$ into \eqref{eq:gamma-inv}, we obtain
\begin{align*}
\#\mathcal{S}_v(\psi^\st\chi\nr)^\Gamma&=\#\sha_{\BK}(K,W_{\psi^\st\chi}(r))\cdot
\#\rH^0(K_{\overline{v}},W_{\psi^\st\chi}(r))^2\\
&\quad\times
\frac{\#\mathcal{S}_v(\psi^\st\chi\nr)_\Gamma}{\#\rH^0(K,W_{\psi^\st\chi}(r))^2}
\cdot\prod_{\substack{w\in\Sigma,w\nmid p}}c_w(W_{\psi^\st\chi}(r)/K),
\end{align*}
using $\rH^0(K_v,W_{\psi\chi^\st}(r))\simeq\rH^0(K_{\overline{v}},W_{\psi^\st\chi}(r))$ and $\rH^0(K,W_{\psi\chi^\st}(r))\simeq\rH^0(K,W_{\psi^\st\chi}(r))$ by the action of complex conjugation and the equality $\sha_\BK(K,W_{\psi^\st\chi}(r))={\rm Sel}_\BK(K,W_{\psi^\st\chi}(r))$ that follows from the finiteness of the latter. Together with \eqref{eq:val-0} this concludes the proof.
\end{proof}

\subsubsection{Interlude: $p$-part of TNC in rank $0$}

Combined with Theorem~\ref{thm:AH}, the preceding proposition yields a result on the Tamagawa number conjecture of Bloch--Kato \cite{BK} that will play an important role in the proof of Theorem~\ref{thmintro:p-BSD-A}.

Recall that a Hecke character $\xi:K\backslash\mathbb{A}_K^\times\rightarrow\C^\times$ of infinity type $(a,b)\in\Z^2$ in the sense of \S\ref{subsec:Lp-pairs} can alternatively be viewed as a continuous group  homomorphism 
\[
\xi:\hat K^\times/U\rightarrow\C^\times
\] 
trivial on an open subgroup $U$ of $\hat K^\times$ (that can be described explicitly in terms of the conductor of $\xi$) satisfying $\xi(\alpha)=\alpha^{-a}(\alpha^\st)^{-b}$ for all $\alpha\in K^\times$. Under our fixed embedding $\imath_\infty$, 
$\xi$ then takes values in a number field $L$, and for every embedding $\sigma\in{\rm Hom}(L,\C)$ we obtain the character 
\[
\xi_\sigma:\hat K^\times/U\xrightarrow{\xi} L^\times\xrightarrow{\sigma}\C^\times
\] 
by composition. Combining their associated Hecke $L$-functions, we can consider the $L_\C:=L\otimes_{\Q}\C$-valued $L$-function 
\[
(L(\xi_\sigma,s))_{\sigma:L\hookrightarrow\C},
\]
with $L(\xi,s)$ corresponding to $\sigma:L\subset\overline{\Q}\xrightarrow{\imath_\infty}\C$.

\begin{thm}\label{thm:pTNC-rk0}
Let $(\psi,\chi)$ be as in Theorem~\ref{thm:AH}, with values in a number field $L$, and suppose that $L(\psi^\st\chi,r)\neq 0$. Then 
\[
\biggl(\frac{L((\psi^\st\chi)_\sigma,r)}{\Omega_\sigma}\biggr)\in L^\times\subset L_\C^\times,
\]
and for all primes $\PP$ of $L$ above $p$ we have
\[
{\rm ord}_{\PP}\biggl(\frac{L(\psi^\st\chi,r)}{\Omega}\biggr)=
{\rm ord}_{\PP}
\biggl(\frac{\#\sha_{\BK}(W_{\psi^\st\chi}(r)/K)}{\#\rH^0(K,W_{\psi^\st\chi}(r))^2}\biggr)
+\sum_{\substack{w\in\Sigma,w\nmid p}}{\rm ord}_{\PP}(c_w(W_{\psi^\st\chi}(r)/K)).
\]
In other words, the $p$-part of the Tamagawa number conjecture holds for $\psi^\st\chi$.
\end{thm}

\begin{proof}
The first assertion is proved in \cite{goldstein-shappacher} and (more generally) in \cite{blasius} (cf. \cite[Thm.~2.12]{BDP-PJM}), and the second is immediate from 
Theorem~\ref{thm:katz}, Theorem~\ref{thm:AH}, and  Proposition~\ref{prop:control-BK}.
\end{proof}

\subsubsection{Formula for $\lambda$} 

\begin{prop}\label{prop:control-opp-BK}
Let $\lambda$ be as in Theorem~\ref{thm:AH-lambda}. If  
$\mathscr{L}_v(\lambda\mathbf{N}^{-r})(0)\neq 0$, 
then the map
\begin{equation}\label{eq:loc-psichi}
{\rm loc}_{\overline{v}}:{\rm Sel}_{\BK}(K,T_{\lambda}(r))\rightarrow
\rH^1_f(K_{\overline{v}},T_{\lambda}(r))
\end{equation}
is nonzero, and letting $\mathscr{F}_v(\lambda\nr)\in\Lambda_\scO$ be any generator of the characteristic ideal of $\mathcal{X}_v(\lambda\nr)$ we have
\[
\mathscr{F}_v(\lambda\nr)(0)\;\sim_p\;\#\rH^0(K_{\overline{v}},W_{\lambda}(r))^2\cdot
\frac{\#\sha_{\BK}(W_{\lambda}(r)/K)}{\#\rH^0(K,W_{\lambda}(r))^2}
\cdot\#{\rm coker}({\rm loc}_{\overline{v}/{\rm tors}})^2\cdot\prod_{\substack{w\in\Sigma,w\nmid p}}c_w(W_{\lambda}(r)/K),
\]
where ${\rm loc}_{\overline{v}/{\rm tors}}$ is the composition of ${\rm loc}_{\overline{v}}$ with $\rH^1_f(K_{\overline{v}},T_{\lambda}(r))\rightarrow\rH^1_f(K_{\overline{v}},T_{\lambda}(r))/\rH^1_f(K_{\overline{v}},T_\lambda(r))_{\rm tors}$.
\end{prop}

\begin{proof}
By Theorem~\ref{thm:AH-lambda} we know that $\mathcal{X}_v(\lambda\nr)$ is $\Lambda_\scO$-torsion and the nonvanishing of $\mathscr{L}_v(\lambda\nr)(0)$ implies that $\mathscr{F}_v(\lambda\nr)(0)\neq 0$. Since by \cite[Prop.~6.2.1]{cas-Kato-Schoen} the natural restriction map
\[
{\rm Sel}_{v}(K,W_{\lambda}(r))\rightarrow\mathcal{S}_{v}(\lambda\nr)^\Gamma
\]
has finite kernel and cokernel, it follows that ${\rm Sel}_v(K,W_{\lambda}(r))$ if finite. The same argument as in Proposition~\ref{prop:control-BK} then gives 
\begin{equation}\label{eq:early-arg}
\mathscr{F}_v(\lambda\nr)(0)\;\sim_p\;\#\rH^0(K_{\overline{v}},W_{\lambda}(r))^2\cdot
\frac{\#{\rm Sel}_{v}(K,W_{\lambda}(r))}{\#\rH^0(K,W_{\lambda}(r))^2}
\cdot\prod_{\substack{w\in\Sigma,w\nmid p}}c_w(W_{\lambda}(r)/K).
\end{equation}
Thus it remains to show the nonvanishing of \eqref{eq:loc-psichi} and to relate the orders of ${\rm Sel}_v(K,W_{\lambda}(r))$ and $\sha_\BK(K,W_\lambda(r))$. 
%
%
%
 
By Theorem~\ref{thm:ERL-katz}, letting $\mathbf{c}_{\lambda\mathbf{N}^{-r}}(\mathds{1})\in\rH^1(K,T_\lambda(r))$ denote the image of $\mathbf{c}_{\lambda\mathbf{N}^{-r}}$ under the projection $\rH^1_{\rm Iw}(K_\infty,T_\lambda(r))\rightarrow\rH^1(K,T_{\lambda}(r))$, the nonvanishing of $\mathscr{L}_v(\lambda\nr)(0)$ implies that ${\rm loc}_{\overline{v}}(\mathbf{c}_{\lambda\mathbf{N}^{-r}}(\mathds{1}))\neq 0$. Thus to show the nonvanishing of \eqref{eq:loc-psichi}, it suffices to show the inclusion $\mathbf{c}_{\lambda\mathbf{N}^{-r}}(\mathds{1})\in{\rm Sel}_\BK(K,T_\lambda(r))$, 
which is seen in the proof of 
Theorem~\ref{thm:nekovar-bis} in the Appendix. 
%
%
Hence ${\rm loc}_{\overline{v}/{\rm tors}}$ has finite cokernel (since its target has $\scO$-rank one),
and from the global duality argument in the proof of \cite[Prop.~3.2.1]{jsw}\footnote{The irreducibility assumption (irred$_{\mathcal{K}}$) in \emph{loc.\,cit.} is only used to deduce that $W^*\simeq T^\tau$, which is automatic for us.} we find 
\[
\#{\rm Sel}_v(K,W_{\lambda}(r))=\#\sha_\BK(W_{\lambda}(r)/K)\cdot\#{\rm coker}({\rm loc}_{\overline{v}/{\rm tors}})^2,
\]
which together with \eqref{eq:early-arg} yields the result.
\end{proof}


\subsubsection{A consequence of the explicit reciprocity law}

Building on the explicit reciprocity law for $\mathbf{z}_{g,\chi}$, 
we can obtain a useful expression for the order of the cokernel of the map ${\rm loc}_{\overline{v}/{\rm tors}}$ in Proposition~\ref{prop:control-opp-BK}. 

\begin{prop}\label{prop:BCGS}
Let $\lambda$ be as in Theorem~\ref{thm:AH-lambda}. Suppose also that $\mathfrak{d}_K\Vert\cc$, let $(\psi,\chi)$ be a good pair for $\lambda$, and put $g=\theta_\psi$. If $\mathscr{L}_v(\lambda\mathbf{N}^{-r})(0)\neq 0$, then the following hold:
\begin{enumerate}
\item $\mathscr{L}_v(g,\chi)(0)\neq 0$.
\item $z_{g,\chi}$ is non-torsion.
\item ${\rm rank}_{\scO}\,{\rm Sel}_\BK(K,T_{g,\chi})=1$.
\item The map ${\rm loc}_{\overline{v}/{\rm tors}}$ of Proposition~\ref{prop:control-opp-BK} satisfies
\[
\#{\rm coker}({\rm loc}_{\overline{v}/{\rm tors}})=\frac{\#\bigl(\mathscr{O}/\mathscr{L}_v(g,\chi)(0)\bigr)}{\#\bigl({\rm Sel}_{\BK}(K,T_{g,\chi})/\scO\cdot z_{g,\chi})}\cdot\frac{\#\rH^0(K,W_{\lambda}(r))}{\#\rH^0(K_{\overline{v}},W_{\lambda}(r))}\cdot\frac{\#\rH^0(K,W_{\psi^\st\chi}(r))}{\#\rH^0(K_{\overline{v}},W_{\psi^\st\chi}(r))}.
\]
\end{enumerate}
\end{prop}

\begin{proof}
Part (1) follows from the factorization  
of Proposition~\ref{prop:factor-BDP}, the fact that $\mathscr{L}_v(\psi^\st\chi\mathbf{N}^{-r})(0)$ is a nonzero multiple of $L(\psi^\st\chi,r)$ by Theorem~\ref{thm:katz}, and the Definition~\ref{def:good} of good pair; part (2) then follows from Theorem~\ref{thm:BDP-formula}. Similarly as in Proposition~\ref{prop:dec-Sel}, from \eqref{eq:dec-V} we have the decomposition
\begin{equation} 
{\rm Sel}_\BK(K,T_{g,\chi})\simeq{\rm Sel}_{\BK}(K,T_\lambda(r))\oplus{\rm Sel}_v(K,T_{\psi^\st\chi}(r)),\nonumber
\end{equation}
and Theorem~\ref{thm:nekovar-bis} in the Appendix shows the implication 
\[
\mathscr{L}_v(\lambda\mathbf{N}^{-r})(0)\neq 0\quad\Longrightarrow\quad{\rm rank}_{\scO}\,{\rm Sel}_{\BK}(K,T_\lambda(r))=1.
\] 
Since the nonvanishing of $\mathscr{L}_v(\psi^\tau\chi\mathbf{N}^{-r})(0)$ implies $\#{\rm Sel}_v(K,T_{\psi^\st\chi}(r))<\infty$ by virtue of Theorem~\ref{thm:AH} and Mazur's control theorem, 
this yields part (3).

Finally, for the proof of part (4) note that by Lemma~\ref{lem:BK=v} we have
\[
\rH^1_f(K_{\overline{v}},T_\lambda(r))=\rH^1(K_{\overline{v}},T_\lambda(r)),\quad\quad
\rH^1_f(K_{\overline{v}},T_{\psi^\st\chi}(r))=\rH^1(K_{\overline{v}},T_{\psi^\st\chi}(r))_{\rm tors},
\]
and so from the decomposition $T_{g,\chi}\simeq T_\lambda(r)\oplus T_{\psi^\st\chi}(r)$ we see that the map ${\rm loc}_{\overline{v}/{\rm tors}}$ of Proposition~\ref{prop:control-opp-BK} is the same as the composite
\[
{\rm loc}_{\overline{v}/{\rm tors}}':{\rm Sel}_\BK(K,T_{g,\chi})\xrightarrow{{\rm loc}_{\overline{v}}}\rH^1_f(K_{\overline{v}},T_{g,\chi})\rightarrow\rH^1_f(K_{\overline{v}},T_{g,\chi})_{/{\rm tors}},
\]
where $\rH^1_f(K_{\overline{v}},T_{g,\chi})_{/{\rm tors}}:=\rH^1_f(K_{\overline{v}},T_{g,\chi})/\rH^1_f(K_{\overline{v}},T_{g,\chi})_{\rm tors}$. 
As a consequence of Theorem~\ref{thm:ERL}, the argument in Lemma~1.2.3  in \cite{BCGS} (using Remark~1.2.4 in \emph{loc.\,cit.}) shows that
\[
\#{\rm coker}({\rm loc}_{\overline{v}/{\rm tors}}')=\frac{\#\bigl(\mathscr{O}/\mathscr{L}_v(g,\chi)(0)\bigr)}{\#\bigl({\rm Sel}_{\BK}(K,T_{g,\chi})/\scO\cdot z_{g,\chi})}\cdot\frac{\#\rH^0(K,W_{g,\chi})}{\#\rH^0(K_{\overline{v}},W_{g,\chi})}.
\]
Since $\#\rH^0(K,W_{g,\chi})=\#\rH^0(K,W_{\lambda}(r))\cdot\#\rH^0(K,W_{\psi^\st\chi}(r))$, and likewise for the $G_{K_{\overline{v}}}$-invariants, the proof of part (4) follows.
\end{proof}

\begin{cor}\label{cor:opp-BK+BCGS}
In the setting of Proposition~\ref{prop:BCGS}, if $\mathscr{L}_v(\lambda\mathbf{N}^{-r})(0)\neq 0$ then 
for any generator $\mathscr{F}_v(\lambda\nr)\in\Lambda_\scO$ of ${\rm char}_{\Lambda_\scO}(\mathcal{X}_v(\lambda\nr))$ we have
\begin{align*}
\mathscr{F}_v(\lambda\nr)(0)\;&\sim_p\;
\frac{\#\rH^0(K,W_{\psi^\st\chi}(r))^2}{\#\rH^0(K_{\overline{v}},W_{\psi^\st\chi}(r))^2}
\cdot\frac{\#\bigl(\mathscr{O}/\mathscr{L}_v(g,\chi)(0)\bigr)^2}{\#\bigl({\rm Sel}_{\BK}(K,T_{g,\chi})/\scO\cdot z_{g,\chi})^2}\\
&\quad\times\#\sha_\BK(W_{\lambda}(r)/K)\cdot\prod_{\substack{w\in\Sigma,w\nmid p}}c_w(W_{\lambda}(r)/K).
\end{align*}
\end{cor}

\begin{proof}
This is the combination of Proposition~\ref{prop:control-opp-BK} and Proposition~\ref{prop:BCGS}. Note that the contributions from $\#\rH^0(K,W_\lambda(r))$ and $\#\rH^0(K_{\overline{v}},W_\lambda(r))$ in the two formulae cancel each other.
\end{proof}


\begin{proof}[Proof of Theorem~\ref{thmintro:sha}]
Multiplying the formulas (up to a unit) for $\mathscr{F}_v(\psi^\st\chi\nr)(0)$ and $\mathscr{F}_v(\lambda\nr)(0)$ in Proposition~\ref{prop:control-BK} and Corollary~\ref{cor:opp-BK+BCGS}, and using  Proposition~\ref{prop:factor-BDP} and Proposition~\ref{prop:dec-Sel}, we obtain
\[
\#\bigl(\scO/\mathscr{F}_v(g,\chi)(0)\bigr)=
\#\sha_\BK(W_{g,\chi}/K)
\cdot\frac{\#\bigl(\mathscr{O}/\mathscr{L}_v(g,\chi)(0)\bigr)^2}{\#\bigl({\rm Sel}_{\BK}(K,T_{g,\chi})/\scO\cdot z_{g,\chi})^2}\cdot\prod_{\substack{w\in\Sigma,w\nmid p}}c_w(W_{g,\chi}/K).
\] 
Since Theorem~\ref{thm:BDP-IMC} gives
\[  
\#\bigl(\scO/\mathscr{F}_v(g,\chi)(0)\bigr)=\#\bigl(\scO/\mathscr{L}_v(g,\chi)(0)\bigr)^2,
\] 
and this is nonzero by part (1) of Proposition~\ref{prop:control-opp-BK},  this concludes the proof of Theorem~\ref{thmintro:sha}.
\end{proof}


\section{The $p$-part of the Birch--Swinnerton-Dyer formula}\label{sec:p-BSD}

In this section we deduce our main results on the $p$-part of the Birch and Swinnerton-Dyer formula in analytic rank $1$. After a period comparison (responsible for our assumption that $p\nmid h_K$), the result is obtained by  specializing Theorem~\ref{thmintro:sha} to the weight $2$ case and combining it with the general Gross--Zagier formula 
\cite{YZZ,CST}.

\subsection{Periods and heights of Heegner points}\label{subsec:GZ-periods}

Fix a prime $p$, and let $g\in S_2(\Gamma_1(N_g))$ be a newform of weight $2$ and level $N_g$ with $p\nmid N_g$.


\begin{defi}\label{def:can-period}
The \emph{canonical period} of $g$ is
\[
\Omega_g^{\rm can}:=\frac{8\pi^2\langle g,g\rangle_\Gamma}{\eta_g},
\]
where $\langle g,g\rangle_\Gamma=\int_{\Gamma_1(N_g)\backslash\mathfrak{H}}g(z)\overline{g(z)}\frac{dxdy}{y^2}$ is the Petersson norm of $g$, and  $\eta_g$ is the congruence number of $g$ relative to $S_2(\Gamma_1(N_g))$ (see \cite[\S{7}]{hida-congruence-81}).
\end{defi}

We are interested in the case where $g=\theta_\psi$ is the theta series of a Hecke character $\psi$ of an imaginary quadratic field $K$ satisfying \eqref{eq:spl}. Suppose $L$ is a number field containing the Fourier coefficients of $g$, and let $\Phi$ be the completion of $L$ at a fixed prime above $p$, with ring of integers $\scO$. 

From now on, we also suppose $p>3$.

\begin{prop}\label{prop:petersson-CM}
Let $\psi$ be a Hecke character of $K$ of infinity type $(-1,0)$, and conductor prime to $p$, and put $g=\theta_\psi$. Suppose $\psi^-:=\psi^\st/\psi$ satisfies conditions (S5)--(S7) of Definition~\ref{def:good} and has conductor divisible only by primes that are split in $K$. Then
%
\[
\Omega_g^{\rm can}=\Omega^2
\]
up to a unit in $\cO_L^\times$.
\end{prop}

\begin{proof}
Let $\boldsymbol{g}$ be the Hida family passing through the ordinary $p$-stabilization 
\[
g_\alpha:=g(q)-\varepsilon_g(p)p\psi(\overline{v})^{-1}g(q^p),
\]
and let $\eta_{\boldsymbol{g}}$ 
be its associated congruence power series, 
normalized so that its specialization at 
the trivial character gives the congruence number $\eta_{g_\alpha}$ of $g_\alpha$.  
Under our hypotheses on $\psi^-$, it follows from the proof of the anticyclotomic Iwasawa main conjecture for CM fields by Hida--Tilouine \cite{HT-ENS,HT-invmath} and Hida \cite{hida-coates} that 
\begin{equation}\label{eq:HT-IMC}
\eta_{\boldsymbol{g}}=h_K\cdot\mathscr{L}_{v}(\psi^-\mathbf{N}^{-1})
\end{equation}
up to a $p$-adic unit, where $h_K=\#{\rm Pic}(\cO_K)$ is the class number of $K$. (Note that in \eqref{eq:HT-IMC} we have used the functional equation of Theorem~\ref{thm:katz} to write $\mathscr{L}_v(\psi^-\mathbf{N}^{-1})$ in place of $\mathscr{L}_v(\psi^-)$.) 

Having infinity type $(2,0)$, the character $\psi^-\mathbf{N}^{-1}$ is in the range of interpolation of the Katz $v$-adic $L$-function of Theorem~\ref{thm:katz}, and from there we obtain
\begin{equation}\label{eq:interp-katz}
h_K\cdot\frac{\mathscr{L}_{v}(\psi^-\mathbf{N}^{-1})(\mathds{1})}{\Omega_p^2}=\mathcal{E}_p(\psi^-)\cdot h_K\cdot\frac{L(\psi^-,1)}{\Omega^2},
\end{equation}
where $\mathcal{E}_p(\psi^-):=(1-\psi^-(v))(1-\psi^-(v)p^{-1})$. By Hida's formula (see \cite[Thm.\,7.1]{HT-ENS}) and Dirichlet's class number formula, we have
\begin{equation}\label{eq:hida-formula}
\langle g,g\rangle_\Gamma
=\frac{D_K^2}{2^3\pi^2}\cdot\frac{h_K}{w_K\sqrt{D_K}}\cdot L(\psi^-,1),\nonumber
\end{equation}
where $w_K=\#\cO_K^\times$, and so together with \eqref{eq:HT-IMC}, equality \eqref{eq:interp-katz} can be rewritten as the equality
\begin{equation}\label{eq:interp-rewrite}
\frac{\eta_{g_\alpha}}{\Omega_p^2}\cdot u=\mathcal{E}_p(\psi^-)\cdot\frac{8\pi^2\langle g,g\rangle_\Gamma}{\Omega^2}
\end{equation}
for some unit $u\in\cO_{\C_p}^\times$. Since the right-hand side of \eqref{eq:interp-katz} is in $L^\times$ as a consequence of the proof of Deligne's conjecture for Hecke characters (see \cite{goldstein-shappacher} and \cite{blasius}), it follows that $u/\Omega_p^2\in\cO_L^\times$, and hence \eqref{eq:interp-rewrite} together with the well-known relation $\eta_{g_\alpha}=\mathcal{E}_p(\psi^-)\eta_g$ (see e.g. \cite[Rem.~2.18]{wiles}) 
yields the result.
\end{proof}

Using 
Proposition~\ref{prop:petersson-CM}, we can  now rewrite the general Gross--Zagier formula  for modular curves in a form that will be convenient for us. 

Let $B_{g,\chi}/K$ be the abelian variety (up to $K$-isogeny)  associated to $(g,\chi)$; letting $B_\psi/K$ be a CM abelian variety  in the isogeny class of such abelian varieties associated to $\psi$ by Casselman's theorem (see \cite[Thm.~6]{shimura-1971}), it can be described as the Serre tensor $B_{g,\chi}:=B_\psi\otimes\chi$, and satisfies
\begin{equation}\label{eq:L-B-g-chi}
L(B_{g,\chi}/K,s)=\prod_{\sigma:L\hookrightarrow\C}L(g^\sigma/K,\chi^\sigma,s).\nonumber
\end{equation}
After possibly changing $B_{g,\chi}$ it within its $K$-isogeny class, we have an embedding $\cO_L\hookrightarrow{\rm End}_K(B_{g,\chi})$. 

The N\'{e}ron--Tate height pairing gives a $\Q$-bilinear non-degenerate pairing
\begin{equation}\label{eq:NT}
\langle\cdot,\cdot\rangle_{\rm NT}:B_{g,\chi}(K)_\Q\times B_{g,\chi}^\vee(K)_\Q\rightarrow\R, \nonumber
\end{equation}
where we put $(\cdot)_\Q=(\cdot)\otimes_{\Z}\Q$ and $B_{g,\chi}^\vee/K$ is the dual abelian variety. As explained in \cite[\S{1.2.4}]{YZZ}, 
$\langle\cdot,\cdot\rangle_{\rm NT}$ induces an $L$-linear pairing
\begin{equation}\label{eq:NT-L}
\langle\cdot,\cdot\rangle_{L}:B_{g,\chi}(K)_\Q\otimes_L B_{g,\chi}^\vee(K)_\Q\rightarrow L_\R:=L\otimes_\Q\R.
\end{equation}

In the next result we view $L(B_{g,\chi}/K,s)$ 
as valued in $L_\C:=L\otimes_{\Q}\C$. 

\begin{thm}\label{thm:YZZ-CST}
Let $\psi$ be as in Proposition~\ref{prop:petersson-CM}, and assume further that it satisfies condition (S8) of Definition~\ref{def:good}. Put $g=\theta_\psi$, and let $\chi$ be  such that we have $(g,\chi)\in S_2(\Gamma_0(N_g),\varepsilon_g)\times\Sigma_{\rm cc}(c,\mathfrak{N}_g,\varepsilon_g)$ for some $c>0$ prime to $N_g$. If $p\nmid h_K$, there exist Heegner points $y_{g,\chi}\in B_{g,\chi}(K)$ and $y_{g,\chi}'\in B_{g,\chi}^\vee(K)$ such that
%
\[
\frac{L'(B_{g,\chi}/K,1)}{\Omega_g^{\rm can}}=\langle y_{g,\chi},y_{g,\chi}'\rangle_L,
\]
where the equality is up to a $p$-adic unit.
\end{thm}

\begin{proof}
The general Gross--Zagier formula of \cite{YZZ}, as made explicit in \cite[Thm.~1.5]{CST}, reads
\[
L'(B_{g,\chi}/K,1)=\frac{8\pi^2\langle g,g\rangle_\Gamma}{u_c^2\cdot\sqrt{D_K}\cdot c}\cdot\frac{\langle y_{g,\chi},y_{g,\chi}'\rangle_L}{{\rm deg}(\pi_g)},
\]
where $u_c:=\frac{1}{2}\#\cO_c^\times$ and $\pi_g:X_1(N_g)\rightarrow A_g$ is an optimal quotient. Hence it suffices to show that the congruence number $\eta_g$ and ${\rm deg}(\pi_g)$ differ by a $p$-adic unit. As shown in the proof of Proposition~\ref{prop:petersson-CM}, combining \eqref{eq:HT-IMC} and \eqref{eq:interp-katz} we obtain the equality up to a $p$-adic unit
\[
\mathcal{E}_p(\psi^-)\cdot\eta_g=\mathcal{E}_p(\psi^-)\cdot h_K\cdot\frac{L(\psi^-,1)}{\Omega^2}.
\]
In particular, if $p\nmid h_K$ and $\psi^-$ satisfies (S8) in Definition~\ref{def:good}, this shows that $\eta_g$ is a $p$-adic unit. Since \cite[Thm.~3.6(a)]{ARS} shows that ${\rm deg}(\pi_g)$ divides $\eta_g$, the result follows. 
\end{proof}

\begin{rem}
Let $\PP$ be a prime of $L$ above $p$ and denote by $\scO$ the ring of integers of the completion of $L$ at $\PP$. Then $T_{g,\chi}\simeq\varprojlim_m B_{g,\chi}[\PP^m]$ as $G_K$-modules, and the Heegner point $y_{g,\chi}$ of Theorem~\ref{thm:YZZ-CST} can be taken so that its image under the Kummer map 
\[
B_{g,\chi}\otimes_{\cO_L}\scO\rightarrow\varprojlim_m{\rm Sel}_{\PP^m}(B_{g,\chi}/K)\simeq{\rm Sel}_\BK(K,T_{g,\chi})
\] 
agrees with the Heegner class $z_{g,\chi}$ of Theorem~\ref{thm:BDP-formula}. and $y_{g,\chi}'$ so that its Kummer image agrees with image of $z_{g,\chi}$ under the isomorphism ${\rm Sel}_\BK(K,T_{g,\chi})\simeq{\rm Sel}_\BK(K,T_{g,\chi}^\st)$ given by the action of complex conjugation.
\end{rem}



\subsection{Theorem~\ref{thmintro:p-BSD-A} implies Theorem~\ref{thmintro:p-BSD-E}}




As is well-known (see e.g. \cite[Thm.~4.1]{goldstein-shappacher}), the assumption that the field extension $F(E_{\rm tors})/K$ is abelian implies 
that the Weil restriction $B:={\rm Res}_{F/K}(E)$ is an abelian variety with complex multiplication by an order in a product of CM fields 
\[
L=L_1\times\cdots\times L_r
\] 
containing $K$ with $[L:K]=\sum_{i=1}^r[L_i:K]
={\rm dim}(B)$. Moreover, for each $i$ there is an abelian variety $B_i/K$ with CM by an order in $L_i$ which combine to an isogeny
\[
B\rightarrow\prod_{i=1}^rB_i
\] 
defined over $K$. By invariance of the Birch Swinnerton-Dyer conjecture under isogenies and restriction of scalars, Theorem~\ref{thmintro:p-BSD-E} thus follows from  Theorem~\ref{thmintro:p-BSD-A} applied to each of the isogeny factors $B_i$.

\subsection{Proof of Theorem~\ref{thmintro:p-BSD-A}}\label{subsec:p-BSD} 

\begin{proof}[Proof of Theorem~\ref{thmintro:p-BSD-A}(i)] 
%
Let $\mathcal{L}$ be a number field containing the values of $\psi$ and $\chi$ (so $\mathcal{L}$ contains the field $L$ of values of $\lambda=\psi\chi$). Since ${\rm ord}_{s=1}L(\lambda,s)=1$, $\lambda\mathbf{N}^{-1}$ has root number $w(\lambda\mathbf{N}^{-1})=-1$, and by Lemma~\ref{lem:good} we can fix a good pair $(\psi,\chi)$ for $\lambda$, and let $B_{g,\chi}/K$ be the associated abelian variety. By Theorem~\ref{thm:YZZ-CST}, the nonvanishing of $L'(g/K,\chi,1)$ also gives $y_{g,\chi}\notin B_{g,\chi}(K)_{\rm tors}$, and so 
\[
{\rm rank}_{\cO_{\mathcal{L}}}B_{g,\chi}(K)\geq 1.
\]

On the other hand, let $\PP$ be a prime of $\mathcal{L}$ above $p$, and denote by $\scO_\PP$ and $T_{\lambda,\PP}(1)$ the completion of $\cO_{\mathcal{L}}$ at $\PP$ and the associated $\scO_\PP$-module of rank $1$ with $G_K$-action via the $\PP$-adic avatar of $\lambda$, and define $T_{\psi^\st\chi,\PP}(1)$, $T_{g,\chi,\PP}$, $W_{\psi^\st\chi,\PP}(1)$, etc. similarly. By Theorem~\ref{thm:nekovar-bis} and Remark~\ref{rem:wt2-ran1} in the Appendix, the assumption that ${\rm ord}_{s=1}L(\lambda,s)=1$ implies that ${\rm rank}_{\scO_\PP}{\rm Sel}_\BK(K,T_{\lambda,\PP})=1$; since by Theorem~\ref{thm:AH} the nonvanishing of $L(\psi^\st\chi,1)$ gives $\#{\rm Sel}_\BK(K,T_{\psi^\st\chi,\PP}(1))<\infty$, by the decomposition
\[
{\rm Sel}_\BK(K,T_{g,\chi,\PP})\simeq{\rm Sel}_\BK(K,T_{\lambda,\PP}(1))\oplus{\rm Sel}_\BK(K,T_{\psi^\st\chi,\PP}(1))
\]
it follows that ${\rm rank}_{\scO_\PP}{\rm Sel}_\BK(K,T_{g,\chi,\PP})=1$, and hence
\[
{\rm rank}_{\cO_{\mathcal{L}}}B_{g,\chi}(K)\leq 1.
\]
Thus ${\rm rank}_{\cO_{\mathcal{L}}}B_{g,\chi}(K)=1$ and $\#\sha(B_{g,\chi}/K)[\PP^\infty]<\infty$; in particular, ${\rm rank}_{\Z} B_{g,\chi}(K)=[{\mathcal{L}}:\Q]$.  
Since up to isogeny the CM abelian variety $A$ is the CM abelian variety $B_\lambda$ attached to $K$ by Casselman's theorem, there is an isogeny defined over $K$
\[
i_\lambda:B_{g,\chi}\rightarrow A\otimes_{\cO_{L}}\cO_{\mathcal{L}}
\]
compatible with the action of $\cO_{\mathcal{L}}$ by endomorphisms on both sides (cf. \cite[Lem.~2.9]{BDP-PJM}), so the above shows that
\[
{\rm rank}_\Z A(K)=[L:\Q].
\]
Since $L(A/K,s)=\prod_{\sigma:L\hookrightarrow\C}L(\lambda^\sigma,s)$, this gives 
\[
{\rm ord}_{s=1}L(A/K,s)={\rm rank}_\Z A(K).
\] 
Moreover, since we have also shown that
\[
{\rm rank}_{\scO_{\PP}}{\rm Sel}_\BK(K,T_{\lambda,\PP}(1))={\rm rank}_{\scO_\PP}{\rm Sel}_\BK(K,T_{g,\chi,\PP})=1,
\]
the conclusion $\#\sha(A/K)[\PP^\infty]<\infty$ also follows.
\end{proof}

\begin{proof}[Proof of Theorem~\ref{thmintro:p-BSD-A}(ii)] 
Put $\Lg_\C=\Lg\otimes_\Q\C\simeq\prod_\sigma\C$, where $\sigma$ runs over all field embeddings $\Lg\hookrightarrow\C$, write the $\Lg$-linear
pairing $\langle\cdot,\cdot\rangle_\Lg$ in \eqref{eq:NT-L} as $(\langle\cdot,\cdot\rangle_{\Lg,\sigma})_\sigma$ according to this decomposition, and write the associated regulator ${\rm Reg}(B_{g,\chi})$ as $({\rm Reg}^\sigma(B_{g,\chi}))_\sigma$. Similar remarks apply to the $L$-linear pairing 
\[
\langle\cdot,\cdot\rangle_{L}:A(K)_\Q\otimes_L A^\vee(K)_\Q\rightarrow L_\R.
\] 
Then by Theorem~\ref{thm:YZZ-CST} we have
\begin{equation}\label{eq:GZ-rational}
\biggl(\frac{L'(g^\sigma/K,\chi^\sigma,1)}{\Omega_{g^\sigma}^{\rm can}\cdot\langle y_{g,\chi},y_{g,\chi}'\rangle_{\Lg,\sigma}}\biggr)_\sigma\in \Lg^\times\subset \Lg_\C^\times.
\end{equation}
Recall the $K$-isogeny $i_\lambda:B_{g,\chi}\rightarrow A\otimes_{\cO_{L}}\cO_\Lg$. Put 
\[
y_\lambda:=i_\lambda(y_{g,\chi})\in A(K),
\] 
and define $y_\lambda'\in A^\vee(K)$ by $i_\lambda^\vee(y_\lambda')=y_{g,\chi}'$. 
Then from the factorization $L(g/K,\chi,s)=L(\lambda,s)L(\psi^\st\chi,s)$, Proposition~\ref{prop:petersson-CM}, and the projection formula for heights (see \cite[(3.4.3)]{MT-biextensions}), we have
\begin{equation}\label{eq:factor-rational}
\eqref{eq:GZ-rational}=\biggl(\frac{L'(\lambda^\sigma,1)}{\Omega_{\sigma}\cdot\langle y_\lambda,y_\lambda'\rangle_{L,\sigma}}\cdot\frac{L((\psi^\st\chi)^\sigma,1)}{\Omega_\sigma}\biggr)_\sigma
\end{equation}
up to multiplication by an element in $\Lg^\times$. Combining the first assertion of Theorem~\ref{thm:pTNC-rk0}, \eqref{eq:GZ-rational} and \eqref{eq:factor-rational}, we conclude that 
\[
\biggl(\frac{L'(\lambda^\sigma,1)}{\Omega_\sigma\cdot\langle y_\lambda,y_\lambda'\rangle_{L,\sigma}}\biggr)_\sigma\in L^\times\subset L_{\C}^\times.
\]
Since ${\rm rank}_{\cO_{L}}A(K)=1$ by part (i) and $y_\lambda\notin A(K)_{\rm tors}$ (given that $y_{g,\chi}\notin B_{g,\chi}(K)_{\rm tors}$ by Theorem~\ref{thm:YZZ-CST} and the fact that ${\rm ord}_{s=1}L(g/K,\chi,s)=1$), we deduce that 
\[
\biggl(\frac{L'(\lambda^\sigma,1)}{\Omega_\sigma\cdot{\rm Reg}^\sigma(A)}\biggr)_\sigma\in L^\times\subset L_{\C}^\times, 
\]
which together with the equalities $L(A/K)=\prod_\sigma L(\lambda^\sigma,s)$ and $\Omega(A)=\prod_\sigma\Omega_\sigma$ 
(see\cite[Prop.~3.3]{burungale-flach}) yields the result.
\end{proof}

\begin{proof}[Proof of Theorem~\ref{thmintro:p-BSD-A}(iii)] 
Since ${\rm rank}_{\cO_\Lg} B_{g,\chi}(K)=1$ and $y_{g,\chi}\notin B_{g,\chi}(K)_{\rm tors}$, 
the definition of ${\rm Reg}(B_{g,\chi})$ gives
\[
\langle y_{g,\chi},y_{g,\chi}'\rangle_{\Lg}={\rm Reg}(B_{g,\chi})\cdot[B_{g,\chi}(K):\cO_\Lg\cdot y_{g,\chi}]\cdot[B_{g,\chi}^\vee(K):\cO_\Lg\cdot y_{g,\chi}'].
\]
Thus letting $\PP$ be any prime of $\Lg$ above $p$, by Theorem~\ref{thm:YZZ-CST} we have
\begin{equation}\label{eq:YZZ-pBSD}
{\rm ord}_{\PP}\biggl(\frac{L'(g/K,\chi,1)}{\Omega_{g}^{\rm can}\cdot{\rm Reg}(B_{g,\chi})}\biggr)={\rm ord}_\PP\bigl([B_{g,\chi}(K)\otimes_{\cO_\Lg}\scO_\PP:\scO_\PP\cdot y_{g,\chi}]\cdot[B_{g,\chi}^\vee(K)\otimes_{\cO_\Lg}\scO_\PP:\scO_\PP\cdot y_{g,\chi}']\bigr).
\end{equation}
Since $\#\sha(B_{g,\chi}/K)[\PP^\infty]<\infty$ as shown in the proof of part (i), the Kummer map gives isomorphisms
\[
B_{g,\chi}(K)\otimes_{\cO_\Lg}\scO_\PP\rightarrow{\rm Sel}_\BK(K,T_{g,\chi,\PP}(1)),\quad
B_{g,\chi}^\vee(K)\otimes_{\cO_\Lg}\scO_\PP\rightarrow{\rm Sel}_\BK(K,T_{g,\chi,\PP}^\st(1)).
\]
By the isomorphism ${\rm Sel}_\BK(K,T_{g,\chi,\PP}^\st(1))\simeq{\rm Sel}_\BK(K,T_{g,\chi,\PP}(1))$ given by the action of complex conjugation, \eqref{eq:YZZ-pBSD} can therefore be rewritten as
\[
{\rm ord}_{\PP}\biggl(\frac{L'(g/K,\chi,1)}{\Omega_{g}^{\rm can}\cdot{\rm Reg}(B_{g,\chi})}\biggr)=2\cdot{\rm length}_{\scO_\PP}
\bigl({\rm Sel}_{\BK}(K,T_{g,\chi,\PP}(1))/\scO_\PP\cdot z_{g,\chi}\bigr),
\]
which by Theorem~\ref{thmintro:sha} 
becomes the equality
\begin{equation}\label{eq:use-GZ}
{\rm ord}_{\PP}\biggl(\frac{L'(g/K,\chi,1)}{\Omega_{g}^{\rm can}\cdot{\rm Reg}(B_{g,\chi})}\biggr)
=
{\rm ord}_\PP(\#\sha_\BK(W_{g,\chi,\PP}(1)/K))
+\sum_{\substack{w\in\Sigma,w\nmid p}}{\rm ord}_\PP(c_w(W_{g,\chi,\PP}(1)/K)).
\end{equation}
On the other hand, in the same manner as in the proof of part (ii), from the factorization $L(g/K,\chi,s)=L(\lambda,s) L(\psi^\st\chi,s)$ and Proposition~\ref{prop:petersson-CM}, and using Theorem~\ref{thm:pTNC-rk0} for the second equality, we obtain
\begin{equation}\label{eq:use-work}
\begin{aligned}
&{\rm ord}_{\PP}\biggl(\frac{L'(g/K,\chi,1)}{\Omega_{g}^{\rm can}\cdot{\rm Reg}(B_{g,\chi})}\biggr)\\
&\quad={\rm ord}_\PP\biggl(\frac{L'(\lambda,1)}{\Omega\cdot{\rm Reg}(A)}\biggr)+{\rm ord}_\PP(\#\rH^0(K,W_{\psi^\st\chi,\PP}(1))^2)+{\rm ord}_\PP\biggl(\frac{L(\psi^\st\chi,1)}{\Omega}\biggr)\\
&\quad={\rm ord}_\PP\biggl(\frac{L'(\lambda,1)}{\Omega\cdot{\rm Reg}(A)}\biggr)+{\rm ord}_\PP(\#\sha_{\BK}(W_{\psi^\st\chi,\PP}(1)/K))
+\sum_{\substack{w\in\Sigma,w\nmid p}}{\rm ord}_\PP(c_w(W_{\psi^\st\chi,\PP}(1)/K)).
\end{aligned}
\end{equation}
Combining \eqref{eq:use-GZ} and \eqref{eq:use-work} and using the decomposition $W_{g,\chi}\simeq W_{\lambda}(1)\oplus W_{\psi^\st\chi}(1)$ we thus arrive at
\begin{equation}\label{eq:combined-pBSD}
{\rm ord}_{\PP}\biggl(\frac{L'(\lambda,1)}{\Omega\cdot{\rm Reg}(A)}\biggr)={\rm ord}_{\PP}(\#\sha_{\BK}(W_{\lambda,\PP}(1)/K))+\sum_{\substack{w\in\Sigma,w\nmid p}}{\rm ord}_{\PP}(c_w(W_{\lambda,\PP}(1)/K)).
\end{equation}
Noting that the right-hand side of \eqref{eq:combined-pBSD} can be rewritten as
\[
{\rm ord}_{\PP}(\#\sha(A/K)[\PP^\infty])+\sum_{\substack{\substack{w\in\Sigma,w\nmid p}}}{\rm ord}_{\PP}({\rm Tam}(A/K)),
\]
this concludes the proof of Theorem~\ref{thmintro:p-BSD-A}.
\end{proof}

\appendix

\section{Higher weight CM analogues of theorems of Nekov\'{a}\v{r} and Skinner}

\smallskip
\begin{center}by \textsc{Francesc Castella and Mychelle Parker\footnote{\emph{Email address: \tt{mychelleparker@math.ucsb.edu}}}}\end{center}
\medskip

In this appendix we prove analogues, for CM modular forms of  weight $2r\geq 2$, of Nekov\'{a}\v{r}'s Euler system argument \cite{nekovar-invmath} and Skinner's $p$-converse  
to the theorem of Gross--Zagier--Kolyvagin \cite{skinner-converse}. 
The first result (Theorem~\ref{thm:nekovar-bis})  
is an ingredient in the proof of the main results in the body of the paper 
and is new even in weight $2$, while the second (Theorem~\ref{thm:p-conv}) is a higher weight generalization of the main result of \cite{BT-conv} and \cite{BCST}. 
A form of these results is studied by the second author of this Appendix in her PhD thesis \cite{parker-PhD}.

Throughout we let $K$ be an imaginary quadratic field satisfying 
\eqref{eq:spl} for a prime $p>3$.

\subsection{Nekov\'{a}\v{r}'s theorem for higher weight CM forms} 


Let $f\in S_{2r}(\Gamma_0(N_f))$ be an eigenform of weight $2r\ge 2$ and level $N_f$ prime to $p$. Granted the expected  injectivity of the $p$-adic 
Abel--Jacobi map, the works of Nekov\'{a}\v{r} \cite{nekovar-invmath} and S.-W.~Zhang \cite{zhang-GZ} (extending to weights $2r>2$ landmark results of Kolyvagin \cite{kolyvagin-mw-sha} and Gross--Zagier \cite{grosszagier}) yield a proof of the implication
\begin{equation}\label{eq:nek-swzhang}
{\rm ord}_{s=r}L(f,s)=1\quad\Longrightarrow\quad
{\rm dim}_\Phi\,{\rm Sel}_\BK(\Q,V_f(r))=1.
\end{equation}
Key to the proof of this result are Heegner cycles attached to an auxiliary imaginary quadratic field $K'$ satisfying the Heegner hypothesis \eqref{eq:Heeg} relative to $N_f$.

Here we are interested in the case where $f=\theta_\lambda$ is the theta series of a Hecke character $\lambda$ of infinity type $(1-2r,0)$, so 
$\varepsilon_\lambda=\eta_K$. In particular, the conductor $\cc$ of $\lambda$ is divisible by $\mathfrak{d}_K$,  and so $K$ does \emph{not} satisfy the Heegner hypothesis \eqref{eq:Heeg} relative to $N_f=D_K\mathbf{N}(\cc)$. Notwithstanding, 
sidestepping the use of Heegner classes we can prove the following  analogue of \eqref{eq:nek-swzhang} for $f=\theta_\lambda$ with CM by $K$ using only arithmetic over $K$. 

\begin{thm}\label{thm:nekovar-bis}
Let $\lambda$ be a Hecke character of $K$ of infinity type $(1-2r,0)$ for some $r\geq 1$, conductor $\cc$ prime to $p$, and central character $\varepsilon_\lambda=\eta_K$.  Then
\[
\mathscr{L}_v(\lambda\mathbf{N}^{-r})(0)\neq 0\quad\Longrightarrow\quad{\rm dim}_\Phi\,{\rm Sel}_\BK(K,V_\lambda(r))=1.
\]
Equivalently, if $\mathscr{L}_v(\lambda\mathbf{N}^{-r})(0)\neq 0$ then ${\rm Sel}_\BK(\Q,V_f(r))$ is $1$-dimensional, where $f=\theta_\lambda$. 
\end{thm}

\begin{rem}\label{rem:wt2-ran1}
When $r=1$ and $\mathfrak{d}_K\Vert\cc$, the assumption $\mathscr{L}_v(\lambda\nr)(0)\neq 0$ follows from ${\rm ord}_{s=1}L(\lambda,s)=1$. Indeed, letting $(\psi,\chi)$ be a good pair for $\lambda$ (as they exist by Lemma~\ref{lem:good}; note that $w(\lambda\mathbf{N}^{-1})=-1$ by the assumption that ${\rm ord}_{s=1}L(\lambda,s)=1$), putting $g=\theta_\psi$  and letting $z_{g,\chi}\in{\rm Sel}_\BK(K,T_{g,\chi})$ be the Heegner class of Theorem~\ref{thm:BDP-formula}, which in this case arises as the image under the Kummer map
\begin{equation}\label{eq:Kummer}
B_{g,\chi}(K)\otimes_{\cO_L}\scO\rightarrow{\rm Sel}_\BK(K,T_{g,\chi})
\end{equation}
of a Heegner cycle on a certain CM abelian variety $B_{g,\chi}/K$, we have
\begin{align*}
{\rm ord}_{s=1}L(\lambda,s)=1&\quad\Longrightarrow\quad{\rm ord}_{s=1}L(g/K,\chi,s)=1\tag{by \eqref{eq:factor-good}}\\
&\quad\Longrightarrow\quad z_{g,\chi}\neq 0\tag{by \cite{YZZ}}\\
&\quad\Longrightarrow\quad\mathscr{L}_v(g,\chi)(0)\neq 0\tag{by Theorem~\ref{thm:BDP-formula}}\\
&\quad\Longrightarrow\quad\mathscr{L}_v(\lambda\mathbf{N}^{-1})(0)\neq 0\tag{by Proposition~\ref{prop:factor-BDP}}.
\end{align*}
Note that the second implication relies on the injectivity   \eqref{eq:Kummer}, and the third on the non-triviality of the localization map 
\[
{\rm loc}_{\overline{v}}:{\rm Sel}_\BK(K,T_{g,\chi})\rightarrow\rH^1_f(K_{\overline{v}},T_{g,\chi}),
\] 
both of which are expected---but not known---to continue to hold in weight $2r>2$, with \eqref{eq:Kummer} replaced by the $p$-adic \'{e}tale Abel--Jacobi map on generalized Kuga--Sato varieties studied in \cite{BDP}.
\end{rem}

\begin{proof}[Proof of Theorem~\ref{thm:nekovar-bis}]
Let $\mathbf{c}_{\lambda\mathbf{N}^{-r}}\in\rH^1_{\rm Iw}(K_\infty,T_\lambda(r))$ be the $\Lambda_\scO$-adic class (of twisted elliptic units)
of Theorem~\ref{thm:ERL-katz}, and denote by $\mathbf{c}_{\lambda\mathbf{N}^{-r}}(\mathds{1})$ its image  under the projection 
\[
\rH^1_{\rm Iw}(K_\infty,T_\lambda(r))\rightarrow\rH^1(K,T_{\lambda}(r)).
\]
As a consequence of Theorem~\ref{thm:ERL-katz}, we have
\[
\mathscr{L}_v(\lambda\mathbf{N}^{-r})(0)\neq 0\quad\Longrightarrow\quad{\rm loc}_{\overline{v}}(\mathbf{c}_{\lambda\mathbf{N}^{-r}}(\mathds{1}))\neq 0.
\] 
We claim that $\mathbf{c}_{\lambda\mathbf{N}^{-r}}(\mathds{1})$ lands in ${\rm Sel}_\BK(K,T_\lambda(r))$. Indeed, since $\lambda\mathbf{N}^{-r}$ has infinity type $(1-r,r)$, by Lemma~\ref{lem:BK=v} the claim amounts to the assertion that
\begin{equation}\label{eq:ell-Sel}
{\rm loc}_v(\mathbf{c}_{\lambda\mathbf{N}^{-r}}(\mathds{1}))=0,
\end{equation}
which by Theorem~\ref{thm:ERL-katz} for $w=\overline{v}$ amounts to the assertion that $\mathscr{L}_{v}(\lambda^\st\mathbf{N}^{-r})(0)=0$. Since $\lambda^\st\mathbf{N}^{-r}$ is within the range of interpolation of $\mathscr{L}_{v,\cc}$ (as its infinity type is $(r,1-r)$), by Theorem~\ref{thm:katz} we have
\[
\mathscr{L}_{v}(\lambda^\st\mathbf{N}^{-r})(0)=0\quad\Longleftrightarrow\quad L(\lambda^\st,r)=L(\lambda,r)=0, 
\]
whence the claim.  In particular, this shows that the map 
\begin{equation}\label{eq:locv-nek}
{\rm loc}_{\overline{v}}:{\rm Sel}_\BK(K,T_\lambda(r))
\rightarrow\rH^1_f(K_{\overline{v}},T_\lambda(r))=\rH^1(K_{\overline{v}},T_\lambda(r))
\end{equation}
is nonzero, using Lemma~\ref{lem:BK=v} for the last equality. 

It follows from the discussion in $\S{3.3}$ and Theorem~6.4.1 in \cite{rubin-ES} that the class $\mathbf{c}_{\lambda\mathbf{N}^{-r}}(\mathds{1})$ extends to an anticyclotomic Euler system for $T_\lambda(r)$; since the above shows that this Euler system is non-trivial, by \cite[Thm.~2.2.3]{rubin-ES} (see also  [\emph{loc.\,cit.}, \S{9.3}]) it follows that ${\rm rank}_\scO\,{\rm Sel}_{\rm rel}(K,T_\lambda(r))=1$. From the global duality exact sequence
\[
{\rm Sel}_\BK(K,T_{\lambda^\st}(r))\xrightarrow{{\rm loc}_{v}}\rH^1(K_{v},T_{\lambda^{\st}}(r))\rightarrow{\rm Sel}_{\rm rel}(K,W_\lambda(r))^\vee\rightarrow{\rm Sel}_\BK(K,W_\lambda(r))^\vee\rightarrow 0,
\]
noting that the non-vanishing of \eqref{eq:locv-nek} implies the non-vanishing of the map ${\rm loc}_v$ in this sequence, we conclude that 
${\rm corank}_\scO\,{\rm Sel}_\BK(K,W_\lambda(r))={\rm corank}_\scO\,{\rm Sel}_{\rm rel}(K,W_\lambda(r))=1$, whence the result.
%
%
\end{proof}


\subsection{Skinner's $p$-converse theorem for higher weight CM forms} \label{subsec:conv}

As in the preceding section, let $f=\theta_\lambda\in S_{2r}(\Gamma_0(N_f))$. If $(g,\chi)$ is any good pair for $\lambda$, we denote by $z_f\in{\rm Sel}_\BK(\Q,V_f(r))$ the projection of the Heegner class $z_{g,\chi}$ of Theorem~\ref{thm:BDP-formula} onto the first factor in the decomposition
\begin{equation}\label{eq:dec-Sel-BK}
{\rm Sel}_\BK(K,V_{g,\chi})\simeq{\rm Sel}_{\BK}(\Q,V_f(r))\oplus{\rm Sel}_v(K,V_{\psi^\st\chi})
\end{equation}
arising from \eqref{eq:dec-V} as in Proposition~\ref{prop:dec-Sel}. (Note that here we used the identifications ${\rm Sel}_{\overline{v}}(K,V_\lambda)={\rm Sel}_\BK(K,V_\lambda)\simeq{\rm Sel}_{\BK}(\Q,V_f(r))$ arising from Lemma~\ref{lem:BK=v} and Shapiro's lemma.)

\begin{thm}\label{thm:p-conv}
Let $f=\theta_\lambda\in S_{2r}(\Gamma_0(N_f))$ be an eigenform of weight $2r\geq 2$ with CM by $K$, with associated Hecke character $\lambda$ of infinity type $(1-2r,0)$, conductor $\cc$, and central character $\varepsilon_\lambda=\eta_K$.  Suppose $\mathfrak{d}_K\Vert\cc$. Then
\[
{\rm dim}_\Phi\,{\rm Sel}_{\rm BK}(\Q,V_f(r))=1\quad\Longrightarrow\quad z_f\neq 0.
\]
In particular, if ${\rm loc}_{\overline{v}}({\rm Sel}_\BK(\Q,V_f(r)))\neq 0$, then the converse to Theorem~\ref{thm:nekovar-bis} holds.
\end{thm}

\begin{proof}
By Corollary~\ref{cor:parity}, if ${\rm Sel}_\BK(\Q,V_f(r))$ 
is $1$-dimensional then $\lambda$ has root number 
$w(\lambda\mathbf{N}^{-r})=-1$. Let $(g,\chi)$ be a strong good pair for $\lambda$ (see Lemma~\ref{lem:good}). As in Proposition~\ref{prop:dec-Sel}\footnote{Noting that by Remark~\ref{rem:isog-inv} we may assume that $T_{g,\chi}\simeq T_\lambda\oplus T_{\psi^\st\chi}$.}, we then have the decomposition
\begin{equation}\label{eq:dec-pconv}
\begin{aligned}
\mathcal{X}_\BK(g,\chi)&\simeq\mathcal{X}_{\overline{v}}(\lambda)\oplus\mathcal{X}_{v}(\psi^\st\chi),
\end{aligned}
\end{equation}
Let $\gamma\in\Gamma$ denote a topological generator. By Mazur's control theorem\footnote{See e.g. \cite[Prop.~4.3]{arnold} 
for the case at hand.}, the natural map
\begin{equation}\label{eq:dec-lambda}
\mathcal{X}_v(\psi^\st\chi)/(\gamma-1)\mathcal{X}_v(\psi^\st\chi)\rightarrow{\rm Hom}_{\Z_p}({\rm Sel}_v(K,A_{\psi^\st\chi}),\Q_p/\Z_p)
\end{equation}
has finite kernel and cokernel, and so from Theorem~\ref{thm:katz} and Theorem~\ref{thm:AH}, we find
\begin{equation}\label{eq:psichi=0}
\begin{aligned}
L(\psi^\st\chi,r)\neq 0\quad\Longrightarrow\quad\mathscr{L}_v(\psi^\st\chi\mathbf{N}^{-r})(\mathds{1})\neq 0\quad&\Longrightarrow\quad\#\bigl(\mathcal{X}_v(\psi^\st\chi)/(\gamma-1)\mathcal{X}_v(\psi^\st\chi)\bigr)<\infty\\
\quad&\Longrightarrow\quad{\rm Sel}_v(K,V_{\psi^\st\chi})=0.
\end{aligned}
\end{equation}
In view of \eqref{eq:dec-Sel-BK}, we are reduced to showing the implication
\begin{equation}\label{eq:heeg-nonzero}
{\rm dim}_\Phi\,{\rm Sel}_\BK(K,V_{g,\chi})=1\quad\Longrightarrow\quad z_{f}\neq 0.\nonumber
\end{equation}
By Corollary~\ref{cor:HP-IMC} and Mazur's control theorem for $\mathcal{X}_\BK(g,\chi)$ (as follows from e.g. \cite[Prop.~4.3]{arnold} applied to each of the direct summands in \eqref{eq:dec-pconv}), we see that if $\check{\rm Sel}_\BK(K,V_{g,\chi})$ is $1$-dimensional then ${\rm char}_\Lambda(\mathcal{S}_\BK(g,\chi)/\Lambda\cdot\mathbf{z}_{g,\chi})$ is not divisible by $(\gamma-1)$, and therefore $\mathbf{z}_{g,\chi}$ has non-torsion image in the quotient $\check{\mathcal{S}}_\BK(g,\chi)/(\gamma-1)\check{\mathcal{S}}_\BK(g,\chi)$. Since by another application of Mazur's control theorem and \eqref{eq:psichi=0} the composite map
\[
\check{\mathcal{S}}_\BK(g,\chi)/(\gamma-1)\check{\mathcal{S}}_\BK(g,\chi)\rightarrow{\rm Sel}_{\BK}(K,T_{g,\chi})\rightarrow{\rm Sel}_{\BK}(\Q,V_f(r))
\]
has finite kernel, and by Theorem~\ref{thm:ERL} it sends 
$\mathbf{z}_{g,\chi}$ to a nonzero multiple of $z_f$, the first assertion follows. The second assertion then follows from Theorem~\ref{thm:BDP-formula}.
\end{proof}

\bibliography{Katz-converse-references}
\bibliographystyle{alpha}
\end{document}